\documentclass[11pt,reqno,tbtags,a4paper]{amsart}
\usepackage{amsxtra}
\usepackage{amssymb}
\usepackage{xpunctuate}
\usepackage{url}
\usepackage[square,numbers]{natbib}
\bibpunct[, ]{[}{]}{;}{n}{,}{,}

\title%[On distance covariance]
{On distance covariance in metric and Hilbert spaces}

\date{29 October, 2019}
\author{Svante Janson}
\thanks{Partly supported by the Knut and Alice Wallenberg Foundation}
\address{Department of Mathematics, Uppsala University, PO Box 480,
SE-751~06 Uppsala, Sweden}
\email{svante.janson@math.uu.se}
\newcommand\urladdrx[1]{{\urladdr{\def~{{\tiny$\sim$}}#1}}}
\urladdrx{http://www2.math.uu.se/~svante/}
\numberwithin{equation}{section}

\renewcommand\le{\leqslant}
\renewcommand\ge{\geqslant}

\allowdisplaybreaks

\theoremstyle{plain}% default
\newtheorem{theorem}{Theorem}[section]
\newtheorem{lemma}[theorem]{Lemma}

\newtheorem{corollary}[theorem]{Corollary}

\theoremstyle{definition}

\newtheorem{exampleqqq}[theorem]{Example}
\newenvironment{example}{\begin{exampleqqq}}
  {\hfill\qedsymbol\end{exampleqqq}}
\newtheorem{remarkqqq}[theorem]{Remark}
\newenvironment{remark}{\begin{remarkqqq}}
  {\hfill\qedsymbol\end{remarkqqq}}
\newtheorem{definition}[theorem]{Definition}
\newtheorem{problem}[theorem]{Problem}

\theoremstyle{remark}

\newenvironment{romenumerate}[1][-10pt]{% optional argument changes indentation
\addtolength{\leftmargini}{#1}\begin{enumerate}% gives (i), (ii) etc.
 \renewcommand{\labelenumi}{\textup{(\roman{enumi})}}%
 }{\end{enumerate}}

\newcounter{oldenumi}
{\setcounter{oldenumi}{\value{enumi}}
\begin{romenumerate} \setcounter{enumi}{\value{oldenumi}}}
{\end{romenumerate}}

\newenvironment{dcenumerate}{%  argument yields prefix
\begin{enumerate}% gives (#1 1), (#1 2) etc.
 \renewcommand{\labelenumi}{\textup{(dc\arabic{enumi})}}%
 }{\end{enumerate}}

\newenvironment{dcenumerateq}% continues numbering from previous romenumerate
{\setcounter{oldenumi}{\value{enumi}}
\begin{dcenumerate} \setcounter{enumi}{\value{oldenumi}}}
{\end{dcenumerate}}

\newcounter{thmenumerate}
\newenvironment{thmenumerate}
{\setcounter{thmenumerate}{0}%
 \def\item{\par% \ifnum\thethmenumerate=0\else\par\fi %\noindent\fi
 \refstepcounter{thmenumerate}\textup{(\roman{thmenumerate})\enspace}}
}
{}

\newcounter{xenumerate}   %no left indentation; thus wider lines

\newcommand\pfitemx[1]{\par#1:}
\newcommand\pfitemref[1]{\pfitemx{\ref{#1}}}

\newcommand\pfcase[2]{\smallskip\noindent\emph{Case #1: #2.}\noindent}

\newcommand{\refT}[1]{Theorem~\ref{#1}}
\newcommand{\refTs}[1]{Theorems~\ref{#1}}
\newcommand{\refC}[1]{Corollary~\ref{#1}}

\newcommand{\refL}[1]{Lemma~\ref{#1}}

\newcommand{\refR}[1]{Remark~\ref{#1}}
\newcommand{\refRs}[1]{Remarks~\ref{#1}}
\newcommand{\refS}[1]{Section~\ref{#1}}
\newcommand{\refSs}[1]{Sections~\ref{#1}}
\newcommand{\refSS}[1]{Section~\ref{#1}}
\newcommand{\refP}[1]{Problem~\ref{#1}}

\newcommand{\refD}[1]{Definition~\ref{#1}}
\newcommand{\refDs}[1]{Definitions~\ref{#1}}
\newcommand{\refE}[1]{Example~\ref{#1}}
\newcommand{\refEs}[1]{Examples~\ref{#1}}

\newcommand{\refApp}[1]{Appendix~\ref{#1}}
\newcommand{\refTab}[1]{Table~\ref{#1}}

\begingroup
  \divide\count255 by 60
  \multiply\count255 by -60
  \advance\count255 by \time
  \ifnum \count255 < 10 \xdef\klockan{\the\count1.0\the\count255}
  \else\xdef\klockan{\the\count1.\the\count255}\fi
\endgroup

\newcommand{\sumn}{\sum_{n=1}^\infty}
\newcommand{\sumin}{\sum_{i=1}^n}

\newcommand\set[1]{\ensuremath{\{#1\}}}
\newcommand\bigset[1]{\ensuremath{\bigl\{#1\bigr\}}}

\newcommand\xpar[1]{(#1)}
\newcommand\bigpar[1]{\bigl(#1\bigr)}
\newcommand\Bigpar[1]{\Bigl(#1\Bigr)}

\newcommand\sqpar[1]{[#1]}
\newcommand\bigsqpar[1]{\bigl[#1\bigr]}
\newcommand\Bigsqpar[1]{\Bigl[#1\Bigr]}

\newcommand\xcpar[1]{\{#1\}}
\newcommand\bigcpar[1]{\bigl\{#1\bigr\}}

\newcommand\abs[1]{\lvert#1\rvert}
\newcommand\bigabs[1]{\bigl\lvert#1\bigr\rvert}
\newcommand\Bigabs[1]{\Bigl\lvert#1\Bigr\rvert}

\def\rompar(#1){\textup(#1\textup)}    % usage: \rompar(...)

\newcommand\innprod[1]{\langle#1\rangle}
\newcommand\biginnprod[1]{\bigl\langle#1\bigr\rangle}

\def\xexp(#1){e^{#1}}

\newcommand\ntoo{\ensuremath{{n\to\infty}}}
\newcommand\Ntoo{\ensuremath{{N\to\infty}}}
\newcommand\asntoo{\text{as }\ntoo}

\newcommand\Mtoo{\ensuremath{{M\to\infty}}}

\newcommand\xtoo{\ensuremath{{x\to\infty}}}
\newcommand\bmin{\wedge}
\newcommand\norm[1]{\lVert#1\rVert}
\newcommand\bignorm[1]{\bigl\lVert#1\bigr\rVert}

\newcommand\upto{\nearrow}
\newcommand\punkt{\xperiod}    % xpunctuate
\newcommand\iid{i.i.d\punkt}    
\newcommand\ie{i.e\punkt}
\newcommand\eg{e.g\punkt}
\newcommand\viz{viz\punkt}
\newcommand\cf{cf\punkt}
\newcommand{\as}{a.s\punkt}
\newcommand{\aex}{a.e\punkt}
\newcommand\ii{\mathrm{i}}

\newcommand{\tend}{\longrightarrow}
\newcommand\dto{\overset{\mathrm{d}}{\tend}}
\newcommand\pto{\overset{\mathrm{p}}{\tend}}
\newcommand\asto{\overset{\mathrm{a.s.}}{\tend}}
\newcommand\lxto[1]{\overset{L^{#1}}{\tend}}
\newcommand\lto{\lxto1}
\newcommand\llto{\lxto2}
\newcommand\lpto{\lxto{p}}
\newcommand\eqd{\overset{\mathrm{d}}{=}}

\newcommand\bbR{\mathbb R}

\newcounter{CC}
\newcounter{cc}
\newcommand\E{\operatorname{\mathbb E{}}}
\renewcommand\P{\operatorname{\mathbb P{}}}

\newcommand\Var{\operatorname{Var}}
\newcommand\Cov{\operatorname{Cov}}

\newcommand\Be{\operatorname{Be}}

\newcommand\sgn{\operatorname{sgn}}

\newcommand\ga{\alpha}
\newcommand\gb{\alpha}
\newcommand\gd{\delta}

\newcommand\gf{\varphi}
\newcommand\gam{\gamma}
\newcommand\gG{\Gamma}

\newcommand\gl{\lambda}
\newcommand\gL{\Lambda}

\newcommand\gO{\Omega}
\newcommand\gs{\sigma}

\newcommand\gth{\theta}
\newcommand\eps{\varepsilon}

\newcommand\cF{\mathcal F}

\newcommand\cH{\mathcal H}
\newcommand\cI{\mathcal I}

\newcommand\cL{{\mathcal L}}

\newcommand\cP{\mathcal P}

\newcommand\cX{{\mathcal X}}
\newcommand\cY{{\mathcal Y}}

\newcommand\tX{{\widetilde X}}
\newcommand\tY{{\widetilde Y}}

\newcommand\indic[1]{\boldsymbol1\xcpar{#1}} 
\newcommand\bigindic[1]{\boldsymbol1\bigcpar{#1}}

\newcommand\qw{^{-1}}

\newcommand\qq{^{1/2}}

\newcommand\intoi{\int_0^1}
\newcommand\intoo{\int_0^\infty}

\newcommand\oi{\ensuremath{[0,1]}}

\newcommand\ooo{[0,\infty)}
\newcommand\ooox{[0,\infty]}

\newcommand\dd{\,\mathrm{d}}
\newcommand\ddx{\mathrm{d}}

\newcommand{\chf}{characteristic function}
\newcommand{\gsf}{$\gs$-field}
\newcommand{\ui}{uniformly integrable}
\newcommand\rv{random variable}
\newcommand\lhs{left-hand side}
\newcommand\rhs{right-hand side}

\newcommand\bX{\mathbf X}
\newcommand\bY{\mathbf Y}
\newcommand\bZ{\mathbf Z}
\newcommand\be{\mathbf e}
\newcommand\bex{\mathbf e'}

\newcommand\bt{\mathbf t}
\newcommand\bu{\mathbf u}
\newcommand\bx{\mathbf x}
\newcommand\by{\mathbf y}
\newcommand\bz{\mathbf z}
\newcommand\bxi{\boldsymbol \xi}
\newcommand\bbeta{\boldsymbol \eta}
\newcommand\bzeta{\boldsymbol\zeta}
\newcommand\hX{\widehat X}
\newcommand\hY{\widehat Y}
\newcommand\hXb{{\widehat X}_\gb}
\newcommand\hYb{{\widehat Y}_\gb}
\newcommand\hXbM{{\widehat X}_{\gb;M}}
\newcommand\hYbM{{\widehat Y}_{\gb;M}}
\newcommand\tXb{{\widetilde X}_\gb}
\newcommand\tYb{{\widetilde Y}_\gb}
\newcommand\DC{\operatorname{dcov}}
\newcommand\DCx[1]{\operatorname{dcov}_{#1}}
\newcommand\hDCx[1]{\DCx{#1}\sphat}
\newcommand\tDCx[1]{\DCx{#1}\sptilde}

\newcommand\DCb{\DCx{\gb}}
\newcommand\DCii{\DCx{2}}
\newcommand\hDCii{\hDCx{2}}
\newcommand\tDCii{\tDCx{2}}
\newcommand\hDCb{\DC_{\gb}\sphat}
\newcommand\tDCb{\DC_{\gb}\sptilde}
\newcommand\xDCb{\DC_{\gb}^*}
\newcommand\EDCb{\DC_{\gb}^{\mathsf E}}
\newcommand\HDCb{\DC_{\gb}^{\mathsf H}}
\newcommand\ZDCb{\DC_{\gb}^=}
\newcommand\ox{\bx_o}
\newcommand\oy{\by_o}
\newcommand\HX{\cH}
\newcommand\HY{\cH'}
\newcommand\dimH{\dim\cH}
\newcommand\PhiX{\Phi_{\bX}}
\newcommand\PhiY{\Phi_{\bY}}
\newcommand\phiX{\varphi_{\bX}}

\newcommand\textas{\quad\text{a.s.}}
\newcommand\cb{c_{\gb}}
\newcommand\cbx[1]{c_{\gb,#1}}
\newcommand\bigmid{\bigm|}
\newcommand\rrq{\frac{r^2}{2}}
\newcommand\ssq{\frac{s^2}{2}}
\newcommand\gbx{{\gb^{*}}}
\newcommand\sumiv{\sum_{i=1}^4}
\newcommand\nn{^{(n)}}
\newcommand\normll[1]{\norm{#1}_{L^2}}
\newcommand\cXY{\cX\times\cY}
\newcommand\Xx{X^*}
\newcommand\Xxx{\widehat{X}^{*}}
\newcommand\Yxx{\widehat Y^{*}}
\newcommand\hXii{\hX_2}
\newcommand\tXii{\tX_2}
\newcommand\hYii{\hY_2}
\newcommand\tYii{\tY_2}
\newcommand\psib{\hXb}
\newcommand\psibY{\hYb}
\newcommand\ellk{k}
\newcommand\GLX{\gL_X}
\newcommand\GLY{\gL_Y}
\newcommand\tensor{\otimes}
\newcommand\hphi{\hat\phi}
\newcommand\IM{I^{M}}
\newcommand\MOD[1]{modulo~#1}
\newcommand\XN{X_{\le N}}
\newcommand\YN{Y_{\le N}}
\newcommand\gaxx{\iota}
\newcommand\cAxx{\cI}
\newcommand\Lgbgb{L^{\gb,2\gb}}
\newcommand\Xxxl{\widehat{X}^{**}}
\newcommand\Yxxl{\widehat{Y}^{**}}
\newcommand\minx[2]{#1\bmin#2}
\newcommand\HS{Hilbert--Schmidt}
\newcommand\normHS[1]{\norm{#1}_{\mathrm{HS}}}
\newcommand\Tz{T_\bz}
\newcommand\ja{+}
\newcommand\nej{$-$}
\newcommand\CS{Cauchy--Schwarz}
\newcommand\CSineq{\CS{} inequality}

\begin{document}

\begin{abstract} 
Distance covariance 
is a measure of
dependence between two random variables that take values in
two, in general different, metric spaces, see
Sz\'ekely, Rizzo and Bakirov (2007) and Lyons (2013).
It is known that the distance covariance, and its generalization
$\gb$-distance covariance,
can be defined in several different ways that are equivalent under 
some moment conditions.
The present paper considers four such definitions and find minimal moment
conditions for each of them, together with some partial results when these
conditions are not satisfied.

The paper also studies the special 
case when the variables are Hilbert space valued,
and shows under weak moment conditions that two such variables are
independent if and only if their ($\gb$-)distance covariance is 0; this extends
results by Lyons (2013) and 
Dehling et al\punkt (2018+).
The proof uses a new definition of distance covariance in the Hilbert space
case, generalizing the definition 
for Euclidean spaces
using characteristic functions
by Sz\'ekely, Rizzo and Bakirov (2007).
\end{abstract}

\maketitle

\section{Introduction}\label{S:intro}

\emph{Distance covariance} %and \emph{distance correlation} 
is a measure of
dependence between two random variables $\bX$ and $\bY$ that take values in
two, in general different, spaces $\cX$ and $\cY$.
This measure appears in \citet{Feuerverger} 
as a test statistic when $\cX=\cY=\bbR$;
it was  more generally
introduced by \citet{SRB2007} for the case 
of random variables in
%when $\cX$ and $\cY$ are 
Euclidean spaces, possibly of different dimensions.
This was extended to general separable measure spaces  by
\citet{Lyons}, see also \citet{Jakobsen},
and to semimetric spaces (of negative type, see below)
by \citet{Sejdinovic2013}.

Our setting throughout this paper is the following 
(see also \refR{Rsemimetric}):
 $(\bX,\bY)$
is a pair of random variables taking values in $\cX\times\cY$, 
where $\cX$ and $\cY$ are separable metric spaces, with metrics
$d_\cX$ and $d_\cY$; we write just $d$ for both metrics when there is no
risk of confusion.
%We will also use a parameter $\gb>0$; as will be seen below, this means
%replacing $d$ by $d^\gb$.

We denote the distance covariance by
$\DCb(\bX,\bY)$, where $\gb>0$ is a parameter. 
The standard choice is $\gb=1$; in this case we may drop
the subscript and write $\DC(\bX,\bY)$.

One interesting feature of distance correlation is that it can be defined in
several ways that look very different but are equivalent (at least assuming
sufficient moment conditions). 
We will give several definitions (sometimes for special cases)
and begin with three related 
definitions that work in the general setting just described.

Let, throughout the paper, 
$(\bX_1,\bY_1), (\bX_2,\bY_2), \dots$ be independent copies of $(\bX,\bY)$.
Also, let
$\ox\in\cX$  and $\oy\in\cY$ be two fixed points, and write for convenience
$\norm{\bx}:=d(\bx,\ox)$ and $\norm{\by}:=d(\by,\oy)$
for $\bx\in\cX$ and $\by\in\cY$.
(In the case of Euclidean spaces, or Hilbert spaces, we choose $\ox=\oy=0$,
and
$\norm{\bx}$ is the usual norm.)
We use $\ox$ and $\oy$ for
moment conditions of the type $\E \norm{\bX}^\gb<\infty$;
note that 
by the triangle inequality,
for this condition the choice of $\ox$ does not matter, and that 
this
condition is equivalent to $\E d(\bX_1,\bX_2)^\gb<\infty$.

Also, define for convenience
\begin{align}\label{gbx}
  \gbx:=
\max(\gb,2\gb-2)=
  \begin{cases}
    \gb, & 0<\gb\le 2,\\
    2\gb-2, & \gb>2.
  \end{cases}
\end{align}
As will be seen below, the case of main interest is $\gb\in(0,2]$;  
in this case thus simply $\gbx=\gb$.

When necessary, we distinguish the versions of distance covariance by
different superscripts such as
$\xDCb,\hDCb,\tDCb$, but usually this is omitted because
the choice of definition does not matter, or is clear from the context.

\begin{definition}\label{D1}
Assume $\E \norm{\bX}^{2\gb}<\infty$ and $\E \norm{\bY}^{2\gb}<\infty$. Then
  \begin{align}
\DCb(\bX,\bY)&=
  \xDCb(\bX,\bY)
\notag\\&\hskip-0em
:=\E\bigsqpar{d(\bX_1,\bX_2)^\gb d(\bY_1,\bY_2)^\gb}
+ \E\bigsqpar{d(\bX_1,\bX_2)^\gb}\E\bigsqpar{d(\bY_1,\bY_2)^\gb}
\notag\\&\hskip2em
-2\E\bigsqpar{d(\bX_1,\bX_2)^\gb d(\bY_1,\bY_3)^\gb}.
\label{d1}  \end{align}
\end{definition}

\begin{definition}\label{D2}
Assume $\E \norm{\bX}^{\gbx}<\infty$ and $\E \norm{\bY}^{\gbx}<\infty$. Then
\begin{align}\label{d2}
\DCb(\bX,\bY)=
  \hDCb(\bX,\bY):=\tfrac14\E\bigsqpar{\hXb\hYb}, 
\end{align}
where
\begin{align}
 \hXb&:=d(\bX_1,\bX_2)^\gb-d(\bX_2,\bX_3)^\gb+d(\bX_3,\bX_4)^\gb-d(\bX_4,\bX_1)^\gb
\label{hxb}
%\\
% \hYb&:=
% d(\bY_1,\bY_2)^\gb-d(\bY_2,\bY_3)^\gb+d(\bY_3,\bY_4)^\gb-d(\bY_4,\bY_1)^\gb. 
\end{align}  
and similarly for $\hYb$.
\end{definition}

\begin{definition}\label{D3}
Assume $\E \norm{\bX}^{\gbx}<\infty$ and $\E \norm{\bY}^{\gbx}<\infty$. Then
\begin{align}\label{d3}
\DCb(\bX,\bY)=  
\tDCb(\bX,\bY):=\E\bigsqpar{\tXb\tYb}, 
\end{align}
where
\begin{align}
\tXb&:=\E\xpar{\hXb\mid \bX_1,\bX_2}
\label{txb1}\\&\phantom:
=d(\bX_1,\bX_2)^\gb-\E_{\bX} d(\bX_1,\bX)^\gb- \E_{\bX} d(\bX_2,\bX)^\gb
 + \E d(\bX_1,\bX_2)^\gb
\label{txb2}
%\\
%\tYb&:=\E\xpar{\hY\mid Y_1,Y_2}
%\notag\\&
%=
%d(Y_1,Y_2)^\gb-\E_{Y_3} d(Y_2,Y_3)^\gb- \E_{Y_3} d(Y_1,Y_3)^\gb+
%\E d(Y_1,Y_2)^\gb ,
\end{align}  
and similarly for $\tYb$, where $\E_{\bX}$ denotes 
integrating over $\bX$ only, \ie,
the conditional expectation given all $\bX_j$ (but not $\bX$).
\end{definition}

The role of the parameter $\gb$ is thus to replace the metric $d$ by
$d^\gb$ in the definition of $\DC=\DCx1$. See further \refR{Rsemimetric} below.

Note that $\DCb(\bX,\bY)$  only depends on the joint distribution of
$\bX$ and $\bY$; thus distance covariance can be seen as a
functional on distributions in $\cX\times\cY$.

 The moment condition
$\E \norm{\bX}^{2\gb}<\infty$ and $\E \norm{\bY}^{2\gb}<\infty$
in \refD{D1} is equivalent to
$\E d(\bX_1,\bX_2)^{2\gb}<\infty$ and $\E d(\bY_1,\bY_2)^{2\gb}<\infty$,
which implies that all expectations in \eqref{d1} are finite; it implies
also
$\hXb,\hYb\in L^2$ and thus $\tXb,\tYb\in L^2$, so the expectations
in \eqref{d2} and \eqref{d3} are also finite.
Moreover, in this case,
it is easy to see that  \refDs{D1}--\ref{D3} are equivalent:
by expanding the products $\hXb\hYb$ and $\tXb\tYb$ in \eqref{d2} and
\eqref{d3},
we obtain \eqref{d1} after simple calculations.
It is less obvious that the weaker moment condition in \refDs{D2} and
\ref{D3}
is enough to guarantee that the expectations in \eqref{d2} and \eqref{d3}
are finite and equal; we show this, and in particular that
$\hXb,\hYb,\tXb,\tYb\in L^2$, 
in \refS{Sexist} (\refT{T1}).
In \refS{Sopt} we show that the exponents $2\gb$ and $\gbx$ in the moment
conditions are  optimal in general;
in \refS{Sbeyond} we discuss extensions when the moment conditions fail.

The original definition of distance covariance by \citet{SRB2007}, 
for random
variables $\bX$ and $\bY$ in Euclidean spaces $\bbR^p$ and $\bbR^q$,
see also \citet{Feuerverger},
is quite different and is based on characteristic functions.
The general version with a $\gb\in(0,2)$
\cite[Section 3.1]{SRB2007} is as follows.

Let $\gf_{\bX}(\bt):=\E e^{\ii\bt\cdot\bX}$,
$\gf_{\bY}(\bu):=\E e^{\ii\bu\cdot\bY}$ and
$\gf_{\bX,\bY}(\bt,\bu):=\E e^{\ii(\bt\cdot\bX+\bu\cdot\bY)}$
be the \chf{s} of $\bX$, $\bY$ and $(\bX,\bY)$.
Define also the constants
\begin{align}\label{cbx}
  \cbx{\ellk}:=\frac{2^\gb\gG((\ellk+\gb)/2)}{-\pi^{\ellk/2}\gG(-\gb/2)}
=\frac{\gb 2^{\gb-1}\gG((\ellk+\gb)/2)}{\pi^{\ellk/2}\gG(1-\gb/2)}>0.
\end{align}
(The values of these normalization
constants are unimportant; they are chosen to make the
definition agree with the preceding ones.)

\begin{definition}\label{D4}
Let\/ $(\bX,\bY)$ be a pair of random vectors in $\bbR^p$ and $\bbR^q$,
respectively, where $p,q\ge1$,
and let $0<\gb<2$.
Then
  \begin{align}\label{d4}
&\DCb(\bX,\bY)=
  \EDCb(\bX,\bY):=
\notag\\&\hskip2em
\cbx{p}\cbx{q}\int_{\bt\in\bbR^p}\int_{\bu\in\bbR^q}
\bigabs{\gf_{\bX,\bY}(\bt,\bu)-\gf_{\bX}(\bt)\gf_{\bY}(\bu)}^2
\frac{\dd\bt\dd\bu}
{|\bt|^{p+\gb} |\bu|^{q+\gb}}.
\end{align}
\end{definition}

\begin{remark}\label{RD4mom}
No moment condition is needed in \refD{D4}, since the
integrand in \eqref{d4} is non-negative; 
with this definition (for Euclidean spaces and $\gb<2$),
$\DCb(\bX,\bY)$ is always defined,
although it may be $\infty$. 
As shown in \cite{SRB2007}, 
$\DCb(\bX,\bY)$ is finite at
least when 
$\E \norm{\bX}^{\gb}<\infty$ and $\E \norm{\bY}^{\gb}<\infty$;
this also follows from the equivalence with \refDs{D2} and \ref{D3},
see \refTs{TEH} and \ref{T52}.

In contrast, we have in \refDs{D1}--\ref{D3} imposed moment conditions
making $\DCb(\bX,\bY)$ finite. 
These definitions can be used somewhat more generally when the expectations
in them are 
finite, and even when the result is $+\infty$; 
see \refSs{Sopt} and \ref{Sbeyond}.
%for example, in this sense,
%$\DCb(\bX,\bX)$ is always defined by \refD{D2} as a number in 
%$[0,\infty]$.
However, without moment conditions,
there are cases, even with $\bX=\bY=\bbR$, when 
\refDs{D1}--\ref{D3} yield results of the type $\infty-\infty$ and thus 
cannot be used at all;
see \refEs{ED1}, \ref{ED2bad}, \ref{ED3bad} and \ref{ED3bad3}.
\end{remark}

\begin{remark}\label{R2div}
  \refD{D4} requires $\gb<2$, since typically the integral in \eqref{d4}
  diverges for $\ga\ge2$. For example, if $p=q$ and $\bX=\bY\sim N(0,I)$,
then 
$|\gf_{\bX,\bY}(\bt,\bu)-\gf_\bX(\bt)\gf_\bY(\bu)|\sim |\innprod{\bt,\bu}|$
as 
$\bt,\bu\to0$, and
\eqref{d4} diverges for $\ga\ge2$.
\end{remark}

\citet{Feuerverger} gave \refD{D4} with $\gb=1$
for $\cX=\cY=\bbR$ and the special case when $(\bX,\bY)$ have the empirical
distribution of a finite sample from an unknown bivariate distribution, 
thus defining a test statistic for independence. 
He also showed that it has  the equivalent forms \eqref{d2} and \eqref{d1}.
More generally, 
for arbitrary random $(\bX,\bY)$ in Euclidean spaces and $0<\gb<2$, 
\citet{SRB2007} gave \refD{D4};
they also showed that it is equivalent to
\refD{D1}
when the moment condition in the latter holds
\cite[Remark 3 for $\gb=1$; implicit in \S3.1 for $\gb\in(0,2)$]{SRB2007};
see also 
\citet[(3.7), (4.1) and Theorem 8]{SR2009}.
Furthermore, \eqref{d3} was used for finite samples in
\citet[(2.8) and \S3.1]{SRB2007} and
\citet[(2.8) and \S4.1]{SR2009}.
% and \cite[\S 3]{SR2009-rejoinder}.
The name \emph{distance covariance} 
was introduced by \cite{SRB2007}
(for the case $\gb=1$, and \emph{$\gb$-distance covariance} in general).
(Actually, \cite{SRB2007} and \cite{SR2009} define 
the distance covariance as the square root
of $\DC(\bX,\bY)$; we ignore this difference in terminology.)

In the Euclidean setting in \cite{Feuerverger} and \cite{SRB2007},
with $\gb<2$,
\refD{D4} implies immediately
the fundamental property that
$\DCb(\bX,\bY)\ge0$ for any $\bX$ and $\bY$, 
%such that $\DCb(\bX,\bY)$ is defined, 
and  furthermore 
\begin{equation}\label{iff}
 \DCb(\bX,\bY)=0\iff \text{$\bX$ and $\bY$ are independent}.
\end{equation}
Hence,
$\DCb(\bX,\bY)$ can be regarded as a measure of dependency,
and distance covariance can be used to  test independence.
(As noted in \cite{SRB2007}, \eqref{iff} does not hold for $\gb=2$; see
\refS{S2}.)

\citet{Lyons} extended the theory to general (separable) metric spaces,
 with $\gb=1$, using \refD{D3} as his definition. 
(This was also suggested in \cite[\S 3]{SR2009-rejoinder}.)
\citet{Lyons} showed also that, although the definition works for arbitrary
metrics,
$\DC$ is useful as a measure of dependence
mainly in the case when $\cX$ and $\cY$ are metric spaces 
of \emph{negative type} 
(see \cite{Lyons} for a definition; 
see also \cite{Sejdinovic2013}, \cite{Berg} and \refR{Rsemimetric} below), 
because in this case, but not otherwise, $\DC(\bX,\bY)\ge0$ for any $\bX$
and $\bY$ such that $\DC(\bX,\bY)$ is defined; if furthermore the spaces are
of \emph{strong negative type} (see again \cite{Lyons}), then also
\eqref{iff} holds for $\gb=1$.
%$\DC(\bX,\bY)=0$ if and only if $\bX$ and $\bY$ are independent.
(The implication that 
$\DCb(\bX,\bY)=0$ for independent 
variables is trivial, for any $\gb$,
but not the converse.)
Hence, for metric spaces of strong negative type,
$\DC$ can be regarded as a measure of dependence and for tests of
independence just as in the Euclidean case.

We have here, as \cite{Lyons}, 
assumed that $d_\cX$ and $d_\cY$ are metrics.
However, we can formally use 
\refDs{D1}--\ref{D3}  for any symmetric
measurable
functions $d_\cX:\cX\times\cX\to\ooo$ and $d_\cY:\cY\times\cY\to\ooo$.
(For $\cX$ and $\cY$ such that the expectations exist,
and
still assuming $\cX$ and $\cY$ to be separable metric spaces, to avoid
technical problems.)
It seems natural to assume at least that
$d_\cX$ and $d_\cY$ are
\emph{semimetrics}; a semimetric on a space $\cX$ is a symmetric 
function $d:\cX\times\cX\to\ooo$ such that $d(\bx_1,\bx_2)=0\iff \bx_1=\bx_2$.
(Thus, the triangle inequality is not assumed. Note that the term semimetric
also is used in other context with  a different meaning.)
This extension was made by \citet{Sejdinovic2013}; they considered
semimetrics of negative type and showed that much of the theory extends to
this case.

\begin{remark}\label{Rsemimetric}
If $0<\gb\le1$, then $d^\gb$ is also a metric for any metric $d$, and
$\DCb$ is just $\DC$ 
applied to the spaces $\cX$ and $\cY$ equipped with the
metrics $d_\cX^\gb$ and $d_\cY^\gb$.
(From an abstract point of view, the case $\gb\le1$ thus does not add
anything new.)

If we allow general semimetrics, there is no such restriction; $d^\gb$ is a
semimetric for every $\gb>0$, and $\DCb$ is just $\DC$ applied to
the semimetrics $d_\cX^\gb$ and $d_\cY^\gb$ for any $\gb>0$.

On the other hand, 
see \cite{Lyons} and \cite{Schoenberg1938},
a semimetric $d$ on a space $\cX$ is of negative type if and only there exists
an embedding $\gf:\cX\to\cH$ into a Hilbert space such that
\begin{align}\label{-}
  d(\bx_1,\bx_2)=\norm{\gf(\bx_1)-\gf(\bx_2)}^2.
\end{align}
In particular, \eqref{-} implies that $d\qq$ is a metric.
(We assume that balls for
the semimetric define the topology, and thus
the metric $d\qq$
defines the topology %and the Borel $\gs$-field 
of $\cX$.)
Hence, for semimetrics of negative type, $\DCb$ is the same as
$\DCx{2\gb}$ for the metrics $d_\cX\qq$ and $d_\cY\qq$;
in particular, $\DC$ equals $\DCx2$ for these metrics.
Consequently, our setting with metrics but arbitrary $\gb$ 
includes also semimetrics of negative type.
Furthermore, using the embedding $\gf$, we see that $\DC$ for
semimetric spaces of negative type can be reduced to $\DCx2$ for
Hilbert spaces, see \refR{Rembedded}.
(This is implicit in \cite{Sejdinovic2013}, where this 
embedding is used to give another interpretation of distance covariance,
see \refRs{RHS} and \ref{RHS2}.)

We will in the sequel assume that $d_\cX$ and $d_\cY$ are metrics
(without assuming negative type), 
but note that as just said, by changing $\gb$, this really includes the case
of semimetrics of negative type.
%Conversely, for the case of Hilbert spaces, if $0<\gb\le2$, we can regard

In this context we note that
if $\cX$ is a Euclidean space $\bbR^q$, 
or more generally a Hilbert space, then the semimetric
$\norm{\bx_1-\bx_2}^\gb$ is of negative type if and only if $0<\gb\le2$, see
\cite{Schoenberg1938}.
(It is thus a metric of negative type if and only if $0<\gb\le1$.)
Consequently, for Hilbert spaces, if $0<\gb\le2$, 
we can conversely regard $\DCb$ as 
$\DC$ for the semimetric of negative type $\norm{\bx_1-\bx_2}^\gb$.
\end{remark}

In the first part of the present paper, we 
consider general metric spaces and general $\gb>0$, 
and study and compare \refDs{D1}--\ref{D3}.
In particular, we 
show that  the definitions agree under the moment conditions above
(\refS{Sexist}). % and \ref{Sopt}).
We also show that $\DCb$ depends continuously on the distribution of
$(\bX,\bY)$, assuming convergence of the $\gbx$ moments 
$\E\norm{\bX}^\gbx$ and $\E\norm{\bY}^\gbx$
(\refT{T4} and \refR{RT4}).

\citet{SRB2007} showed that, in the Euclidean case and with
$\gb\in(0,2)$, computing $\DCb$ for the empirical distribution of a sample
gives a strongly consistent estimator of $\DCb$, provided $\gb$ moments are
finite. 
This was extended to general metric spaces, with $\gb=1$, by \citet{Lyons},  
who claimed consistency in this sense assuming only finite first moments;
however, the proof is incorrect as noted in the Errata.
As also noted in \cite{Lyons}, there is a simple
proof assuming second moments, and \citet{Jakobsen} proved the result when
$\E(\norm{\bX}\norm{\bY})^{5/6}<\infty$, and thus in particular when $\bX$
  and $\bY$ have moments of order $5/3$.
We remove this condition and show (\refT{Tcons}) consistency assuming only
    first moments (as stated in \cite{Lyons}); furthermore, this is extended
    to all $\gb>0$, now assuming $\gbx$ moments.

In the second part of the paper, we consider Hilbert spaces.
\citet{Mikosch} studied $\DCb$ 
for $\gb\in(0,2)$
in
the infinite-dimensional Hilbert space $L^2\oi$, using \refD{D1}.
(Since all separable infinite-dimensional Hilbert spaces are isomorphic;
this is equivalent to considering %letting $\cX$ and $\cY$ be 
arbitrary separable Hilbert spaces.) 
\citet{Lyons} showed that a Hilbert space is of strong negative type, and
thus \eqref{iff} holds for $\gb=1$.
\citet[Theorem 4.2]{Mikosch} extended this to all $\gb\in(0,2)$.

We consider the Hilbert space case in \refSs{SH1}--\ref{S2}.
We give yet another definition of $\DCb$ in this case
(\refD{D5}), which is related to
\refD{D4} in Euclidean spaces, but where we 
replace the \chf{s} by certain \emph{characteristic random variables}, which
are Gaussian random variables that can be defined also for variables in
infinite-dimensional Hilbert spaces.
We show that this definition is equivalent to the ones above
under suitable moment conditions.
We then use
this definition to give a new proof, assuming only $\gb$ moments, 
of the theorem by
\citet{Mikosch} just mentioned that
\eqref{iff} holds
 for Hilbert spaces and any $\gb\in(0,2)$
(our \refT{TH}).
Our proof (and \refD{D5}) 
is based on the ideas in \cite{Mikosch}; however,
the proof %of \refT{TH} 
in \cite{Mikosch} is formulated
for the Hilbert space $L^2\oi$ and uses arguments with
Brownian motion.
Our proof can be regarded as a more abstract version of their
proof,
stated for arbitrary (separable) Hilbert spaces
and using \iid{} Gaussian sequences instead of Brownian motion;
we believe that this makes the proof clearer since it avoids irrelevant
details related to the particular choice $L^2\oi$ of the Hilbert space.

\refS{S2} studies the case $\gb=2$
for Hilbert spaces. This case  is rather trivial, and markedly
different from $\gb<2$. In particular, 
even in one dimension,
\eqref{iff} does \emph{not} hold for
$\gb=2$, as is well known 
since
\cite[\S 3.1]{SRB2007}  and   \cite[\S 4.1]{SR2009}. 
However, this case is of special interest because,
as said in \refR{Rsemimetric}, and in more detail in \refR{Rembedded},
arbitrary (semi)metric spaces of negative type can be embedded
into it.

In the third part of the paper,
we return to general metric spaces and study whether the
moment conditions above are optimal. 
\refS{Sopt}  shows that the exponents in the conditions cannot be
decreased, in general.
However, some other weakenings are possible, and
in \refS{Sbeyond} 
we further study and compare the various definitions when the moment 
conditions above fail. We give some results;
in particular, we consider Lorentz spaces.
We also state
some open problems
that we have failed to solve.

The appendices contain some general results on uniform integrability and on
integrals in a Hilbert space used in the paper; for completeness full proofs
are given although some or all results are known.

\begin{remark}\label{RD6}
  Another version of the definitions above is obtained if we denote the
\rhs{} of \eqref{hxb} by $\psib(\bX_1,\bX_2,\bX_3,\bX_4)$ and then define
\begin{align}\label{d6}
    \DCb(\bX,\bY)&=
    \ZDCb(\bX,\bY)
\notag\\&:=
\E\bigpar{\psib(\bX_1,\bX_2,\bX_3,\bX_4)\psibY(\bY_1,\bY_2,\bY_5,\bY_6)}.
\end{align}
This version is used in proofs in \cite{Lyons} and \cite{Jakobsen}.

It is obvious that if $\hXb,\hYb\in L^2$, then the expectation in \eqref{d6} 
is finite, and, using Fubini's theorem to integrate first over
$\bX_3,\bX_4,\bY_5,\bY_6$, it equals $\E\bigpar{\tXb\tYb}$; thus, at least
in this case, \eqref{d6} agrees with \eqref{d3}.
In particular, by \refL{L2} below, this holds 
when  $\E \norm{\bX}^{\gbx}<\infty$ and $\E \norm{\bY}^{\gbx}<\infty$. 
We will not consider this definition further, and we leave the case when the
moment condition just stated fails to the reader.
(We conjecture results similar to those in \refSs{Sopt} and \ref{Sbeyond}.)
\end{remark}

\begin{remark}
  We have defined $\tXb$ as a conditional expectation of $\hXb$; this can be
  regarded as an orthogonal projection in the Hilbert space $L^2(\P)$.

If $\E\norm{\bX}^{2\gb}<\infty$, so $d(\bX_1,\bX_2)^\gb\in L^2$, then, as
noted by \citet{Jakobsen},
$\tXb$ can also be regarded as a projection in another way, \viz{} as the
orthogonal projection of $d(\bX_1,\bX_2)^\gb$ onto the subspace of $L^2(\P)$
consisting of functions $g(\bX_1,\bX_2)$ with 
$\E\bigpar{g(\bX_1,\bX_2)\mid\bX_1}=\E\bigpar{g(\bX_1,\bX_2)\mid\bX_2}=0$ a.s.
\end{remark}

\begin{remark}\label{RHS}
For semimetrics of negative type, 
another interpretation of distance covariance is given by
\citet[Theorem 24]{Sejdinovic2013}, showing that it coincides with 
the \emph{Hilbert-Schmidt independence criterion}, a distance measure
between the distributions $\cL(\bX,\bY)$ and
$\cL(\bX_1,\bY_2)=\cL(\bX)\times\cL(\bY)$ that is defined using
reproducing Hilbert spaces
given by some kernels on the spaces, provided one chooses the kernels to be
defined in a specific way by the metrics $d_\cX$ and $d_{\cY}$.
See also \refR{RHS2}.
\end{remark}

\begin{remark}
  Yet another interpretation (or definition) of distance covariance was given 
by \citet{SR2009} for Euclidean spaces; it was called 
\emph{Brownian covariance distance}. In the one-dimensional case
$\cX=\cY=\bbR$, and with $\gb=1$, let $W$ and $W'$ be two two-sided Brownian
motions, independent of each other and of $\bX$ and $\bY$; then
\begin{align}\label{brown}
  \DC(\bX,\bY)=\E \bigsqpar{\Cov\bigpar{W(\bX),W'(\bY)\bigmid W,W'}^2}
\end{align}
This was extended, also in \cite{SR2009}, to arbitrary dimension by 
using Brownian fields on $\bbR^k$,
and to $\gb\in(0,2)$ by using fractional Brownian fields.
%This is shown to be equivalent to \eqref{d3} and \eqref{d1}.
% \eqref{d3} in the proof in Appendix, but \eqref{d1} in statement.

This approach was further generalized to arbitrary spaces with
semimetrics of negative type
by \citet[Section 6.4]{Kanagawa2018+},
letting $W$ and $W'$ be Gaussian stochastic processes on $\cX$ and $\cY$,
with suitable covariance kernels.
%see the equivalence in \cite[Theorem 24]{Sejdinovic2013}.
\end{remark}

\begin{remark}\label{R>0}
  \refDs{D2}--\ref{D4} show immediately that $\DCb(\bX,\bX)\ge0$ whenever
the definition applies (even in the extended sense
discussed in \refR{RD4mom}).
Moreover, 
$\DCb(\bX,\bX)>0$ unless $\bX$ is degenerate (\ie, is concentrated at a
single value);
this is immediate for \refD{D4}; it was shown by \citet{Lyons} for
\refD{D3} (for $\gb=1$), and his proof extends to general $\gb$, 
and to \refD{D2},
for the latter even without any moment assumption (allowing $+\infty$).
\end{remark}

\begin{remark}\label{Rcorr}
  \emph{Distance correlation} is defined by \cite{SRB2007} as 
\begin{align}
  \frac{\DCb(\bX,\bY)}{\DCb(\bX,\bX)\qq\DCb(\bY,\bY)\qq},
\end{align}
provided $\bX$ and $\bY$ are non-degenerate so
the denominator is strictly positive (see \refR{R>0}).

Various properties of distance correlation follow from properties of
distance covariance; we leave this to the reader.
\end{remark}

%\begin{remark}
%  The papers on distance covariance  cited above contain also other results, 
%  in particular, results on consistency of estimators.
%Such results will not be considered here.
%\end{remark}

\section{Some notation}\label{Snot}

As said in the introduction,
 $(\bX,\bY)$
is a pair of random variables taking values in 
separable metric spaces  $\cX$ and $\cY$,
and 
$(\bX_i,\bY_i)$, $i\ge1$, are  independent copies of $(\bX,\bY)$.
$\gb$ is a fixed parameter, and $\gbx$ is given by \eqref{gbx}. 
Unless stated otherwise, we assume only $\gb>0$.
(This condition is sometimes repeated for emphasis.)

$\cP(\cX)$ denotes the set of all Borel probability measures in $\cX$.

Convergence almost surely, in probability, in distribution and in $L^p$
are denoted by
$\asto$, $\pto$, $\dto$, $\lpto$.

We use the standard definition of covariance
\begin{align}
  \label{Cov}
\Cov(Z,W):=\E\sqpar{ ZW}-\E Z\E W
\end{align}
not only for real random variables, but also more generally
for any complex random variables
$Z$ and $W$ with $\E|Z|^2,\E|W|^2<\infty$;
we further extend this notation to
conditional covariance.

For real $x,y$, 
$x\land y:=\min\set{x,y}$ and
$x\lor y:=\max\set{x,y}$; also
$x_+:=x\lor 0$ and $x_-:=(-x)_+=-(x\land0)$, so $x=x_+-x_-$.

The inner product in a Hilbert space is denoted by $\innprod{x,y}$;
for finite-dimensional $\bbR^q$ we also use
$x\cdot y$.
All Hilbert spaces have real scalars, so the inner product is real-valued.

$C$ and $c$
will denote some unimportant positive constants that depend only on $\gb$
(and may be taken as universal constants for $\gb\le2$).
Their value may differ from one occurence to the next.

\section{Existence and continuity}\label{Sexist}

We begin by recording the simple fact that with enough moments, 
\refDs{D1}--\ref{D3} agree.

\begin{lemma}\label{L0}
Let $\gb>0$.
If\/  $\E \norm{\bX}^{2\gb}<\infty$ and $\E \norm{\bY}^{2\gb}<\infty$,
then all expectations in \eqref{d1}, \eqref{d2} and \eqref{d3} are finite,
and the three definitions of $\DCb(\bX,\bY)$ agree,
\ie,
$\xDCb(\bX,\bY) = \hDCb(\bX,\bY)= \tDCb(\bX,\bY)$.
\end{lemma}
\begin{proof}
As said in the introduction, this is elementary; we omit the details.
\end{proof}

We will extend this to the weaker moment conditions used in \refDs{D2} and
\ref{D3}.
We argue similarly to \citet{Lyons}, who showed the case $\gb=1$
(and thus implicitly $0<\gb\le1$, see \refR{Rsemimetric}).
We first show some useful estimates of the variable $\hXb$ defined in 
\eqref{hxb}. Note the symmetry
up to sign under cyclic permutations of the indices $1,\dots,4$.

Although we state the next lemma for the random variables $\bX_i$,
it is really a pointwise inequality that could have been stated for four
non-random points $\bx_1,\dots,\bx_4$.
In sums such as \eqref{l1b} and \eqref{l1c}, the indices are interpreted
\MOD 4; moreover, 
a term containing an index $i\pm1$ should be interpreted as two terms, with
$i+1$ and $i-1$; the sum in \eqref{l1c} is thus really a sum of 8 terms.

\begin{lemma}\label{L1} 
Let $\cX$ be a metric space.
  \begin{thmenumerate}
  \item \label{L1a}
If $0<\gb\le1$, then
\begin{align}\label{l1a}
  |\hXb|&
%\le 8\sumiv \min\bigpar{\norm{\bX_i},\norm{\bX_{i+1}}}.
\le 2\sumiv\bigpar{\norm{\bX_i}^\gb\land\norm{\bX_{i+1}}^\gb}.
\end{align}

  \item \label{L1b}
If\/ $0< \gb\le2$, then
\begin{align}\label{l1b}
  |\hXb|
\le C\sumiv \norm{\bX_i}^{\gb/2}\norm{\bX_{i+1}}^{\gb/2}.
\end{align}
  \item \label{L1c}
If\/ $\gb\ge1$, then
\begin{align}\label{l1c}
  |\hXb|
\le C \sumiv \norm{\bX_i}^{\gb-1}\norm{\bX_{i\pm1}}.
\end{align}
  \end{thmenumerate}
\end{lemma}

\begin{proof}
Write $d_{ij}:=d(\bX_i,\bX_j)$.
Thus $\hXb=d_{12}^\gb-d_{23}^\gb+d_{34}^\gb-d_{41}^\gb$. 
Note the triangle
inequality 
\begin{align}\label{dij}
d_{ij}\le\norm{\bX_i}+\norm{\bX_j}.  
\end{align}

\pfcase{1}{$\gb\le1$}
Since $d^\gb$ is a metric when $\gb\le1$, it suffices to consider the case
$\gb=1$. 
The triangle inequality yields
\begin{align}\label{lja}
\bigabs{\hX}\le
  \bigabs{d_{12}-d_{41}} +  \bigabs{d_{23}-d_{34}} \le d_{24}+d_{24}=2d_{24}.
\end{align}
Similarly, by shifting the indices,
\begin{align}\label{ljb}
\bigabs{\hX}
%\le  \bigabs{d_{12}-d_{41}} +  \bigabs{d_{23}-d_{34}} \le d_{24}+d_{24}=
\le 2d_{13}.
\end{align}
Hence, using \eqref{lja}--\eqref{ljb} and  \eqref{dij},
\begin{align}\label{gemmb}
  |\hX|&\le 2 \min\bigpar{d_{13},d_{24}}
\le 2\min\bigpar{\norm{\bX_1}+\norm{\bX_3},\norm{\bX_2}+\norm{\bX_4}}.
\end{align}
We claim that for any real $x_1,\dots,x_4\ge0$,
\begin{align}\label{gemz}
(x_1+x_3)\land(x_2+x_4)
\le \sumiv \bigpar{x_i\bmin x_{i+1}}.
\end{align}
In fact, by cyclic symmetry, we may without loss of generality 
assume that $x_1$ is the largest of $x_1,\dots,x_4$, and in this case
\begin{align}\label{gemw}
%(x_1+x_3)\land(x_2+x_4)\le 
x_2+x_4 =x_1\land x_2 + x_4\land x_1
\le \sumiv \bigpar{x_i\bmin x_{i+1}},
\end{align}
and \eqref{gemz} follows. Hence \eqref{gemz} holds, and \eqref{gemmb}
implies 
\eqref{l1a} for $\gb=1$.
As said above, this shows \eqref{l1a} in general.

Furthermore,
for $\gb\le1$,
\eqref{l1b} follows from \eqref{l1a}
since $x\bmin y\le x\qq y\qq$ when $x,y\ge0$.

\pfcase{2}{$\gb>1$}
By the cyclic symmetry we may assume that $\norm{\bX_1}$ is the largest
of $\norm{\bX_1},\dots,\norm{\bX_4}$. Then, \eqref{dij} implies
\begin{align}\label{dij2}
  d_{ij}\le 2\norm{\bX_1}, \qquad i,j=1,\dots,4.
\end{align}
As above, the triangle inequality yields
\begin{align}\label{ljac}
  \bigabs{d_{12}-d_{41}} \le d_{24}
\end{align}
and thus, by the mean value theorem, for some $\gth\in\oi$,
\begin{align}
  \bigabs{d_{12}^\gb-d_{41}^\gb}
\le d_{24} \gb \bigpar{\gth d_{12}+(1-\gth)d_{41}}^{\gb-1}  .
\end{align}
Using \eqref{dij2}, this yields
\begin{align}\label{de124}
  \bigabs{d_{12}^\gb-d_{41}^\gb}
\le d_{24} \gb 2^{\gb-1} \norm{\bX_1}^{\gb-1}  .
\end{align}
Similarly,
\begin{align}\label{de324}
  \bigabs{d_{23}^\gb-d_{34}^\gb}
\le d_{24} \gb \bigpar{\gth' d_{23}+(1-\gth')d_{34}}^{\gb-1}  
\le d_{24} \gb 2^{\gb-1} \norm{\bX_1}^{\gb-1}  .
\end{align}
Summing \eqref{de124} and \eqref{de324} yields,
using again \eqref{dij},
\begin{align}\label{dea}
\bigabs{\hXb} 
&\le  \bigabs{d_{12}^\gb-d_{41}^\gb}
+  \bigabs{d_{23}^\gb-d_{34}^\gb}
\le \gb 2^{\gb} \norm{\bX_1}^{\gb-1} d_{24}
\notag\\&
\le \gb 2^{\gb} \norm{\bX_1}^{\gb-1} (\norm{\bX_2}+\norm{\bX_4}).
\end{align}
This proves \eqref{l1c} for any $\gb\ge1$.

If $1\le\gb\le2$, we further note that our assumption
$\norm{\bX_j}\le\norm{\bX_1}$ implies
\begin{align}
\norm{\bX_1}^{\gb-1} \norm{\bX_j}
\le
\norm{\bX_1}^{\gb/2} \norm{\bX_j}^{\gb/2},
\qquad j=1,\dots,4,
\end{align}
and thus \eqref{dea} also yields \eqref{l1b}.
\end{proof}

\begin{lemma}\label{L2}
  If\/ $\E \norm{\bX}^\gbx<\infty$, then
$\E \hXb^2<\infty$ and\/ $\E \tXb^2<\infty$.
\end{lemma}
For $\gb=1$, this is shown by \citet[Errata]{Lyons}.

\begin{proof}
\pfcase1{$\gb\le2$}
In this case $\gbx=\gb$.
  Recall that, by definition, $\bX_i$ and $\bX_{i\pm1}$ are independent.
Hence, 
\begin{align}
  \E\bigpar{\norm{\bX_i}^{\gb/2}\norm{\bX_{i+1}}^{\gb/2}}^2
=
\E \norm{\bX_i}^{\gb}\E\norm{\bX_{i+1}}^{\gb}<\infty,
\end{align}
so each term in the sum in \eqref{l1b} belongs to $L^2$, and thus
\eqref{l1b} implies $\hXb\in L^2$.
Since $\tXb$ is defined by \eqref{txb1} as a conditional expectation of
$\hXb$, this further implies $\tXb\in L^2$.

\pfcase2{$\gb\ge2$}
In this case $\gbx=2(\gb-1)\ge2$, and the result follows in the same way
from \eqref{l1c}.
\end{proof}

In the following lemma, we consider together with $\bX$ also
a sequences $(\bX\nn)_{n\ge1}$ of random variables in $\cX$.
We then define $\bX_i\nn$ for $i\ge1$ such that the random variables
$\bigpar{\bX_i,(\bX_i\nn)_n}$ in $\cX^\infty$
are independent copies of
$\bigpar{\bX,(\bX\nn)_n}$. 
This extends in the obvious way when we consider sequences 
$\bigpar{(\bX\nn,\bY\nn)}_n$.
We use the superscript ${}\nn$ in the natural way and let \eg{} %also \tXb
$\hXb\nn$ be defined as in \eqref{hxb} using $\bX_i\nn$.

\begin{lemma}\label{L3}
  Let\/ $\bX$ and\/ $\bX\nn$, $n\ge1$, be random variables in $\cX$, and assume
  that\/
$\E\norm{\bX}^\gbx<\infty$ and 
$\E d(\bX\nn,\bX)^\gbx\to0$ as \ntoo.
Then $\E\bigpar{\hXb\nn-\hXb}^2\to0$ and\/ $\E\bigpar{\tXb\nn-\tXb}^2\to0$.
\end{lemma}

\begin{proof}
We use without further comments some elementary facts about uniform
integrability, see \eg{} \cite[Theorems 5.5.4, 5.4.5 and 5.4.6]{Gut}.

  Since $\E d(\bX\nn,\bX)^\gbx\to0$,
the sequence $d(\bX\nn,\bX)^\gbx$ of random variables is uniformly integrable.
%\cite[Theorem 5.5.4]{Gut}
The triangle inequality yields
$\norm{\bX\nn}\le d(\bX\nn,\bX)+\norm{\bX}$, and thus
\begin{align}\label{kbt}
  \norm{\bX\nn}^\gbx\le C\bigpar{ d(\bX\nn,\bX)^\gbx+\norm{\bX}^\gbx},
\end{align}
and it follows that the sequence $\norm{\bX\nn}^\gbx$ is uniformly
integrable.
% \cite[Theorems 5.4.5 and 5.4.6]{Gut}
\refL{L1} and the argument in the proof of \refL{L2}, using \refL{LA1} in
the appendix, show that the sequence $(\hXb\nn)^2$ is uniformly integrable.

Furthermore, we have $d(\bX\nn,\bX)\pto0$, and thus
$d(\bX_i\nn,\bX_i)\pto0$ for every $i$.
The triangle inequality then implies
$d(\bX_i\nn,\bX_j\nn)\pto d(\bX_i,\bX_j)$ for every $i$ and $j$,
and thus the definition \eqref{hxb} implies $\hXb\nn\pto\hXb$.

This and the uniform square integrability just established yield
$\E\bigpar{\hXb\nn-\hXb}^2\to0$.

Furthermore, by \eqref{txb1},
if $\cF$ is the \gsf{} generated by 
all $\bX_j$ and $\bX_j\nn$ with $j\in\set{1,2}$, then $\tXb=\E(\hXb\mid\cF)$
and $\tXb\nn=\E(\hXb\nn\mid\cF)$. Consequently,
\begin{align}
  \E\bigabs{\tXb\nn-\tXb}^2
=
  \E\bigabs{\E(\hXb\nn-\hXb\mid\cF)}^2
\le   \E\bigabs{\hXb\nn-\hXb}^2
\to0.
\end{align}
\end{proof}

\begin{theorem}\label{T1}
\refDs{D1}--\ref{D3} are well-defined; more precisely, 
for any $\gb>0$,
assuming the stated moment conditions,
the expectations in \eqref{d1}, \eqref{d2} and \eqref{d3} are finite.
Furthermore, any two of these definitions yield the same result,
whenever the moment conditions in both are satisfied.
\end{theorem}
\begin{proof}
\refL{L0} shows that all three definitions are valid and agree under the
condition of \refD{D1}, \ie, when 
$\E \norm{\bX}^{2\gb}<\infty$ and $\E \norm{\bY}^{2\gb}<\infty$.

It remains to show that \eqref{d2} and \eqref{d3} are finite and agree
under the weaker assumption
$\E \norm{\bX}^{\gbx}<\infty$ and $\E \norm{\bY}^{\gbx}<\infty$.
In this case, \refL{L2} shows that $\hXb,\hYb,\tXb,\tYb\in L^2$, and thus
\eqref{d2} and \eqref{d3} are finite.

We do not know a simple direct argument to show the equality of the two
expressions, so we use truncations as follows. 
Let, for $n\ge1$,
\begin{align}
  \bX\nn:=
  \begin{cases}
    \bX, & \norm{\bX}\le n,
\\
\ox, & \text{otherwise},
  \end{cases}
\end{align}
and define $\bY\nn$ similarly.
Then 
\begin{align}
\E d(\bX\nn,\bX)^\gbx
= \E\bigsqpar{\norm{\bX}^\gbx\indic{\norm{\bX}>n}}
\asto0,
\qquad \asntoo.  
\end{align}
Thus, \refL{L3} yields 
$\normll{\hXb\nn-\hXb}\to0$ and $\normll{\tXb\nn-\tXb}\to0$.
Similarly,
$\normll{\hYb\nn-\hYb}\to0$ and $\normll{\tYb\nn-\tYb}\to0$.

The $L^2$-convergence just shown implies that, as \ntoo,
\begin{align}\label{tix}
  \hDCb(\bX\nn,\bY\nn)
=\tfrac{1}4\E\bigsqpar{\hXb\nn\hYb\nn}
\to \tfrac{1}{4}\E\bigsqpar{\hXb\hYb}
=   \hDCb(\bX,\bY)
\end{align}
and similarly
\begin{align}\label{tics}
  \tDCb(\bX\nn,\bY\nn)
=\E\bigsqpar{\hXb\nn\hYb\nn}
\to 
\E\bigsqpar{\hXb\hYb}=   
\tDCb(\bX,\bY),
\end{align}
Furthermore, for each $n$, $\norm{\bX\nn}$ and $\norm{\bY\nn}$ are bounded,
and thus \refL{L0} applies and shows
$  \hDCb(\bX\nn,\bY\nn)=  \tDCb(\bX\nn,\bY\nn)$.
Consequently, \eqref{tix}--\eqref{tics} imply
$\hDCb(\bX,\bY)=  \tDCb(\bX,\bY)$.
\end{proof}

We return in \refS{Sopt} to the case when the moment conditions fail.

\section{Continuity and consistency}\label{Scon}
The lemmas in \refS{Sexist} yield also  continuity results.
Unspecified convergence is as \ntoo. 

\begin{theorem}\label{T3}
Let $\gb>0$.
 Let\/ $(\bX,\bY)$ and\/ $(\bX\nn,\bY\nn)$, $n\ge1$, 
be pairs of random variables in $\cX\times\cY$, and assume
  that\/
$\E\norm{\bX}^\gbx<\infty$,
$\E\norm{\bY}^\gbx<\infty$ and,
as \ntoo, 
$\E d(\bX\nn,\bX)^\gbx\to0$ 
and\/
$\E d(\bY\nn,\bY)^\gbx\to0$. 
Then, %as \ntoo, 
\begin{align}
\DCb(\bX\nn,\bY\nn)\to\DCb(\bX,\bY).  
\end{align}
\end{theorem}
\begin{proof}
  \refL{L3} yields 
$\hXb\nn\llto\hXb$ and $\hYb\nn\llto\hYb$, % as \ntoo, 
%$\normll{\hXb\nn-\hXb}\to0$ and $\normll{\hYb\nn-\hYb}\to0$.
and thus
\begin{align}
\DCb(\bX\nn,\bY\nn)
=\frac{1}{4}\E\bigsqpar{\hXb\nn\hYb\nn}
\to \frac{1}{4}\E\bigsqpar{\hXb\hYb}=\DCb(\bX,\bY).  
\end{align}
\end{proof}

We can extend this result and assume only
convergence in distribution of 
$(\bX\nn,\bY\nn)$ together with a moment condition.

\begin{theorem}
  \label{T4}
Let $\gb>0$.
 Let\/ $(\bX,\bY)$ and\/ $(\bX\nn,\bY\nn)$, $n\ge1$, 
be pairs of random variables in $\cX\times\cY$, and assume
  that, as \ntoo,
$(\bX\nn,\bY\nn)\dto(\bX,\bY)$.
Assume further one of the following two conditions.
\begin{romenumerate}
\item \label{T4a}
The sequences $\norm{\bX\nn}^\gbx$ and $\norm{\bY\nn}^\gbx$
are \ui.
\item \label{T4b}
$\E\norm{\bX\nn}^\gbx\to\E\norm{\bX}^\gbx<\infty$
and\/
$\E\norm{\bY\nn}^\gbx\to\E\norm{\bY}^\gbx<\infty$.
%as \ntoo.
\end{romenumerate}
Then, %as \ntoo, 
\begin{align}\label{t4}
\DCb(\bX\nn,\bY\nn)\to\DCb(\bX,\bY).  
\end{align}
\end{theorem}
\begin{proof}
  \pfitemref{T4a}
Since $\cX\times\cY$ is a separable metric space,
we may by the Skorohod coupling theorem \cite[Theorem~4.30]{Kallenberg}
without loss of generality assume that $(\bX\nn,\bY\nn)\asto(\bX,\bY)$.
Furthermore, the assumption in \ref{T4a} implies that
$\sup_n\E\norm{\bX\nn}^\gbx<\infty$,
and thus
$\E\norm{\bX}^\gbx<\infty$ by Fatou's lemma.
Since $d(\bX\nn,\bX)\le\norm{\bX\nn}+\norm{\bX}$, it follows, similarly
to \eqref{kbt}, that the sequence $d(\bX\nn,\bX)^\gbx$ is \ui.
Since we have assumed $d(\bX\nn,\bX)\asto0$,
this implies $\E d(\bX\nn,\bX)^\gbx\to0$.
Similarly, $\E d(\bY\nn,\bY)^\gbx\to0$.
Thus \refT{T3} applies and yields \eqref{t4}.

\pfitemref{T4b}
We have $\bX\nn\dto\bX$ and thus $\norm{\bX\nn}\dto\norm{\bX}$.
This and our assumption $\E\norm{\bX\nn}^\gbx\to\E\norm{\bX}^\gbx$ imply
that the sequence $\norm{\bX\nn}^\gbx$ is \ui{}
\cite[Theorem 5.5.9]{Gut}.
The same holds for ${\bY\nn}$, and thus part \ref{T4a} applies.
\end{proof}

\begin{remark}\label{RT4}
Suppose that the metric spaces $\cX$ and $\cY$ are complete.
(This ensures that all probability measures are tight;
see \eg{} \cite{Billingsley}.)
Give $\cX\times\cY$ the metric (for example)
\begin{align}
d\bigpar{(\bx_1,\by_1),(\bx_2,\by_2)}:=d_\cX(\bx_1,\bx_2)+d_\cY(\by_1,\by_2).   
\end{align}
Let $\cP^\gb\xpar{\cXY}$ be the space of all Borel probability measures
$\mu$ on $\cXY$ 
such that $\int_{\cXY}\norm{(\bx,\by)}^\gb\dd\mu(\bx,\by)<\infty$.
In other words, $\cP^\gb\xpar{\cXY}$ is the space of all distributions of 
pairs of random
variables $(\bX,\bY)\in\cXY$ such that $\E\norm{\bX}^\gb<\infty$
and $\E\norm{\bY}^\gb<\infty$.

Define a metric in $\cP^\gb\xpar{\cXY}$ by
\begin{align}
  d_\gb(\mu,\mu'):=
  \begin{cases}
\inf\bigset{\E \bigsqpar{d\bigpar{(\bX,\bY),(\bX',\bY')}^\gb}},    
& 0<\gb\le1,
\\
\inf\bigset{\E \bigsqpar{d\bigpar{(\bX,\bY),(\bX',\bY')}^\gb}^{1/\gb}},    
& \gb>1.
  \end{cases}
\end{align}
taking the infimum over all pairs of random variables 
$(\bX,\bY)$ and $(\bX',\bY')$ in $\cXY$
such that $(\bX,\bY)\sim\mu$ and $(\bX',\bY')\sim\mu'$;
see \eg{}
\cite[pp.~796--799 (in the English translation)]{BogachevK}.
(This is known under various names, including \emph{Kantorovich distance},
\emph{Wasserstein distance} and \emph{minimal $L^\gb$ distance},
see also 
%\cite{BogachevK}
\cite{Wasserstein-metric}.)
Convergence of a sequence $\cL(\bX\nn,\bY\nn)$
of distributions to $\cL(\bX,\bY)$ in this metric is equivalent to 
convergence in distribution $(\bX\nn,\bY\nn)\dto(\bX,\bY)$
(\ie, weak convergence of the distributions)
together with uniform integrability of $\norm{\bX\nn,\bY\nn)}^\gbx$
(or, equivalently, convergence of  moments 
$\E\norm{(\bX\nn,\bY\nn)}^\gbx\to\E\norm{(\bX,\bY)}^\gbx$).

\refT{T3} then says that $\DCb$ is a continuous functional on
$\cP^\gbx\xpar{\cXY}$,
for every $\gb>0$.
\end{remark}

\subsection{Consistency}

Let $\mu\in\cP(\cX\times\cY)$ be the distribution of $(\bX,\bY)$.
Then, $(\bX_1,\bY_1),\dots$ can be regarded as a sequence of independent
samples from $\mu$.
Let $\nu_n$ be the empirical distribution of the first $n$ samples, \ie,
\begin{align}\label{nun}
  \nu_n := \frac{1}{n}\sumin \gd_{(\bX_i,\bY_i)} \in \cP(\cX\times\cY).
\end{align}
Note that $\nu_n$ is a random probability measure.
Hence, its distance covariance $\DCb(\nu_n)$ is a random variable.
The following theorem shows that this random variable converges to
$\DCb(\mu)$ \as; in other words, the distance covariance of the empirical
distribution is a consistent estimator of the covariance distance of $\mu$.
As said in the introduction, 
this was proved
by \citet{SRB2007} for the Euclidean case with $\gb\in(0,2)$;
for general metric spaces, with $\gb=1$, 
the result was stated by \citet{Lyons},  
but his proof requires a stronger moment condition.
Second moments are enough for $\gb=1$, 
see \cite[Remark 3]{SRB2007};
\citet{Jakobsen} improved this and showed that $5/3$
moments are enough. The proof in \cite[Remark 3]{SRB2007}
generalizes to arbitrary $\gb>0$, 
assuming $2\gb$ moments.

We can now show consistency assuming only
$\gbx$ moments, as required by our definitions.
In particular, this shows that for $\gb=1$,
    first moments  suffice, as stated in \cite{Lyons}.

\begin{theorem}\label{Tcons}
  Let $\mu$ be the distribution of $(\bX,\bY)\in\cX\times\cY$
and assume that
$\E\norm{\bX}^\gbx,\E\norm{\bY}^\gbx<\infty$.
If $\nu_n$ is the empirical distribution \eqref{nun}, then
\begin{align}\label{tcons}
\DCb(\nu_n)\asto  \DCb(\mu).
\end{align}
\end{theorem}

\begin{proof}
Conditionally on the sequence $(\nu_n)_n$ of empirical measures,
let $(\bX\nn,\bY\nn)$ be a random variable with distribution $\nu_n$.
Since $\cX\times\cY$ is a separable metric space, the distribution $\nu_n$
converges \as{} to $\mu$ (in the usual weak topology);
see \cite{Varadarajan} or \cite[Problem 4.4]{Billingsley}.
In other words, \as,
conditionally on $(\nu_n)_n$, 
$(\bX\nn,\bY\nn)\dto (\bX,\bY)$.

Furthermore, by the definition \eqref{nun} of $\nu_n$,
conditioning on the sequence $(\nu_k)_k$, 
\begin{align}
  \E\bigpar{ \norm{\bX\nn}^\gbx \mid (\nu_k)_k}
=\frac{1}{n}\sumin \norm{\bX_i\nn}^\gbx.
\end{align}
Hence, the strong law of large numbers (in $\bbR$) 
shows that \as, conditioned on $(\nu_k)_k$,
$  \E \norm{\bX\nn}^\gbx \asto \E\norm{\bX}^\gbx $,
and similarly also
$  \E \norm{\bY\nn}^\gbx \asto \E\norm{\bY}^\gbx $.
Consequently, \refT{T4}\ref{T4b} applies \as{} to the sequence
$(\nu_n)_n$ and the corresponding random variables $(\bX\nn,\bY\nn)$;
hence $\DCb(\nu_n)\asto\DCb(\mu)$.
\end{proof}

Our proofs of \refTs{T4} and \ref{Tcons} give no information on the rate of
convergence, leading to the following problems.

\begin{problem}\label{PT4}
  What is the rate of convergence in \eqref{t4},
under suitable hypotheses on $(\bX_n,\bY_n)$?
\end{problem}

\begin{problem}\label{Pcons}
  What is the rate of convergence in \eqref{tcons},
under suitable hypotheses on $(\bX,\bY)$?
\end{problem}

\section{Hilbert spaces, preliminaries}\label{SH1}

In this and the next two sections
we assume that $\cX$ and $\cY$ are separable Hilbert spaces;
we therefore change notation and
write $\cX=\HX$ and $\cY=\HY$.

We give our extension of \refD{D4} of covariance distance in \refS{SD5},
but we first need some preliminaries.

\subsection{Characteristic random variables}\label{Sch}

Let $\cH$ be a separable Hilbert space, of finite or infinite dimension
$\dimH$. %$\dim\cH$  

Fix an ON-basis $(\be_i)_1^{\dim \cH}$ in $\cH$, 
and let
$\xi_i$, $i=1,2,\dots$, be \iid{} $N(0,1)$ random variables.
Let $\bxi:=(\xi_i)_1^{\dimH}$, a
random vector of length $\dimH$ (finite or infinite).
Define for any
$\bx\in \cH$,
\begin{align}\label{cdot}
\bxi\cdot \bx=  \bx\cdot\bxi := \sum_{i=1}^{\dim \cH} \innprod{\bx,\be_i}\xi_i.
\end{align}
Note that in the finite-dimensional case, $\bxi\in \cH$ and this is the
usual inner product. In the infinite-dimensional case $\bxi\notin \cH$ a.s., 
but the sum in \eqref{cdot} converges \as{} since
$\sum_i|\innprod{\bx,\be_i}|^2=\norm{\bx}^2<\infty$. Hence, $\bxi\cdot \bx$ is
defined \as{} in any case. 
Note also that $\bxi\cdot\bx$ is a real-valued random variable, and that
\begin{align}
  \label{norma}
\bxi\cdot\bx\sim N\bigpar{0,\norm{\bx}^2}.
\end{align}

Let $\bX$ be an  $\cH$-valued random variable, and assume that $\bxi$ is
independent of $\bX$.
Then $\bxi\cdot\bX$ exists \as{}; % (by a standard Fubini argument);
thus $\bxi\cdot\bX$ is a well-defined real-valued random variable.
Consider the conditional expectation
\begin{align}\label{PhiX}
  \PhiX(\bxi):=\E\bigpar{e^{\ii\bxi\cdot\bX}\bigmid\bxi}.
\end{align}
This is a complex-valued random variable (determined \as),
which can be written as a (deterministic) function of $\bxi$.

In the finite-dimensional case $\dimH<\infty$, 
we may identify $\cH$ with $\bbR^q$, with
$(\be_j)_1^q$ as the standard basis.
Then \eqref{cdot} and \eqref{PhiX} show that 
\begin{align}\label{Phiphi}
\PhiX(\bxi)=\phiX(\bxi) \quad\text{a.s.},
\end{align}
where $\phiX(\bt):=\E e^{\ii \bt\cdot\bX}$ is the usual \chf.
For this reason, we say, for a general Hilbert space
$\cH$, that $\PhiX(\bxi)$ is the
\emph{characteristic random variable} of $\bX$.

Note that $\PhiX(\bxi)$ is a complex random variable, with 
\begin{align}
  \bigabs{\PhiX(\bxi)}\le1\quad\text{a.s.}
\end{align}
$\PhiX(\bxi)$
depends on the choices of $(\be_j)_j$ and $(\xi_j)_j$, but these choices
are regarded as fixed.
Moreover, the following theorem says that $\PhiX(\bxi)$
has the same fundamental property as the usual characteristic function:
it depends on $\bX$ only through its distribution, and conversely, it
characterizes the distribution.

\begin{theorem}
  \label{TPhi}
Let $\cH$ be a separable Hilbert space, and let $\bX$ and $\bY$ be
$\cH$-valued random variables.
Fix as above an ON-basis $(\be_i)_1^{\dim \cH}$ in $\cH$, 
and a random vector 
$\bxi:=(\xi_i)_1$
of \iid{} standard normal random variables
$\xi_i$, $i=1,2,\dots$, and assume further that these are independent of
$\bX$ and $\bY$.
Then
\begin{align}
  \bX\eqd \bY
\iff
\PhiX(\bxi)=\PhiY(\bxi)
\quad \text{a.s.}
\end{align}
\end{theorem}

We prove first a lemma that will help to reduce to the finite-dimensional case.

\begin{lemma}\label{Leps}
Let  $\bX$ be an $\cH$-valued  random variable
and let $\bxi=(\xi_i)_i$ be as  above, and in particular independent of $\bX$.
Then, for any $\eps>0$,
  the event
  $\bigset{\E\bigpar{1\land \abs{\bxi\cdot\bX}\bigmid\bxi}<\eps}$
has positive probability.

More generally, for any finite set of random variables $\bX^{(1)},\dots,\bX^{(m)}$
in $\HX$, all independent of $\bxi$, the events
  $\bigset{\E\bigpar{1\land \abs{\bxi\cdot\bX^{(j)}}\bigmid\bxi}<\eps}$
hold simultaneously with positive probability.
\end{lemma}

\begin{proof}
 For finite $N\le\dimH$, let
$\Pi_N$ be the orthogonal projection of $\HX$ onto the subspace $\cH_N$
  spanned by $\be_1,\dots,\be_N$.
Let $\XN:=\Pi_N\bX$ and
  $\bX_{>N}:=\bX-\XN$, and define $\bxi_{N}:=(\xi_1,\dots,\xi_N)$ and
  $\bxi_{>N}:=(\xi_{N+1},\xi_{N+2},\dots)$.
%(In this section, $\XN$ has this meaning, and not the one introduced in
%\refS{S:intro}.)
  Then we can write, 
interpreting the dot products in the obvious way 
in analogy with \eqref{cdot},
  \begin{align}\label{ea}
  \bxi\cdot\bX=
  \bxi_{N}\cdot\XN+\bxi_{>N}\cdot\bX_{>N} . 
  \end{align}

Assume in the remainder of the proof that 
$\dimH=\infty$; the case $\dimH<\infty$ is similar but
simpler, taking $N:=\dimH$ below so $\bX_{>N}=0$.

Since the sum in \eqref{cdot} converges \as, and $\bxi_{>N}\cdot\bX_{>N}$ is
the tail of this sum, it follows that 
$\bxi_{>N}\cdot\bX_{>N}\asto0$ 
as $\Ntoo$.
Consequently, by dominated convergence,
\begin{align}\label{gru}
  \E \bigpar{1\land \abs{\bxi_{>N}\cdot\bX_{>N}}}\to0
\qquad\text{as \Ntoo}.
\end{align}
Let 
\begin{align}
  \label{WN}
W_N:=  \E \bigpar{1\land \abs{\bxi_{>N}\cdot\bX_{>N}}\bigmid\bxi}.
\end{align}
Then \eqref{gru} shows $\E W_N\to0$; hence we may choose $N<\infty$ such
that
$\E W_N<\eps/4$.
%Let $\cE_N$ be the event $\set{W_N<\eps/2}$.
Then Markov's inequality yields
\begin{align}\label{kak}
  \P\bigpar{{W_N<\eps/2}}
\ge 1- \frac{\E W_N}{\eps/2}>\frac12.
\end{align}

Moreover, 
for each $i\le N$,
again by dominated convergence, 
\begin{align}
  \E \bigpar{1\land \abs{s\innprod{\bX,\be_i}}}\to0
\qquad
\text{as $s\to0$},
\end{align}
and thus there exists $\gd_i>0$ such that if $|s|<\gd_i$, then
\begin{align}\label{q10}
  \E \bigpar{1\land \abs{s\innprod{\bX,\be_i}}}<\frac{\eps}{2N}.
\end{align}
Recalling \eqref{ea} and \eqref{cdot},
we see that
\begin{align}
 \abs{\bxi\cdot\bX}
\le \sum_{i=1}^N  \abs{\xi_i\innprod{\bX,\be_i}}
+ \abs{\bxi_{>N}\cdot\bX_{>N}}
\end{align}
and thus
\begin{align}
1\land \abs{\bxi\cdot\bX}
\le \sum_{i=1}^N \bigpar{1\land \abs{\xi_i\innprod{\bX,\be_i}}}
+ \bigpar{1\land \abs{\bxi_{>N}\cdot\bX_{>N}}}.
\end{align}
Hence, recalling \eqref{WN},
\begin{align}
\E\bigpar{1\land \abs{\bxi\cdot\bX}\bigmid\bxi}
\le \sum_{i=1}^N \E\bigpar{1\land \abs{\xi_i\innprod{\bX,\be_i}}\bigmid\xi_i}
+ W_N. 
%+ \bigpar{1\land \abs{\bxi_{>N}\cdot\bX_{>N}}}.
\end{align}
Consequently, if $\bxi$ is such that 
${W_N<\eps/2}$ and $|\xi_i|<\gd_i$ for
$i=1,\dots,N$, then \eqref{q10} implies
\begin{align}\label{finis}
\E\bigpar{1\land \abs{\bxi\cdot\bX}\bigmid\bxi}
< \sum_{i=1}^N \frac{\eps}{2N}
+\frac{ \eps}2
=\eps.
\end{align}
Since the events
$\set{W_N<\eps/2}$ and $\set{|\xi_i|<\gd_i}$ are independent and 
each has positive probability, 
they occur together with positive probability, 
and thus \eqref{finis} holds with positive probability.

This proves the first part of the lemma.
The second is proved in the same way, choosing $N$ so large 
that \eqref{kak} holds with $W_N$ replaced by $\sum_{j=1}^mW_N^{(j)}$,
where $W_N^{(j)}$ is defined by \eqref{WN} but using $\bX^{(j)}$ instead of
$\bX$,
and then choosing  $\gd_i$ so small that
\eqref{q10} holds for each $\bX^{(j)}$
\end{proof}

\begin{proof}[Proof of \refT{TPhi}]
\pfitemx{$\implies$}
If $\bX\eqd\bY$, then
$(\bX,\bxi)\eqd(\bY,\bxi)$ and \eqref{cdot} implies 
$(\bxi\cdot\bX,\bxi)\eqd(\bxi\cdot\bY,\bxi)$
which by \eqref{PhiX} implies
$\PhiX(\bxi)=\PhiY(\bxi)$ 
a.s.

\pfitemx{$\impliedby$}  
We let $N\le\dimH$ be finite and
use the notation in the proof of \refL{Leps}.
Then \eqref{ea} holds,
  and thus
  \begin{align}
    \bigabs{ e^{\ii \bxi\cdot\bX}-e^{\ii \bxi_{N}\cdot\XN}}
    =     \bigabs{ e^{\ii  \bxi_{>N}\cdot\bX_{>N}}-1}
    \le 2\land |\bxi_{>N}\cdot\bX_{>N}|
.
  \end{align}
  Hence,
  \begin{align}\label{qc}
    \bigabs{\E\bigpar{ e^{\ii  \bxi\cdot\bX}\bigmid\bxi}
    -\E\bigpar{e^{\ii  \bxi_{N}\cdot\XN}\bigmid\bxi}}
    \le \E\bigpar{2\land |\bxi_{>N}\cdot\bX_{>N}|\bigmid\bxi}
\textas
  \end{align}
Using \eqref{PhiX},
\eqref{qc} can be written,
since $\bxi_{N}$ and $\bxi_{>N}$ are independent, 
  \begin{align}\label{qca}
\bigabs{\Phi_{\bX}(\bxi)-\Phi_{\XN}(\bxi_N)}
    \le \E\bigpar{2\land |\bxi_{>N}\cdot\bX_{>N}|\bigmid\bxi_{>N}}
\textas
  \end{align}
Similarly, with analoguous notation,
  \begin{align}\label{qcy}
\bigabs{\Phi_{\bY}(\bxi)-\Phi_{\YN}(\bxi_N)}
    \le \E\bigpar{2\land |\bxi_{>N}\cdot\bY_{>N}|\bigmid\bxi_{>N}}
\textas  
  \end{align}
The assumption $\PhiX(\bxi)=\Phi_{\bY}(\bxi)$ \as{} thus implies
  \begin{multline}\label{qcc}
\bigabs{\Phi_{\XN}(\bxi_N)-\Phi_{\YN}(\bxi_N)}
\\
    \le \E\bigpar{2\land |\bxi_{>N}\cdot\bX_{>N}|\bigmid\bxi_{>N}}
+\E\bigpar{2\land |\bxi_{>N}\cdot\bY_{>N}|\bigmid\bxi_{>N}}
\textas  
\end{multline}
\refL{Leps} (applied to $\bX_{>N}$ and $\bY_{>N}$) 
implies that
for any $\eps>0$, the \rhs{} of \eqref{qcc} is less than $4\eps$ with
positive probability.
Furthermore,
the \lhs{} of \eqref{qcc} is a function of $\bxi_{N}$, and the \rhs{} is a
function of $\bxi_{>N}$; thus the two sides are independent.
Consequently, \eqref{qcc} implies
\begin{align}\label{euro}
  \bigabs{\Phi_{\XN}(\bxi_N)-\Phi_{\YN}(\bxi_N)} <4\eps
\textas
\end{align}
Since $\eps$ is arbitrary, this shows
\begin{align}\label{rur}
\Phi_{\XN}(\bxi_N)=\Phi_{\YN}(\bxi_N)
\textas
\end{align}
Since $\XN$ and $\YN$ live in the finite-dimensional space 
$\cH_N$,
\eqref{Phiphi} applies and shows
\begin{align}\label{wuw}
\varphi_{\XN}(\bxi_N)=
\Phi_{\XN}(\bxi_N)=\Phi_{\YN}(\bxi_N)
=\varphi_{\YN}(\bxi_N)
\textas,
\end{align}
where $\varphi_{\XN}(\bt)$ and $\varphi_{\YN}(\bt)$ are the ordinary
characteristic functions in $\bbR^N$ (identified with $\cH_N$).
Hence,
\begin{align}\label{quq}
\varphi_{\XN}(\bt)
=\varphi_{\YN}(\bt)
\end{align}
for \aex{} $\bt\in\bbR^N$, and since characteristic functions are
continuous, \eqref{quq} holds for all $\bt\in\bbR^N$, and thus
\begin{align}\label{eqN}
  \XN\eqd \YN.
\end{align}

If $\dimH<\infty$, we may choose $N=\dimH$ and the result $\bX\eqd \bY$
follows. 
(Much of the argument above is not needed  in this case.)

If $\dimH=\infty$, then \eqref{eqN} holds for every finite $N$.
Furthermore, as \Ntoo, we have $\XN\asto\bX$ and thus $\XN\dto\bX$
and similarly $\YN\dto\bY$.
Consequently,
$\bX\eqd\bY$, which completes the proof.
\end{proof}

\begin{remark}\label{Rito}
  The mapping $\bx\mapsto\bxi\cdot\bx$ is an isometry of $\HX$ onto the
  Gaussian Hilbert space spanned by the random variables $\xi_i$, and it can
  be regarded as an abstract stochastic integral,
  \cf{} \cite[Chapter VII.2]{SJIII}.
It replaces the It\^o integrals used in %the proof in 
\cite{Mikosch}.
\end{remark}

\begin{remark}\label{Rhphi}
The arguments above are related to the proof of \cite[Theorem 3.16]{Lyons}.
 We sketch the connection:
That proof uses an embedding $\phi$
of the Hilbert space into $L^2(\bbR^\infty\times\bbR)$; if we compose 
$\phi$ %this embedding 
with the Fourier transform $f\mapsto\int e^{2\pi\ii tx}f(x)\dd x$
acting on the last variable (which is an isometry), 
we obtain
an equivalent embedding $\hphi$, which in our notation equals 
\begin{align}
\hphi:  \bx\to \frac{\ii}{2\pi t}\bigpar{e^{\ii c' t \bxi\cdot\bx}-1}
\in L^2(\P\times \ddx t)
\end{align}
for a constant $c'>0$. Hence, if $\mu=\cL(\bX)$, the distribution of $\bX$,
then, combining the notation of \cite{Lyons} and ours,
\begin{align}
  \beta_{\hphi}(\mu):=\E\bigpar{ \phi'(X)\mid\bxi}
= \frac{\ii}{2\pi t}\bigpar{\Phi_{c' t\bX}(\xi )-1}.
\end{align}
Hence, the result in \cite[Theorem 3.16]{Lyons} 
that $\beta_{\phi}(\mu)$ characterises $\mu$ is closely related to,
and follows from, \refT{TPhi}. 
Furthermore, the two proofs are similar; both are
based on approximating with the finite-dimensional case which is easy.
\end{remark}

\subsection{Independence and characteristic random variables}
\label{SSind}
Now consider a pair of random variables $(\bX,\bY)$ taking values in two,
possibly different, separable Hilbert spaces $\HX$ and $\HY$.
Fix, as above, 
an ON-basis $(\be_i)_1^{\dim \cH}$ in $\cH$, 
and \iid{} $N(0,1)$ random variables
$\xi_i$, $i=1,2,\dots$.
Similarly, fix
an ON-basis $(\bex_j)_1^{\dim \HY}$ in $\HY$, 
and \iid{} $N(0,1)$ random variables
$\eta_j$, $j=1,2,\dots$.
Assume that all $\xi_i$ and $\eta_j$ are independent of each other and of
$(\bX,\bY)$.

Then $(\bX,\bY)$ is a random variable in the Hilbert space
$\HX\oplus\HY=\HX\times\HY$, and
$\be_1,\bex_1,\be_2,\bex_2,\dots$ is an ON-basis in this space.
Let $\bxi=(\xi_i)_1^{\dim\HX}$,
 $\bbeta:=(\eta_i)_1^{\dim\HY}$,
and 
 $\bzeta:=(\xi_1,\eta_1,\xi_2,\eta_2,\dots)$.
\begin{theorem}\label{Tind}
Let $(\bX,\bY)$ be a pair of random variables taking values in separable
Hilbert spaces $\HX$ and $\HY$.
Then, with notation as above,
$\bX$ and $\bY$ are independent if and only if
\begin{align}\label{tind}
  \E\bigpar{e^{\ii \bxi\cdot\bX+\ii\bbeta\cdot\bY}\bigmid\bxi,\bbeta}
=
  \E\bigpar{e^{\ii \bxi\cdot\bX}\bigmid\bxi}
  \E\bigpar{e^{\ii\bbeta\cdot\bY}\bigmid\bbeta}
\textas
\end{align}
\end{theorem}
\begin{proof}
  Let $\bY'$ be a copy of $\bY$, independent of $\bX,\bxi,\bbeta$.
Then, $\bX$ and $\bY$ are independent if and only if
$(\bX,\bY)\eqd(\bX,\bY')$, and the result follows from 
\refT{TPhi}, applied to the Hilbert space $\HX\times\HY$, noting that
with the bases and Gaussian variables above,
$
 \bzeta\cdot(\bX,\bY)=
 \bxi\cdot\bX+\bbeta\cdot\bY
$ \as,
and thus
\begin{align}\label{tindra}
  \Phi_{(\bX,\bY)}(\bzeta)
= \E\bigpar{e^{\ii \bzeta\cdot(\bX,\bY)}\bigmid\bxi,\bbeta}
= \E\bigpar{e^{\ii \bxi\cdot\bX+\ii\bbeta\cdot\bY}\bigmid\bxi,\bbeta},
\end{align}
while, by independence and $\bY\eqd\bY'$,
\begin{align}
  \Phi_{(\bX,\bY')}(\bzeta)
&= \E\bigpar{e^{\ii \bxi\cdot\bX+\ii\bbeta\cdot\bY'}\bigmid\bxi,\bbeta}
=
  \E\bigpar{e^{\ii \bxi\cdot\bX}\bigmid\bxi}
  \E\bigpar{e^{\ii\bbeta\cdot\bY'}\bigmid\bbeta}
\notag\\&=
  \E\bigpar{e^{\ii \bxi\cdot\bX}\bigmid\bxi}
  \E\bigpar{e^{\ii\bbeta\cdot\bY}\bigmid\bbeta}
\textas\label{blanka}
\end{align}
\end{proof}

Note that, by \eqref{Cov},  \eqref{tind} may be written
\begin{align}\label{olivia}
  \Cov\Bigpar{e^{\ii \bxi\cdot\bX}, e^{\ii \bbeta\cdot\bY}\bigmid\bxi,\bbeta}=0\textas
\end{align}

\section{Covariance distance in Hilbert space}\label{SD5}

We give a new definition of covariance distance for Hilbert spaces;
it can be seen as a version of
\refD{D4} for Euclidean spaces, where we
replace the \chf{s} there by the characteristic random variables
defined in \refS{SH1}, which makes the extension to 
infinite-dimensional Hilbert spaces possible.
(The definition is inspired by \cite[Lemma 4.1]{Mikosch}; see \refR{Rito}.)

 Define, for $0<\gb<2$, 
\begin{align}\label{cb}
  \cb:=\frac{2^{1+\gb/2}}{-\gG(-\gb/2)}
=\frac{\gb 2^{\gb/2}}{\gG(1-\gb/2)}.
\end{align}
\begin{definition}\label{D5}
Let\/ $(\bX,\bY)$ be a pair of random vectors in separable Hilbert spaces,
and let $0<\gb<2$.
Then, with notation as in \refS{SH1},
  \begin{align}
&  \DCb(\bX,\bY)
=   \HDCb(\bX,\bY)
\notag\\&\quad
:=
\cb^2\intoo\intoo
\E\bigabs{\Phi_{(r\bX,s\bY)}(\bxi,\bbeta)-
\Phi_{r\bX}(\bxi)\Phi_{s \bY}(\bbeta)}^2
    \frac{\ddx r\dd s}{r^{\gb+1} s^{\gb+1}}
\label{d5a}
\\&\quad
\phantom:
=
\cb^2\intoo\intoo
\E \Bigabs{ \E \bigpar{e^{\ii r \bxi\cdot\bX+\ii s\bbeta\cdot\bY}\mid\bxi,\bbeta}
         - \E \bigpar{e^{\ii r \bxi\cdot\bX}\mid\bxi}
                      \E \bigpar{e^{\ii s\bbeta\cdot\bY}\mid\bbeta}}^2
\notag\\&\hskip24em
    \frac{\ddx r\dd s}{r^{\gb+1} s^{\gb+1}}
\label{d5b}
\\&\quad
\phantom:
=
   \cb^2\intoo\intoo
\E \bigabs{\Cov\Bigpar{e^{\ii r \bxi\cdot\bX}, e^{\ii s\bbeta\cdot\bY}\mid\bxi,\bbeta}}^2
    \frac{\ddx r\dd s}{r^{\gb+1} s^{\gb+1}}.
\label{d5c}
  \end{align}
\end{definition}

The expressions \eqref{d5a}--\eqref{d5c} are equal by the definitions 
\eqref{PhiX} and \eqref{Cov} above, \cf{} \eqref{tindra} and \eqref{olivia}.
Note that no moment assumtions are made; as for \refD{D4},
the definition works for any $(\bX,\bY)$ in these spaces, 
but $\DCb(\bX,\bY)$ may be infinite.
Furthermore, as shown in the next theorem, for the special case of Euclidean
spaces, \refD{D5} agrees with \refD{D4}, again without moment conditions.

\begin{theorem}\label{TEH}
Let $0<\gb<2$.
If\/ $(\bX,\bY)$ is a pair of random vectors in Euclidean spaces $\bbR^p$ and
$\bbR^q$, then \refDs{D4} and \ref{D5} agree,
\ie, $\EDCb(\bX,\bY)=\HDCb(\bX,\bY)$.
\end{theorem}

\begin{proof}%[Proof of the equivalence of \refDs{D4} and \ref{D5}]
  Assume that $\HX=\bbR^p$ and $\HY=\bbR^q$.
Then \eqref{Phiphi} implies
\begin{align}
  \Phi_{(r\bX,s\bY)}(\bxi,\bbeta)
=  \gf_{(r\bX,s\bY)}(\bxi,\bbeta)
=  \gf_{(\bX,\bY)}(r\bxi,s\bbeta)
\end{align}
and thus, since $r\bxi\sim N(0,r^2I_p)$ and $s\bbeta\sim N(0,s^2I_q)$,
where $I_k$ is the identity matrix in $\bbR^k$,
\begin{align}
&\E\bigabs{\Phi_{(r\bX,s\bY)}(\bxi,\bbeta)-\Phi_{r\bX}(\bxi)\Phi_{s \bY}(\bbeta)}^2
=
\E\bigabs{\gf_{(\bX,\bY)}(r\bxi,s\bbeta)-\gf_{\bX}(r\bxi)\gf_{\bY}(s\bbeta)}^2
\notag\\&=
\int_{\bt\in\bbR^p}\int_{\bu\in\bbR^q}
\bigabs{\gf_{(\bX,\bY)}(\bt,\bu)-\gf_{\bX}(\bt)\gf_{\bY}(\bu)}^2
\frac{e^{-|\bt|^2/2r^2}}{(2\pi r^2)^{p/2}}
\frac{e^{-|\bu|^2/2s^2}}{(2\pi s^2)^{q/2}}
\dd\bt\dd\bu.
\end{align}
Substituting this in \eqref{d5a}, we obtain \eqref{d4} by interchanging the
order of integration, because, by elementary calculations,
\begin{align}
  \intoo \frac{e^{-|\bt|^2/2r^2}}{(2\pi r^2)^{p/2}}  \frac{\ddx r}{r^{\gb+1}}
=
\frac{2^{\gb/2-1}\gG((p+\gb)/2)}{\pi^{p/2}}|\bt|^{-p-\gb}
=\frac{\cbx{p}}{\cb}|\bt|^{-p-\gb},
\end{align}
see \eqref{cbx} and \eqref{cb},
and similarly for the integral over $s$.
\end{proof}

\begin{remark}\label{R2divH}
  The proof of \refT{TEH} together with \refR{R2div} shows that the
  restriction $\ga<2$ in \refD{D5} is necessary; for $\ga\ge2$, the
  integrals diverge typically, for example for $\HX=\HY=\bbR$ and
  $\bX=\bY\sim N(0,1)$.
(We conjecture that for $\ga\ge2$, the integrals always diverge except when 
$\bX$ and $\bY$ are independent, but we have not verified that.)
\end{remark}

We return to the general Hilbert space case, and show that \refD{D5} agrees
with the earlier ones;
this is an abstract version of \cite[Lemma 4.1]{Mikosch},
where the Hilbert spaces are $L^2\oi$, see
\refR{Rito}.

\begin{theorem}\label{T52}
Let $0<\gb<2$.
If\/ $(\bX,\bY)$ is a pair of random vectors in Hilbert spaces $\HX$ and
$\HY$, 
and $\E\norm{\bX}^\gb<\infty$ and $\E\norm{\bY}^\gb<\infty$,
then \refDs{D2}, \ref{D3} and \ref{D5} agree,
\ie, $\HDCb(\bX,\bY)=\hDCb(\bX,\bY)=\tDCb(\bX,\bY)$,
and this value is finite.
\end{theorem}

\begin{proof}
Let again $(\bX_1,\bY_1),\dots$ be \iid{} copies of
$(\bX,\bY)$, and assume that $\bxi$ and $\bbeta$ are independent of all of
them. Then,
using \eqref{tindra}--\eqref{blanka} and \eqref{norma},
\begin{align}
&\E\bigabs{\Phi_{(r\bX,s\bY)}(\bxi,\bbeta)-\Phi_{r\bX}(\bxi)\Phi_{s
                \bY}(\bbeta)}^2
%&\E\Bigabs{\E\bigpar{e^{\ii r\bxi\cdot\bX+\ii s\bbeta\cdot\bY}\mid\bxi,\bbeta}
%         - \E \bigpar{e^{\ii r \bxi\cdot\bX}\mid\bxi}
%                      \E \bigpar{e^{\ii s\bbeta\cdot\bY}\mid\bbeta}}^2
\notag\\&\quad
= \E \E\Bigsqpar{
\Bigpar{e^{\ii r \bxi\cdot\bX_1+\ii s\bbeta\cdot\bY_1}
  -e^{\ii r \bxi\cdot\bX_1+\ii s\bbeta\cdot\bY_2}}
\Bigpar{e^{-\ii r \bxi\cdot\bX_3-\ii s\bbeta\cdot\bY_3} 
  -e^{-\ii r \bxi\cdot\bX_3-\ii s\bbeta\cdot\bY_4}}
\bigmid \bxi,\bbeta}
\notag\\&\quad
= \E\Bigsqpar{
\Bigpar{e^{\ii r \bxi\cdot\bX_1+\ii s\bbeta\cdot\bY_1}
  -e^{\ii r \bxi\cdot\bX_1+\ii s\bbeta\cdot\bY_2}}
\Bigpar{e^{-\ii r \bxi\cdot\bX_3-\ii s\bbeta\cdot\bY_3} 
  -e^{-\ii r \bxi\cdot\bX_3-\ii s\bbeta\cdot\bY_4}}}
\notag\\&\quad
=
\E e^{\ii r \bxi\cdot(\bX_1-\bX_3)+\ii s\bbeta\cdot(\bY_1-\bY_3)}
-
\E e^{\ii r \bxi\cdot(\bX_1-\bX_3)+\ii s\bbeta\cdot(\bY_1-\bY_4)}
\notag\\&\hskip4em{} -
\E e^{\ii r \bxi\cdot(\bX_1-\bX_3)+\ii s\bbeta\cdot(\bY_2-\bY_3)}
+
\E e^{\ii r \bxi\cdot(\bX_1-\bX_3)+\ii s\bbeta\cdot(\bY_2-\bY_4)}
\notag\\&\quad=
\E e^{-\rrq\norm{\bX_1-\bX_3}^2-\ssq\norm{\bY_1-\bY_3}^2}
-
\E e^{-\rrq\norm{\bX_1-\bX_3}^2-\ssq\norm{\bY_1-\bY_4}^2}
\notag\\&\hskip4em
-
\E e^{-\rrq\norm{\bX_1-\bX_3}^2-\ssq\norm{\bY_2-\bY_3}^2}
+
\E e^{-\rrq\norm{\bX_1-\bX_3}^2-\ssq\norm{\bY_2-\bY_4}^2}
.\label{jb}
\end{align}
Define the real-valued random variable
\begin{align}\label{GLX}
\GLX(u)
:=\E e^{-u\norm{\bX_1-\bX_2}^2} 
-\E e^{-u\norm{\bX_2-\bX_3}^2}
+\E e^{-u\norm{\bX_3-\bX_4}^2}
- \E e^{-u\norm{\bX_4-\bX_1}^2}
\end{align}
and define $\GLY(u)$ similarly.
Then, by expanding the product and using symmetry,
\begin{align}
&  \E\bigsqpar{\GLX(u)\GLY(v)}
\notag\\&\hskip1em
=4\Bigl(
\E e^{-u\norm{\bX_1-\bX_3}^2-v\norm{\bY_1-\bY_3}^2}
-
\E e^{-u\norm{\bX_1-\bX_3}^2-v\norm{\bY_1-\bY_4}^2}
\notag\\&\hskip3em
-
\E e^{-u\norm{\bX_1-\bX_3}^2-v\norm{\bY_2-\bY_3}^2}
+
\E e^{-u\norm{\bX_1-\bX_3}^2-v\norm{\bY_2-\bY_4}^2}
\Bigr).
\label{ja}
\end{align}
Consequently,
\eqref{jb} yields
\begin{align}\label{jo}
&\E\bigabs{\Phi_{(r\bX,s\bY)}(\bxi,\bbeta)-\Phi_{r\bX}(\bxi)\Phi_{s \bY}(\bbeta)}^2
%\notag\\&\quad
= \frac{1}{4}
 \E\Bigsqpar{\GLX\Bigpar{\frac{r^2}{2}}\GLY\bigpar{\frac{s^2}2}}
\end{align}
and the definition \eqref{d5b} yields, with a change of variables,
\begin{align}
  \HDCb(\bX,\bY)&
= \frac{\cb^2}4\intoo \intoo 
\E\Bigsqpar{\GLX\Bigpar{\frac{r^2}{2}}\GLY\bigpar{\frac{s^2}2}}
    \frac{\ddx r\dd s}{r^{\gb+1} s^{\gb+1}}
\notag\\&
= \frac{\cb^2}{2^{4+\gb}}\intoo \intoo 
\E\bigsqpar{\GLX\xpar{u}\GLY\xpar{v}}
    \frac{\ddx u\dd v}{u^{\gb/2+1} v^{\gb/2+1}}
.\label{jc}
\end{align}
We rewrite \eqref{GLX} as, 
with indices interpreted \MOD4,
\begin{align}\label{jd}
  \GLX(u)=\sumiv (-1)^{i-1}e^{-u\norm{\bX_i-\bX_{i+1}}^2}
=\sumiv (-1)^{i}\bigpar{1-e^{-u\norm{\bX_i-\bX_{i+1}}^2}}.
\end{align}
Recall that for $0<\gam<1$,
see \cite[(5.9.5)]{NIST},
\begin{align}\label{je}
  \intoo\bigpar{1-e^{-x}}x^{-\gam-1}\dd x = -\gG(-\gam).
\end{align}
Hence, \eqref{jd} and a change of variables yield
\begin{align}\label{jf}
\intoo  \GLX(u)  \frac{\ddx u}{u^{\gb/2+1}} 
 & 
=\sumiv (-1)^{i}\intoo\bigpar{1-e^{-u\norm{\bX_i-\bX_{i+1}}^2}} 
\frac{\ddx u}{u^{\gb/2+1}}
%=\sumiv (-1)^{i}\norm{\bX_i-\bX_{i+1}}^\gb\intoo\bigpar{1-e^{-u}} 
\notag\\&
=-\gG(-\gb/2)\sumiv (-1)^{i}\norm{\bX_i-\bX_{i+1}}^\gb
\notag\\&
=\gG(-\gb/2)\hXb.
\end{align}
If we naively interchange order of integrations and expectation in
\eqref{jc}, and use \eqref{jf}, we obtain \eqref{d2}
and thus $\HDCb(\bX,\bY)=\hDCb(\bX,\bY)$,
since $\cb$ is defined in \eqref{cb}
so that constant factors cancel.
However, this interchange requires justification; indeed it is not always
allowed, since the expectation in \eqref{d2} does not always exist, not even
as an extended real number, see \refE{ED2bad}, while \eqref{d5a}--\eqref{d5c}
always exist in $[0,\infty]$.

Hence, we introduce an integrating factor.
Let $M>0$; we will later let $\Mtoo$. 
Similarly to \eqref{jf}, we have
\begin{align}\label{jfa}
&\intoo  e^{-Mu}\GLX(u)  \frac{\ddx u}{u^{\gb/2+1}} 
\notag\\&\qquad
=\sumiv (-1)^{i}\intoo\bigpar{e^{-Mu}-e^{-u(\norm{\bX_i-\bX_{i+1}}^2+M)}} 
\frac{\ddx u}{u^{\gb/2+1}}
\notag\\&\qquad
=-\gG(-\gb/2)\sumiv (-1)^{i}
  \bigpar{\bigpar{\norm{\bX_i-\bX_{i+1}}^2+M}^{\gb/2}-M^{\gb/2}}.
\end{align}
Let $\gb\in(0,2)$ be given and
define, for $x\ge0$,
\begin{align}\label{hM}
  h_M(x):=x^{\gb/2}+M^{\gb/2}-(x+M)^{\gb/2}.
\end{align}
Then, \eqref{jf} and \eqref{jfa} yield
\begin{align}\label{jfb1}
\intoo \bigpar{1- e^{-Mu}}\GLX(u)  \frac{\ddx u}{u^{\gb/2+1}} 
%\notag\\&\qquad
&=\gG(-\gb/2)\sumiv (-1)^{i-1} h_M\bigpar{\norm{\bX_i-\bX_{i+1}}^2}
\notag\\&
=:\gG(-\gb/2)\hXbM,
\end{align}
%where $\hXbM$ thus is defined as the sum in \eqref{jfb}.
where thus we define
\begin{align}\label{jfb2}
  \hXbM:=\sumiv (-1)^{i-1} h_M\bigpar{\norm{\bX_i-\bX_{i+1}}^2}.
\end{align}
Note also that the integrand in \eqref{jc} is non-negative by \eqref{jo}.
Hence, \eqref{jc} and monotone convergence yield
\begin{align}
&  \HDCb(\bX,\bY)
\notag\\&
= \lim_\Mtoo\frac{\cb^2}{2^{4+\gb}}\intoo \intoo 
\bigpar{1-e^{-Mu}}\bigpar{1-e^{-Mv}}
\E\bigsqpar{\GLX\xpar{u}\GLY\xpar{v}}
    \frac{\ddx u\dd v}{u^{\gb/2+1} v^{\gb/2+1}}
.\label{jcm}
\end{align}
Furthermore, $|\GLX(u)|$ and $|\GLY(v)|$ are bounded (by 4) by \eqref{jd}, and
thus Fubini applies so we may interchange expectation and integrations in
\eqref{jcm}, which by \eqref{jfb1} yields, recalling \eqref{cb},  
\begin{align}
  \HDCb(\bX,\bY)
= \lim_\Mtoo\tfrac{1}{4}\E [\hXbM\hYbM]
.\label{jcg}
\end{align}

Since $\gb/2\in(0,1)$, the function $h_M$ 
in \eqref{hM}
is increasing, with $h_M(0)=0$ and
$h_M(x)\upto M^{\gb/2}$ as $\xtoo$.
Similarly, $h_M(x)=h_x(M)\upto x^{\gb/2}$ as $\Mtoo$; 
hence, the definitions \eqref{jfb2} and \eqref{hxb} yield
\begin{align}\label{jh}
  \hXbM\to\hXb
\qquad\text{as \Mtoo}.
\end{align}
Furthermore, if $0\le x\le y$, then
\begin{align}
  0\le h_M(y)-h_M(x) \le y^{\gb/2}-x^{\gb/2},
\end{align}
and it follows that for any $\bZ_1,\bZ_2\in\HX$, 
\begin{align}\label{jg}
\bigabs{h_M\bigpar{\norm{\bZ_1}^2} - h_M\bigpar{\norm{\bZ_2}^2}}
\le
\bigabs{\norm{\bZ_1}^\gb - \norm{\bZ_2}^\gb}.
\end{align}

We claim that \refL{L1} holds for $\hXbM$ too, so that, in particular,
\begin{align}\label{l1bM}
  |\hXbM|
\le C\sumiv \norm{\bX_i}^{\gb/2}\norm{\bX_{i+1}}^{\gb/2},
\end{align}
where the constant $C$ does not depend on $M$.
This is seen by repeating the proof of \refL{L1}, 
recalling the definition \eqref{jfb2} of $\hXbM$ and
using \eqref{jg}; we omit the details.

Let $\Xxx$ be the \rhs{} of \eqref{l1bM}.
We now use the assumption $\E\norm{\bX}^\gb<\infty$, which implies that
$\Xxx\in L^2$. Similarly, $|\hYbM|\le\Yxx$ with $\Yxx\in L^2$.
Consequently, $|\hXbM\hYbM|\le\Xxx\Yxx\in L^1$, so dominated convergence
applies to \eqref{jcg} and we obtain, by \eqref{jh},
\begin{align}
  \HDCb(\bX,\bY)
= \tfrac{1}{4}\E [\lim_\Mtoo\hXbM\hYbM]
= \tfrac{1}{4}\E [\hXb\hYb]
=  \hDCb(\bX,\bY),
\label{jck}
\end{align}
using \eqref{d2}. 
Hence, \refDs{D5} and \ref{D2} agree
(under the given moment condition).
By \refT{T1}, they agree with \refD{D3} too; furthermore, the value is finite.
\end{proof}

\begin{remark}\label{RD5M}
  Note that the proof shows that %with \refD{D5},
\eqref{jcg} holds for any random variables in
 Hilbert spaces, without any moment condition. 
(With the result possibly  $+\infty$.) 
\end{remark}

\subsection{Independence and distance covariance}\label{SSindDC}
For (separable) Hilbert spaces,
as said in the introduction,
\citet[Theorem 3.16]{Lyons} showed that
\eqref{iff} holds for $\gb=1$,
and
\citet[Theorem 4.2]{Mikosch} extended this to all $\gb\in(0,2)$. 
That is, they proved (a version of) the following, which we now easily can
prove using the results above.

\begin{theorem}[{\citet[Theorem 4.2]{Mikosch}}]\label{TH}
Let $\cX=\HX$ and $\cY=\HY$ be separable Hilbert spaces
and let $\gb\in(0,2)$.
Use \refD{D1}, \ref{D2}, \ref{D3} or \ref{D5}, and assume (for the first
three)
the moment condition there.
Then $\DCb(\bX,\bY)=0$ if and only if\/ $\bX$ and $\bY$ are independent.
\end{theorem}

\begin{proof}
For \refDs{D1}--\ref{D3}, the moment condition there and \refTs{T1} and
\ref{TEH} show that $\DCb(\bX,\bY)$ equals 
$\HDCb(\bX,\bY)$  given by \refD{D5}.
Hence, we may in all cases use $\HDCb$. 
It follows from \eqref{d5b} that $\HDCb(\bX,\bY)=0$ if and only if
\eqref{tind} holds, and the result follows by \refT{Tind}.
\end{proof}

\begin{remark}\label{RH}
  This theorem is stated in \cite{Mikosch} for the case $\cX=\cY=L^2\oi$ (so
  $\bX$ and $\bY$ are stochastic processes on $\oi$),
but since all infinite-dimensional separable Hilbert spaces are isomorphic;
the result can be stated as above.
(Only stochastic processes $\bX,\bY$ that satisfy some
smoothness conditions 
are considered in \cite{Mikosch}, but this is for other reasons and
is not needed for \refT{TH}.)

The theorem in \cite{Mikosch} is stated assuming only finite $\gb$ moments,
as we do above 
for \refDs{D2} and \ref{D3}; 
however, \cite{Mikosch}  uses
\refD{D1} which in general requires somewhat more for existence, see
\refT{TD1?} below.
\end{remark}

\begin{remark}
\refT{TH} includes the case when $\cX$ or $\cY$ has finite dimension,
\ie, is a Euclidean space.

Furthermore, although the theorem is stated for separable Hilbert spaces, it
extends also to non-separable spaces, provided we assume that
$\bX$ and $\bY$ are Bochner measurable, for the trivial reason that 
this implies that $\bX$ and $\bY$ \as{} take values in some separable
subspaces $\HX_1$ and $\HY_1$.
\end{remark}

\begin{remark}
  The proof of \refT{TH} would be much simpler if distance covariance was
  monotone under orthogonal projections, so that we would have
$\DCb(\Pi_N\bX,\Pi_N\bY)\le \DCb(\bX,\bY)$. 
However, this is not always the case,
even in finite dimension, as is seen by the following example.
\end{remark}

\begin{example}
  \label{Eproj}
Let $\cX=\cY=\bbR^2$ and let $\bX=(X',X'')$ and $\bY=(Y',Y'')$, where
$X'=Y'$, but $X',X'',Y''$ are independent and non-degenerate.
(For definiteness, we may take $X',X'',Y''\sim\Be(1/2)$, or $N(0,1)$.)
Let $\Pi:\bbR^2\to\bbR$ be the standard projection onto the first
coordinate, so $(\Pi\bX,\Pi\bY)=(X',Y')$.

For $a\in\bbR$, let $\bX(a):=(X',aX'')$ and $\bY(a):=(Y',aY'')$; thus
$(\bX(1),\bY(1))=(\bX,\bY)$ and 
$(\bX(0),\bY(0))=(X',Y')$ (regarding $\bbR$ as a subspace of $\bbR^2$).
For $\bt=(t',t'')$ and $\bu=(u',u'')$, we have
\begin{align}\label{rex}
  \gf_{\bX(a),\bY(a)}(\bt,\bu)
&=\E e^{\ii(t'X'+u'X'+t''aX''+u''aY'')}
\notag\\&
=\gf_{X'}(t'+u')\gf_{X''}(at'')\gf_{Y''}(au'')
\end{align}
and similarly (or by  taking $\bt=0$ or $\bu=0$ in \eqref{rex})
\begin{align}
    \gf_{\bX(a)}(\bt)
=\gf_{X'}(t')\gf_{X''}(at''),
&&&
  \gf_{\bY(a)}(\bu)
=\gf_{X'}(u')\gf_{Y''}(au'').
\end{align}
Hence, \eqref{d4} yields
{\multlinegap=0pt
\begin{multline}
  \DCb(\bX(a),\bY(a))
\\
=
\cbx{2}\cbx{2}\int_{\bt\in\bbR^2}\int_{\bu\in\bbR^2}
\bigabs{
\gf_{X'}(t'+u')-\gf_{X'}(t')\gf_{X'}(u')}^2
\bigabs{\gf_{X''}(at'')\gf_{Y''}(au'')}^2
\\
\frac{\dd\bt\dd\bu}
{|\bt|^{2+\gb} |\bu|^{2+\gb}}
\end{multline}}%
and it is obvious that
\begin{multline}\label{mv}
 \DCb(\bX,\bY)
=  \DCb(\bX(1),\bY(1))
\\%\notag\\&
<
 \DCb(\bX(0),\bY(0))
=\DCb(X',Y')
=
\DCb(\Pi\bX,\Pi\bY).
\end{multline}
%Thus, $  \DCb(\Pi\bX,\Pi\bY) >\DCb(\bX,\bY)$.
Thus, an orthogonal projection might increase distance covariance.

It can obviously also decrease it; for example the projection onto the
second coordinate above yields $(X'',Y'')$ with $\DCb(X'',Y'')=0$.
\end{example}

\section{Hilbert spaces and $\gb=2$}\label{S2}

We continue to assume that $\cX$ and $\cY$ are Hilbert spaces; we now
consider the case $\gb=2$. Note that \refD{D5} does not apply (it requires
$\gb<2$, see \refR{R2divH}), so we return to the general \refDs{D1}--\ref{D3}.

In this case, \eqref{hxb} yields, by expanding all $\norm{\bX_i-\bX_j}^2$,
\begin{align}\label{perseus}
  \hXii
&=-2\innprod{\bX_1,\bX_2}+2\innprod{\bX_2,\bX_3}-2\innprod{\bX_3,\bX_4}
+2\innprod{\bX_4,\bX_1}
\notag\\&
=2\innprod{\bX_1-\bX_3,\bX_4-\bX_2}.
\end{align}
Assume, as in \refDs{D2} and \ref{D3}, that $\E\norm{\bX}^2<\infty$.
Then $\E\bX$ exists, in Bochner sense (see \refApp{ABP}),
and  \eqref{txb1} together with \eqref{perseus} yield
\begin{align}\label{antigone}
  \tXii
&=\E\bigpar{\hXii\mid \bX_1,\bX_2}
%\notag\\&
=-2\innprod{\bX_1-\E\bX,\bX_2-\E \bX}.
\end{align}
We thus see directly that \eqref{l1b} and \eqref{l1c} hold, and thus
$\hXii,\tXii\in L^2$ if $\E\norm{\bX}^2<\infty$, as asserted by \refL{L2}.

In particular, in the 1-dimensional case $\cX=\bbR$,
\begin{align}\label{dim1}
\hXii=2(\bX_1-\bX_3)(\bX_4-\bX_2),
&&&  
\tXii=-2(\bX_1-\E\bX)(\bX_2-\E\bX),
\end{align}
with the latter assuming $\E|\bX|^2<\infty$.
Consequently, 
if $\cX=\cY=\bbR$ and $\E|\bX|^2,\E|\bY|^2<\infty$, 
then \refD{D3} yields, using \eqref{dim1} and independence, 
\begin{align}\label{dc1}
  \DCii(\bX,\bY)=\E\bigsqpar{\tX_2\tY_2}=4{\Cov(\bX,\bY)^2},
\end{align}
as noted by \citet{SRB2007}.
(\refDs{D1}--\ref{D2} agree by \refTs{T1}.)
This extends to higher dimensional Euclidean spaces and, more generally,
Hilbert spaces as follows.
Let $\HX\tensor\HY$ denote the Hilbert space tensor product of $\HX$ and
$\HY$, see \eg{} \cite[Appendix E]{SJIII};
recall that this is a Hilbert space such that there is a bilinear map 
$\tensor:\HX\times\HY\to\HX\tensor\HY$ with 
\begin{align}
  \label{tensor}
\innprod{\bx_1\tensor\by_1 ,\bx_2\tensor\by_2}_{\HX\tensor\HY}
=
\innprod{\bx_1 ,\bx_2}_{\HX}\innprod{\by_1 ,\by_2}_{\HY};
\end{align}
furthermore, if $\set{\be_i}_i$ and $\set{\bex_j}_j$ are ON-bases in $\HX$ and
$\HY$, then $\set{\be_i\tensor\bex_j}_{i,j}$ is an ON-basis in
$\HX\tensor\HY$.
(Note that the mapping $\tensor$ is neither injective nor surjective, but
the set of finite linear combinations $\sum_i \bx_i\tensor\by_j$ is dense in
$\HX\tensor\HY$.) 
Hence, %Note that
$\bX\tensor\bY$ is a random variable in $\HX\tensor\HY$ with
$\norm{\bX\tensor\bY}=\norm{\bX}\,\norm{\bY}$.

\begin{theorem}\label{TH2}
  Let $\cX=\HX$ and $\cY=\HY$ be separable Hilbert spaces, 
and assume $\E\norm{\bX}^2<\infty$
and $\E\norm{\bY}^2<\infty$.
Let $(\be_i)_i$ and $(\bex_j)_j$ be ON-bases in $\HX$ and $\HY$. Then,
\begin{align}
  \DCii(\bX,\bY)
&=4\sum_{i,j}\Cov\bigpar{\innprod{\bX,\be_i},\innprod{\bY,\bex_j}}^2
\label{th2a}\\
&=4\bignorm{\E(\bX\tensor\bY)-\E\bX\tensor\E\bY}^2_{\HX\tensor\HY}
\label{th2b}
\end{align}
\end{theorem}

\begin{proof}
Since $\DCb(\bX,\bY)$ and the expressions in \eqref{th2a}--\eqref{th2b} are
invariant under (deterministic) shifts of $\bX$ and $\bY$, we may for
convenience assume $\E\bX=\E\bY=0$.
Then,   by \eqref{antigone},
  \begin{align}\label{hektor}
\tXii\tYii   &
=4\innprod{\bX_1,\bX_2}\innprod{\bY_1,\bY_2}
= 4\sum_{i,j}\innprod{\bX_1,\be_i}\innprod{\bX_2,\be_i}
\innprod{\bY_1,\bex_j}\innprod{\bY_2,\bex_j}.
  \end{align}
We have, by the \CSineq,
\begin{align}
\sum_{i}\bigabs{\innprod{\bX_1,\be_i}\innprod{\bX_2,\be_i}}
\le
\Bigpar{\sum_{i}\innprod{\bX_1,\be_i}^2}\qq
\Bigpar{\sum_{i}\innprod{\bX_2,\be_i}^2}\qq
=\norm{\bX_1}\,\norm{\bX_2}
\end{align}
and thus, by independence and the \CSineq{} again,
\begin{align}\label{akilles}
\E
\sum_{i,j}\bigabs{\innprod{\bX_1,\be_i}\innprod{\bX_2,\be_i}
\innprod{\bY_1,\bex_j}\innprod{\bY_2,\bex_j}}
\le
\E \bigsqpar{\norm{\bX_1}\,\norm{\bX_2}\,\norm{\bY_1}\,\norm{\bY_2}}
\notag\\
=  \bigpar{\E\bigsqpar{\norm{\bX_1}\,\norm{\bY_1}}}^2
\le
 \E \norm{\bX_1}^2\E\norm{\bY_1}^2<\infty.
\end{align}
Hence, \eqref{hektor} yields by Fubini's theorem, justified by \eqref{akilles},
  \begin{align}\label{helena}
\E\bigsqpar{\tXii\tYii}   &
= 4\sum_{i,j}\E\bigsqpar{\innprod{\bX_1,\be_i}\innprod{\bX_2,\be_i}
\innprod{\bY_1,\bex_j}\innprod{\bY_2,\bex_j}}
\notag\\&
=4\sum_{i,j}\bigpar{\E\bigsqpar{\innprod{\bX,\be_i}\innprod{\bY,\bex_j}}}^2
  \end{align}
which yields \eqref{th2a}.

Moreover, $\set{\be_i\tensor\bex_j}_{i,j}$ is an ON-basis in 
$\HX\tensor\HY$, and thus
\begin{align}
\bignorm{\E(\bX\tensor\bY)}^2  
&=
\sum_{i,j}\innprod{\E(\bX\tensor\bY),\be_i\tensor\bex_j}^2
%\notag\\&
=
\sum_{i,j}\bigpar{\E\innprod{\bX\tensor\bY,\be_i\tensor\bex_j}}^2
\notag\\&
=
\sum_{i,j}\bigpar{\E\bigsqpar{\innprod{\bX,\be_i}\innprod{\bY,\bex_j}}}^2
\end{align}
which together with \eqref{helena} yields \eqref{th2b}.
\end{proof}

\begin{corollary}  \label{CH2}
  Let $\cX=\HX$ and $\cY=\HY$ be separable Hilbert spaces, 
and assume $\E\norm{\bX}^2<\infty$
and $\E\norm{\bY}^2<\infty$.
Then, the following are equivalent:
\begin{romenumerate}
  
\item \label{CH2a}
$\DCii(\bX,\bY)=0$.
\item \label{CH2b}
$\Cov\bigpar{\innprod{\bX,\bx},\innprod{\bY,\by}}=0$
for every $\bx\in\HX$, $\by\in\HY$.
\item \label{CH2c}
$
\E(\bX\tensor\bY)-\E\bX\tensor\E\bY=0$.
\end{romenumerate}
\end{corollary}
\begin{proof}
For \ref{CH2a}$\implies$\ref{CH2b}, and $\bx,\by\neq0$,
choose ON-bases such that
$\be_1=\bx/\norm{\bx}$ and $\bex_1=\by/\norm{\by}$.
The rest is immediate from \refT{TH2}.
\end{proof}

\citet{SRB2007} observed that for $\gb=2$ and real-valued variables,
$\DCii(\bX,\bY)=0$ does not characterize independence but 
instead that $\bX$ and $\bY$ are uncorrelated;
\refC{CH2} extends this to Hilbert spaces, 
in the sense \ref{CH2b} or \ref{CH2c} above.

\begin{remark}\label{RSJ271}
  $\E(\bX\tensor\bY)-\E\bX\tensor\E\bY
%=\E\bigsqpar{(\bX-\E\bX)\tensor(\bY-\E\bY)}
\in\HX\tensor\HY$ can be regarded as
  the covariance of the vector-valued variables $\bX$ and $\bY$;
\cf{} the general theory of higher moments of Banach space valued variables
in \cite{SJ271}, where the moment lives in a suitable tensor product.
(The general theory in \cite{SJ271} focusses on a single variable and
on the projective and injective
tensor products, but see \cite[Remarks 3.24 and 3.25]{SJ271}.
Since we assume separable spaces and $\E\norm{\bX}^2,\E\norm{\bY}^2<\infty$,
there are no problems with integrability; \cf{} \cite[Theorem 5.14]{SJ271}.)
\end{remark}

\begin{remark}\label{Rembedded}
Let $\cX$ and $\cY$ both be metric spaces such that $d^\gb$ is a semimetric of
negative type.
Then, see \refR{Rsemimetric}, 
there are embeddings  $\gf:\cX\to\HX$  and $\gf':\cY\to\HY$ 
into Hilbert spaces such that
\begin{align}\label{--}
  d_\cX(\bx_1,\bx_2)^\gb=\norm{\gf(\bx_1)-\gf(\bx_2)}^2,
&&&
  d_\cY(\by_1,\by_2)^\gb=\norm{\gf'(\by_1)-\gf'(\by_2)}^2.
\end{align}
It follows immediately that, for any of  Definitions \ref{D1}--\ref{D3},
\begin{align}
  \label{embedded}
\DCb(\bX,\bY)=\DCii\bigpar{\gf(\bX),\gf'(\bY)}.
\end{align}
Hence, $\DCb(\bX,\bY)$ can be interpreted as in
\refT{TH2} for the embedded variables, 
as shown (for $\gb=1$) in \cite[Proposition 3.7]{Lyons}.
\end{remark}

\begin{remark}\label{RHS2}
  The Hilbert space tensor product $\HX\tensor\HY$ can be identified with the
space of Hilbert--Schmidt operators $\HX\to\HY$
(see \eqref{HS} and the proof of \refL{LHTH}); then 
$\E(\bX\tensor\bY)-\E\bX\tensor\E\bY
=\E\sqpar{(\bX-\E\bX)\tensor(\bY-\E\bY)}$ %\in\HX\tensor\HY$ 
corresponds to the operator $\bx\mapsto \E[\innprod{\bx,\bX-\E\bX}(\bY-\E\bY)]$,
known as the \emph{covariance operator}
(or \emph{cross-covariance operator} \cite{Baker}).
Thus \refT{TH2} says that
$\DCii(\bX,\bY)$ is 4 times the squared Hilbert--Schmidt norm of the
covariance operator. 

More generally,
if $\cX$ and $\cY$ both are metric spaces such that $d^\gb$ is a semimetric of
negative type, then 
\eqref{embedded}
%\refR{Rembedded} 
%\eqref{embedded} holds
shows that 
 $\DCb(\bX,\bY)$ equals $\DCii\xpar{\gf(\bX),\gf'(\bY)}$ 
for some embeddings $\gf:\cX\to\HX$
and $\gf':\cY\to\HY$ into Hilbert spaces.
Hence,  
%$\DCb(\bX,\bY)$ can be interpreted as in
$\DCb(\bX,\bY)$ equals 4 times the squared
Hilbert--Schmidt norm of the  covariance operator
corresponding to the embedded variables,
as
shown in
\cite[Theorem 24]{Sejdinovic2013};
this Hilbert--Schmidt norm (or its square)
is called the \emph{Hilbert--Schmidt independence criterion (HSIC)}
\cite{Gretton}, \cite[\S 3.3]{Sejdinovic2013}; 
\cf{} \refR{RHS}.  
\end{remark}

\begin{remark}
  If $\gb$ is an even integer larger than 2, we can similarly express $\DCb$
  in moments of $\bX$ and $\bY$, but the resulting formulas are complicated
  and do not seem to be of any interest.
For example, for $\gb=4$, for $\cX=\cY=\bbR$, and taking 
for simplicity $\bX=\bY$ with
$\E\bX=0$, 
  \begin{align}
\DC_4(\bX,\bX)=
    32 \E[ \bX^2] \E[ \bX^6] - 96 \E[ \bX^3]\E[ \bX^5]+ 68 (\E[ \bX^4])^2
\qquad\notag\\
-72 (\E[ \bX^2])^2\E[ \bX^4]
+ 64 \E[ \bX^2](\E[ \bX^3])^2 + 36(\E[ \bX^2])^4.
  \end{align}
% = 64 for \pm1. In fact, 2^{2\gb-2} for general \gb>0.
% = 912 for standard normal
We do not know any application or interesting properties of $\DCb$ with
$\gb>2$.
\end{remark}

\section{Optimality of moment conditions}\label{Sopt}

We have so far assumed the moment conditions stated in \refDs{D1}--\ref{D3};
these seem natural and convenient for applications. Nevertheless, it is of
interest to study whether they really are required for the definitions, and
what happens when we try to extend one of
the definitions to cases when the moment condition fails.
\refDs{D4} and \ref{D5} are stated without moment conditions, but we
similarly can ask when the results are finite and whether they agree with the other
definitions.

In this section,  we will give examples showing that the moment conditions in
\refDs{D1}--\ref{D3}  are optimal in general,
in the sense  that if we reduce the exponent in the moment condition,
then there exist  counterexamples where the definition
either yields an infinite value or is meaningless.
%(For this it suffices to consider the case $\bY=\bX$.)
On the other hand, there are also cases where the moment conditions do not
hold but the definitions yield a finite value.
%For \refDs{D2} and \ref{D3}, 
%there are cases when the definitions yield well-defined infinite
%values.
We explore these possibilities in the next section,
but our results are incomplete, and we leave
a number of (explicit or implicit) open problems.

In general, if we try to define $\DCb(\bX,\bY)$ by \eqref{d1} or
\eqref{d2} for some  $\bX$ and $\bY$, there are three possibilities:
%, for a given $(\bX,\bY)$.
\begin{dcenumerate}
\item \label{DC1}
The expression yields a finite value;
this may then be taken to be $\DCb(\bX,\bY)$.
This happens when all expectations 
in \eqref{d1} or \eqref{d2}, respectively, 
are finite.
(For \eqref{d1}, it also includes the trivial case when $\bX$ or $\bY$ is
degenerate, so $d(\bX_i,\bX_j)=0$ \as{} or  $d(\bY_i,\bY_j)=0$ \as{};
then all terms in \eqref{d1} are 0, if necessary interpreting
$0\cdot\infty=0$.)  
\item \label{DC2}
The expression makes sense as either $+\infty$ or $-\infty$.
We may then take it as defining
$\DCb(\bX,\bY)$, now with an infinite value in $\set{-\infty,\infty}$.
(We do not know whether $-\infty$ can happen, see \refP{P-}.)
Thus, at least one expectation is infinite.
Furthermore,
for \eqref{d1}, where all expectations are of non-negative variables and
thus defined in $[0,\infty]$, this means that either the two first 
expectations are finite, 
or the third expectation is;
for \eqref{d2} this means that 
one of $\E\bigsqpar{(\hXb\hYb)_+}$
and $\E\bigsqpar{(\hXb\hYb)_-}$ is finite and the other infinite, so
the expectation 
$\E\bigsqpar{\hXb\hYb}$
is defined as $+\infty$ or
$-\infty$. 

\item \label{DC3}
The expression
\eqref{d1} or \eqref{d2} is of the type $\infty-\infty$.
Then it is meaningless, and $\DCb(\bX,\bY)$ is undefined (by this definition).
\end{dcenumerate}

For \refD{D3}, we have the same possibilities as for \refD{D2}, but also the
complication that $\tXb$ and $\tYb$ have to be defined, see
\eqref{txb1}--\eqref{txb2}. We thus have another bad case:
\begin{dcenumerateq}
\item \label{DC4}
$\tXb$ or $\tYb$ is not defined.
Then $\tDCb(\bX,\bY)$ is undefined.
\end{dcenumerateq}

For Euclidean spaces, we also have \refD{D4},
and for Hilbert spaces we have \refD{D5}.
Since \eqref{d4} and \eqref{d5a}--\eqref{d5c} are
integrals of  non-negative functions, 
\refDs{D4} and \ref{D5} are always meaningful, but may
yield $+\infty$. In other words, we have only the cases \ref{DC1} and
\ref{DC2}.
Again we may ask when the definition yields a finite value, and when it
agrees with other definitions; 
in particular whether the moment conditions in \refT{T52} are
best possible.

The moment conditions assumed in \refDs{D1}--\ref{D3} guarantee, as seen in
\refT{T1}, that the good case \ref{DC1} occurs.
In the following subsections we investigate more generally when the cases
\ref{DC1}--\ref{DC4} occur, and whether the different definitions still
agree when more than one of them applies.

\subsection{Optimality in \refD{D1}}
We begin with  \refD{D1}, where we have a simple necessary and sufficient
condition.

\begin{theorem}\label{TD1?}
  \begin{thmenumerate}
  \item   \label{TD1?<}
If\/ 
$\E\norm{\bX}^\gb+\E\norm{\bY}^\gb+\E\sqpar{\norm{\bX}^\gb\norm{\bY}^\gb}<\infty$,
then all expectations in \eqref{d1} are finite, so \eqref{d1} defines
$\xDCb(\bX,\bY)$ as a finite number.

Moreover, in this case also the definitions \eqref{d2} and \eqref{d3} 
yield the same result,
\ie,
$\xDCb(\bX,\bY) = \hDCb(\bX,\bY)= \tDCb(\bX,\bY)$.

\item \label{TD1?oo}
Conversely, if\/ 
$\E\norm{\bX}^\gb+\E\norm{\bY}^\gb+\E\sqpar{\norm{\bX}^\gb\norm{\bY}^\gb}=\infty$,
and $\bX$ and $\bY$ are non-degenerate,
then \eqref{d1} is of the type $\infty-\infty$ and thus meaningless.    
  \end{thmenumerate}
\end{theorem}

In particular, Case \ref{DC2}, \ie, a well-defined infinite value of $\xDCb$,
never occurs for \refD{D1}. 

\begin{proof}
  \pfitemref{TD1?<}
This follows by minor modifications of the argument used 
under slightly stronger assumptions in
\refS{S:intro} and \refL{L0}. Note that
the assumption implies that $\E[\norm{\bX_i}^\gb\norm{\bY_j}^\gb]<\infty$
for all $i$ and $j$, and
thus it follows from the triangle inequality
\eqref{dij} that all expectations in \eqref{d1} are finite.
Moreover, the assumption implies, using \eqref{dij} again, 
that $\hXb,\hYb\in L^1$, and thus $\tXb$ and
$\tYb$ are defined by \eqref{txb1}--\eqref{txb2},
and also that 
$\hXb\hYb\in L^1$ and 
$\tXb\tYb\in L^1$.
%\eqref{d2} and \eqref{d3} are finite.
We omit the details.

\pfitemref{TD1?oo}
If $\E\norm{\bX}^\gb=\infty$, then
$\E d(\bx,\bX_2)^\gb=\infty$ for any $\bx$, and thus, by first conditioning
on $(\bX_1,\bY_1)$ and $(\bX_3,\bY_3)$ and integrating over $\bX_2$ only, 
both 
$\E\bigsqpar{d(\bX_1,\bX_2)^\gb}=\infty$
and
$\E\bigsqpar{d(\bX_1,\bX_2)^\gb d(\bY_1,\bY_3)^\gb}=\infty$;
hence,
since $\E\bigsqpar{d(\bY_1,\bY_2)^\gb}>0$, we see that \eqref{d1} is of 
the type $\infty-\infty$.

By symmetry, the same holds if $\E\norm{\bY}^\gb=\infty$.

Finally, suppose that $\E[\norm{\bX}^\ga\norm{\bY}^\ga]=\infty$.
By the cases just treated, we may assume that also
$\E\norm{\bX}^\gb<\infty$  and $\E\norm{\bY}^\gb<\infty$.
Then, using the triangle inequality and integrating only over the event 
$\set{\norm{\bX_2},\norm{\bY_2},\norm{\bY_3}\le M}$, for an $M$ so large that
  this event has positive probability, we see that both the first and last
  expectations in \eqref{d1} are $\infty$, and thus \eqref{d1} is
  $\infty-\infty$. 
\end{proof}

\begin{remark}
  If \refT{TD1?}\ref{TD1?<} applies and $\EDCb(\bX,\bY)$ or
  $\HDCb(\bX,\bY)$ is defined, \ie, if $\ga<2$ and the spaces are Euclidean
  spaces or Hilbert spaces, respectively, then it too equals $\xDCb(\bX,\bY)$.
This follows by \refT{TD1?} together with \refTs{TEH} and \ref{T52}.
\end{remark}

\begin{example}\label{ED10}
If $\bX$ and $\bY$ are
independent with
$\E \norm{\bX}^{\gb}<\infty$ and $\E \norm{\bY}^{\gb}<\infty$, then
\refT{TD1?}\ref{TD1?<} applies and
\eqref{d1} makes perfect sense;
\refDs{D1}--\ref{D3} all can be used, and all yield 0.
\end{example}

\begin{example}\label{ED1}
  Let $\bX$ be arbitrary with $\E\norm{\bX}^{2\gb}=\infty$, and let
  $\bY=\bX$.
Then, \refT{TD1?}\ref{TD1?oo} shows that $\xDCb(\bX,\bX)$
is of the type $\infty-\infty$ and does not make sense.
Consequently, in general, the moment condition in \refD{D1} is necessary.
(In particular, for every $(\bX,\bY)$ with $\bY=\bX$.)
\end{example}

\subsection{Optimality in \refD{D2}}

We have already seen in \refE{ED1} that the moment condition in \refD{D1} is
necessary, in a strong sense. We next show that the moment conditions in
\refDs{D2} and \ref{D3} also are optimal, in the
sense that if we reduce the
exponent, there are counterexamples. However, there are also examples
where these definitions yield finite values although the moment condition fails.

Consider first \refD{D2}.
$\hXb$ and $\hYb$ are always defined by \eqref{hxb}, so the question is
whether
$\E[\hXb\hYb]$ exists or not, and whether its value is finite of not.
Note, in particular, that $\hDCb(\bX,\bX):=\tfrac14\E\hXb^2$ always is
defined, although it may be $+\infty$; we have
\begin{align}\label{nils}
\hDCb(\bX,\bX)<\infty\iff \hXb\in L^2.  
\end{align}
Note also that, by rotational symmetry in the indices in \eqref{hxb}, 
$\hXb$ has a symmetric distribution.
Thus $\E\hXb=0$ whenever the expectation exists.

\begin{example}\label{Egb}
  Let  $\cX=\cY=\bbR$, and 
suppose
that $\bX\ge0$ with $\P(\bX=0)>0$.
On the event $\bX_3=\bX_4=0$,
we have
\begin{align}
-\hXb=\bX_1^\gb+\bX_2^\gb-|\bX_1-\bX_2|^\gb
\ge \minx{\bX_1^{\gb}}{\bX_2^\gb}.
\label{trias}
\end{align}
Hence, if $\E|\hXb|^2<\infty$,
then 
\begin{align}
  \infty 
&>
\E\bigsqpar{\bigpar{\minx{\bX_1^{\gb}}{\bX_2^\gb}}^2}
=
\E\bigsqpar{\minx{\bX_1^{2\gb}}{\bX_2^{2\gb}}}
\notag\\&
=
\intoo \P\bigsqpar{\minx{\bX_1^{2\gb}}{\bX_2^{2\gb}}>t}\dd t
=
\intoo \P\bigsqpar{\bX^{2\gb}>t}^2\dd t
\notag\\&
=
2\gb\intoo \P\bigsqpar{\bX>x}^2 x^{2\gb-1} \dd x.
\label{gemini}
\end{align}
If we choose $\bX$ such that, for $x\ge2$ say,
\begin{align}\label{twin}
\P(\bX>x)=x^{-\gb},  
\end{align}
then $\E|\bX|^\gam<\infty$ for every $\gam<\gb$, but the integral in
\eqref{gemini} diverges and thus $\E|\hXb|^2=\infty$;
hence \eqref{d2} yields $\hDCb(\bX,\bX)=\infty$ by \eqref{nils}.
(Case \ref{DC2}.)
Consequently, when $\gb\le2$,
the exponent $\gbx=\gb$ is optimal in \refD{D2}
(in order to yield a finite value).
\end{example}

\begin{example}\label{E2gb-2}
  Let $\gb>1$ and $\cX=\cY=\bbR$, and 
suppose, for simplicity, that $\bX\ge0$ with $\P(\bX=0)=\P(\bX=1)=1/4$.
On the event $\bX_2=\bX_3=0$, $\bX_4=1$,
we have
for some $c>0$, assuming that $\bX_1\ge2$, say,
\begin{align}
  \hXb=\bX_1^\gb-|\bX_1-1|^\gb+1
\ge c\bX_1^{\gb-1}.
\end{align}
Hence, for these values of $\bX_2,\bX_3,\bX_4$, we have
$\hXb\ge c\bX_1^{\gb-1}-C$. Consequently,
$\E|\hXb|^2<\infty \implies \E\bX^{2(\gb-1)}<\infty$.

We can choose $\bX$ as above such that
$\E \bX^\gam<\infty$ 
for every $\gam<2\gb-2$, 
but $\E \bX^{2\gb-2}=\infty$ and consequently
$\E|\hXb|^2=\infty$;
thus, \eqref{d2}  yields $\DCb(\bX,\bX)=\infty$.
(Case \ref{DC2}.)
Hence, when $\gb\ge2$, the exponent $\gbx=2\gb-2$ is optimal in \refD{D2}.
\end{example}

\begin{example}\label{ED2bad}
We have here given examples with $\E|\hXb|^2=\infty$, so that \eqref{d2}
gives $\hDCb(\bX,\bX)=+\infty$.

Similarly, \eqref{trias} and a calculation as in \eqref{gemini} show that
if, say, 
%\eqref{twin} is repalced by
$\P(\bX>x)=x^{-\gb/2}$ for $x>2$, then  $\E|\hXb|=\infty$.
Since $\hXb$ has a symmetric distribution by \eqref{d2}, it follows that if 
$\bY$ is any non-degenerate random variable such that
$\bX$ and $\bY$ are independent, then the expectation in \eqref{d2} is of
the type $\infty-\infty$ and thus undefined
(Case  \ref{DC3} above); hence \refD{D2} cannot be
applied at all (even allowing $\pm\infty$ as a result).
\end{example}

\begin{example}\label{ERH2}
Let $\cX=\cY=\bbR$ and
consider the special (and rather exceptional) case $\gb=2$, \cf{} \refS{S2}.
Then $\hXii$ is given by \eqref{dim1}, and
it follows easily that
%In particular, for  $\cX=\bbR$,
\begin{align}\label{b2b}
\hXii\in L^2 %\iff \tXii\in L^2 
\iff \E|\bX|^2<\infty, 
\end{align}
and that for $\bX=\bY$, \eqref{dc1} holds in the form
\begin{align}\label{dc1var}
\hDCii(\bX,\bX)=\tfrac14\E[\hXii^2] 
%= \E[\tXii^2]
%=4\bigpar{\E|\bX-\E\bX|^2}^2
=4 \bigpar{\Var \bX}^2
\end{align}
for any $\bX$, where the expressions all are infinite when $\E|\bX|^2=\infty$.
This shows again that the condition of finite $\gbx$ moment in \refD{D2} 
%and \ref{D3} is 
cannot be improved when $\gb=2$,
if we want
$\DCii(\bX,\bX)$ to be finite.
Furthermore, if we take $\bY=\zeta\bX$ where
$\bX$ and $\zeta$ are independent with
$\E\bX=0$, $\E[\bX^2]=\infty$, 
$\zeta\in\set{\pm1}$ and $\E\zeta=0$,
then $\E[\hXii\hYii]$ %and $\E[\tXii\tYii]$ are both 
is of the type
$\infty-\infty$;
hence,
even allowing infinite values,
$\hDCii(\bX,\bY)$ cannot be defined by \refD{D2} %or \ref{D3} 
without
assuming second moments.
%See also, more generally, \refEs{Egb}--\ref{ED3bad}.
\end{example}

\subsection{Optimality in \refD{D3}}

We now turn to %$\tDCb$. 
\refD{D3}.
As noted above, $\tXb$ is only defined for some $\bX$. If we use the
conditional expectation definition in \eqref{txb1}, then we have to require
$\hXb\in L^1$, \ie, $\E|\hXb|<\infty$.
On the other hand, the explicit formula \eqref{txb2} makes sense only if 
$\E d(\bX_1,\bX_2)^\gb<\infty$, or equivalently $\E\norm{\bX}^\gb<\infty$,
since otherwise also the conditional expectations in \eqref{txb2} are
$+\infty$ \as, and thus \eqref{txb2} is $\infty-\infty$.
Moreover, if $\E\norm{\bX}^\gb<\infty$, then $\hXb\in L^1$ by \eqref{hxb},
and \eqref{txb1} agrees with \eqref{txb2}. Hence we may take \eqref{txb1}
as the primary definition of $\tXb$, and say that $\tXb$ is defined when
$\hXb\in L^1$. This holds  in particular when $\E\norm{\bX}^\gb<\infty$,
and then \eqref{txb2} holds too, but note that \refL{L1} shows that 
$\E\norm{\bX}^{\gbx/2}<\infty$ suffices for $\hXb\in L^1$.

Hence, $\tXb$ is defined if and only if $\hXb\in L^1$,
and then $\tXb\in L^1$;
furthermore
\begin{align}\label{mikael}
  \E \tXb = \E \E\bigpar{\hXb\mid\bX_1,\bX_2}=\E\hXb=0.
\end{align}
Moreover, in this case, also
\begin{align}\label{gabriel}
  \E \bigpar{\tXb\mid \bX_1} 
= \E \bigpar{\E\bigpar{\hXb\mid\bX_1,\bX_2}\mid\bX_1}
=\E\bigpar{\hXb\mid \bX_1}=0,
\end{align}
since $\hXb$ has a symmetric distribution also when conditioned on $\bX_1$, 
by symmetry in \eqref{hxb}.

\begin{example}
  \label{ED3bad}
Recall that \refE{ED2bad} gives an example where $\hXb\notin L^1$;
hence, $\tXb$ is not defined and thus $\tDCb(\bX,\bX)$ is 
undefined. %(Case \ref{DC4}.)
\end{example}

We note a general result relating $\tXb$ and $\hXb$.
By \eqref{txb1}, $\tXb$ is (\as)\ a function of $\bX_1$ and $\bX_2$;
let us (temporarily) write $\tXb$ as $\tXb(\bX_1,\bX_2)$, so that we can
substitute other $\bX_i$ as arguments.
The following lemma shows that $\hXb$ can be recovered from $\tXb$.
\begin{lemma}\label{Lht}
  Suppose that $\hXb\in L^1$. Then, \as,
  \begin{align}\label{lht}
    \hXb 
= \tXb(\bX_1,\bX_2) - \tXb(\bX_2,\bX_3) + \tXb(\bX_3,\bX_4) - \tXb(\bX_4,\bX_1).
  \end{align}
Consequently, for any $p\ge1$,
\begin{align}\label{lhtp}
\text{$\tXb$ exists and } \tXb\in L^p  
\iff
\hXb\in L^p.
\end{align}
\end{lemma}
\begin{proof}
  If $\E\norm{\bX}^\gb<\infty$, this is obvious from \eqref{txb2} and
  cancellations. In general, we use truncations.
Let, for $M>0$,
$\IM:=\indic{|\bX|\le M}$, 
$\IM_i:=\indic{|\bX_i|\le M}$, and let $p_M:=\E\IM=P\bigpar{|\bX_i|\le M}$.
Then,
\begin{multline}
  \begin{aligned}
&     \E\bigpar{\IM_3\IM_4\hXb\mid \bX_1,\bX_2}
\\&\qquad
=p_M^2 d(\bX_1,\bX_2)^\gb
-p_M\E_{\bX}\bigpar{\IM d(\bX_2,\bX)^\gb}
  \end{aligned}
\\
+ p_M^2 \E\bigpar{\IM_3\IM_4 d(\bX_3,\bX_4)^\gb}
-p_M\E_{\bX}\bigpar{\IM d(\bX_1,\bX)^\gb}
\end{multline}
and consequently, by rotational symmetry and cancellations,
interpreting all indices \MOD4,
\begin{align}\label{selma}
 \sumiv (-1)^{i-1} \E\bigpar{\IM_{i+2}\IM_{i+3}\hXb\mid \bX_i,\bX_{i+1}}
%\\&
&=p_M^2 \sumiv (-1)^{i-1} d(\bX_i,\bX_{i+1})^\gb
\notag\\&
= p_M^2 \hXb.  
\end{align}
Since we assume $\hXb\in L^1$, we have $\IM_{i+2}\IM_{i+3}\hXb\lto\hXb$ as \Mtoo,
and thus 
\begin{align}\label{temla}
  \E\bigpar{\IM_{i+2}\IM_{i+3}\hXb\mid \bX_i,\bX_{i+1}}
\lto
  \E\bigpar{\hXb\mid \bX_i,\bX_{i+1}}
=\tXb(\bX_i,\bX_{i+1}).
\end{align}
Hence, as \Mtoo, the \lhs{} of \eqref{selma} converges in $L^1$ to the
\rhs{} of \eqref{lht}, while the \rhs{} of \eqref{selma} obviously converges
to $\hXb$. Hence, \eqref{lht} follows.

Finally, \eqref{lhtp} is an immediate consequence of \eqref{lht} and
\eqref{txb1}.
\end{proof}

In particular, this leads to the following for the case $\bY=\bX$.
Note that
$\tDCb(\bX,\bX):=\E\tXb^2$ is defined whenever $\tXb$ is, 
although it may be $+\infty$; 
\cf{}  $\hDCb(\bX,\bX)$ discussed above.

\begin{theorem}\label{TXX}
Let $\bX$ be a random variable in a metric space.
Then the following are equivalent:
  \begin{romenumerate}
  \item \label{TXXhDC}
$\hDCb(\bX,\bX)<\infty$.
  \item \label{TXXtDC}
$\tDCb(\bX,\bX)<\infty$ (which includes that $\tXb$ is defined).
\item \label{TXXh}
$\hXb\in L^2$.
\item \label{TXXt}
$\tXb$ is defined and $\tXb\in L^2$.
  \end{romenumerate}
Furthermore, if these hold, then
$\hDCb(\bX,\bX)=\tDCb(\bX,\bX)$.
\end{theorem}
\begin{proof}
$\ref{TXXhDC}\iff\ref{TXXh}$:
This follows directly from the definition \eqref{d2}, as noted in
\eqref{nils}.  

$\ref{TXXtDC}\iff\ref{TXXt}$:
Follows similarly  from the definition \eqref{d3}.
 
$\ref{TXXh}\iff\ref{TXXt}$:
By \refL{Lht}.

Finally, suppose that \ref{TXXhDC}--\ref{TXXt} hold.
Use \eqref{lht} and expand $(\hXb)^2$ as a sum of
products. Since $\tXb\in L^2$, each product is in $L^1$,
so we may take their expectations separately. 
Furthermore,
\eqref{gabriel} implies that all off-diagonal terms such as 
$\E\sqpar{\tXb(\bX_1,\bX_2)\tXb(\bX_2,\bX_3)}=0$, and we obtain
\begin{align}
  \E\bigsqpar{\hXb^2}
=\sumiv \E\bigsqpar{\tXb(\bX_i,\bX_{i+1})^2}
=4  \E\bigsqpar{\tXb^2}.
\end{align}
Hence, $\hDCb(\bX,\bX)=  \frac{1}4\E\bigsqpar{\hXb^2}
=  \E\bigsqpar{\tXb^2}=\tDCb(\bX,\bX)$.
\end{proof}

\begin{corollary}  \label{Cmag}
  \begin{thmenumerate}
  \item     
If $\hXb\in L^1$, so $\tXb$ is defined, then
$\hDCb(\bX,\bX)=\tDCb(\bX,\bX)$ (finite or infinite).
\item 
If $\hXb\notin L^1$, then
$\hDCb(\bX,\bX)=\infty$ and
$\tDCb(\bX,\bX)$ is undefined.
  \end{thmenumerate}
\end{corollary}
\begin{proof}
Follows from \refT{TXX}, considering the three cases
$\hXb\in L^2$, $\hXb\in L^1\setminus L^2$ and $\hXb\notin L^1$
separately.
\end{proof}
If we only care about finite values and regard $\infty$ as 'undefined', we
thus see that $\tDCb(\bX,\bX)=\hDCb(\bX,\bX)$ for all $\bX$.

\begin{example}
  \label{Egbt}
Let $\gb\le2$ and let $\bX\in\bbR$ be as in \refE{Egb}; thus 
$\bX\ge0$ and \eqref{twin} holds. 
We have $\E|\bX|^\gam<\infty$ for every $\gam<\gb$, and in particular
$\E\abs{\bX}^{\gb/2}<\infty$; hence
\refL{L1}\ref{L1b} implies that $\hXb\in L^1$.
Thus, $\tXb$ exists, but $\hXb\notin L^2$ by \refE{Egb};
hence \refT{TXX} shows that $\tDCb(\bX,\bX)=\infty$.
Consequently, the exponent $\gbx=\gb$ is optimal in \refD{D3} when $\gb\le2$.
\end{example}

\begin{example}
  \label{Egb-2t}
Similarly,
let $\gb>2$ and let $\bX\in\bbR$ be as in \refE{E2gb-2}.
Then, 
$\E\abs{\bX}^\gam<\infty$ for every $\gam<2\gb-2$, and
in particular
$\E\abs{\bX}^{\gb-1}<\infty$; hence
\refL{L1}\ref{L1c} implies that $\hXb\in L^1$.
Thus, $\tXb$ exists, but $\hXb\notin L^2$ by \refE{E2gb-2};
hence \refT{TXX} shows that $\tDCb(\bX,\bX)=\infty$.
Consequently, the exponent $\gbx=2\gb-2$ is optimal in \refD{D3} when $\gb>2$.
\end{example}

Hence, the exponent $\gbx$ is optimal in \refD{D3} too.

\begin{example}\label{ED3bad3}
  Let $\gb=2$ and $\cX=\cY=\bbR$ as in \refE{ERH2}, and assume that
  $\E\bX=0$.
Then, by \eqref{dim1}, $\tXii$ exists and
$\tXii=-2\bX_1\bX_2$. 
Hence, we find directly the same conclusions 
for $\tDCb$ 
as found for $\hDCb$ in \refE{ERH2}.
In particular,
with $\bX$ and $\bY=\zeta\bX$ 
as in the final part of \refE{ERH2}, 
$\E\bigsqpar{\tXii\tYii}$ is of the type $\infty-\infty$ and thus
undefined. (Case \ref{DC3}.)
\end{example}

\subsection{Optimality for $\EDCb$ and $\HDCb$}

\refDs{D4} and \ref{D5} do not require any moment conditions;
if
$\cX$ and $\cY$ are Euclidean spaces or Hilbert spaces, respectively, then
$\EDCb(\bX,\bY)$ and $\HDCb(\bX,\bY)$ are always defined, but may be
$+\infty$.
(Recall also that \refT{TEH} shows that
for spaces where both are defined, we always have
$\EDCb(\bX,\bY)=\HDCb(\bX,\bY)$, finite or not.)
\refT{T52} shows that the moment condition
$\E\norm{\bX}^\gb,\E\norm{\bY}^\gb<\infty$ is sufficient to 
guarantee that 
$\EDCb(\bX,\bY)=\HDCb(\bX,\bY)$ is finite.
(Recall that this is the same moment condition as in \refDs{D2} and
\ref{D3}.)
The following example shows that the exponent $\gb$
in this moment condition is
optimal, even for random variables in $\bbR$.

\begin{example}\label{ED4}
  Let $0<\gb<2$, and let $\bX$ be a symmetric stable random variable in
  $\bbR$ with the \chf{} $\gf_\bX(t)=e^{-|t|^\gb}$.
Then $\E\abs{\bX}^\gb=\infty$, but $\E|\bX|^\gam<\infty$ for every
$\gam<\gb$.

Take $\bY=\bX$. Then, for $0\le t\le 1$ and $t\le u\le 2t$,
\begin{align}
  \gf_{\bX,\bX}(t,-u)-\gf_\bX(t)\gf_\bX(-u)
&=e^{-|t-u|^\gb}-e^{-|t|^\gb-|u|^\gb}
\notag\\&
\ge e^{-t^\gb}-e^{-2t^\gb}
\ge ct^\gb,
\end{align}
for some $c>0$.
Consequently, \eqref{d4} yields, 
changing the sign of $u$,
\begin{align}
  \EDCb(\bX,\bX)
\ge c\int_{t=0}^1 \int_{u=t}^{2t} t^{2\gb}\frac{\dd t\dd u}{t^{1+\gb}u^{1+\gb}}
=
c\int_{t=0}^1 \frac{t^{2\gb}}{t^{1+2\gb}}{\dd t}
=\infty.
\end{align}
Hence, using \refT{TEH},
 $\HDCb(\bX,\bX)=  \EDCb(\bX,\bX)=\infty$.
The condition in \refT{T52} 
on finite $\gb$ moments thus cannot be replaced by any lower moments
in order to guarantee finite values.
\end{example}

\section{Beyond the moment conditions}\label{Sbeyond}

We continue to investigate cases when the moment condition in
\refDs{D2}--\ref{D3} fails; now with the aim of obtaining positive results.

\subsection{A weaker condition}
We begin with $\hDCb$ in \refD{D2}, and
show first that
the counterexample in \refE{Egb} is optimal, at least when $\gb\le1$.
\begin{theorem}\label{TL}
Let $\cX$ be any separable metric space, and
  let $0<\gb\le1$.
If\/
\begin{align}\label{gemini2}
\intoo \P\bigsqpar{\norm{\bX}>x}^2 x^{2\gb-1} \dd x<\infty,  
\end{align}
then $\E|\hXb|^2<\infty$ and thus $\hDCb(\bX,\bX)<\infty$.
\end{theorem}
\begin{proof}
  The calculation in \eqref{gemini} shows that \eqref{gemini2} is equivalent
  to
  \begin{align}\label{gemm}
\E\bigsqpar{\bigpar{\minx{\norm{\bX_1}^{\gb}}{\norm{\bX_2}^\gb}}^2}<\infty.    
  \end{align}
In other words, $\norm{\bX_1}^\gb\land\norm{\bX_2}^\gb\in L^2$.
Hence, \refL{L1}\ref{L1a} shows that 
$\hXb\in L^2$.
%$\E|\hXb|^2<\infty$.
\end{proof}

\begin{remark}\label{RL1}
Let $0<\gb\le1$. Then,
\eg{} using \refL{Lvaa} below,
the argument in \refE{Egb} is easily extended to show that if $\cX=\bbR$, 
then \eqref{gemini2} is also necessary for $\E|\hXb|^2<\infty$.
Thus, at least for $\gb\le1$ and $\cX=\bbR$, \eqref{gemini2} is both
necessary and sufficient for $\hDCb(\bX,\bX)<\infty$. %$\hXb\in L^2$.
\end{remark}

\begin{remark}\label{RL2}
It is easy to see directly that
the condition \eqref{gemini2} follows from the condition
$\E\norm{\bX}^\gb<\infty$ in \refL{L2}. (We omit the details.)
Furthermore, \eqref{gemini2} is a strictly weaker condition, and thus,
 for $\gb\le1$,
\refT{TL} is stronger than \refL{L2}.
For example,
    if we instead of \eqref{twin} choose, for $x\ge e$,
  \begin{align}
    \P(\bX>x)=x^{-\gb}/\log x,
  \end{align}
then $\E\bX^\gb=\infty$,
but the integral in \eqref{gemini} converges and 
\refT{TL} shows that $\E|\hXb|^2<\infty$
and $\hDCb(\bX,\bX)<\infty$.
\end{remark}

Hence, although we have seen that
the exponent in the moment condition in \refD{D2} is 
best possible, \refT{TL} shows that 
for $\gb\le1$,
the moment condition can be  weakened 
to the condition \eqref{gemini2} (together with the same for $\bY$);
we postpone the details to \refT{TL2}, where we also extend it to
 $\EDCb$ and $\HDCb$.

Before proceeding,
we  note that
when $\gb\le1$, 
we may simplify the condition $\hXb\in L^2$ by the following lemma.

\begin{lemma}\label{Lvaa}
Let $p>0$. 
If $0<\gb\le1$,
then
  \begin{align}\label{lvaa}
    \hXb\in L^p\iff \norm{\bX_1}^\gb+\norm{\bX_2}^\gb-d(\bX_1,\bX_2)^\gb\in L^p.
  \end{align}
\end{lemma}

Note that (for $\gb\le1$) the \rhs{} is non-negative by the triangle
inequality.

\begin{proof}
\pfitemx{$\implies$}  
Since $\E|\hXb|^p<\infty$, the conditional expectation 
$\E\bigpar{|\hXb|^p\mid X_3,X_4}\in L^1$.
Hence, there exist some $\bx_3$ and $\bx_4$ such that
$\E\bigpar{|\hXb|^p\mid X_3=\bx_3,X_4=\bx_4}\in L^1$, which by the
definition \eqref{hxb} means
\begin{align}
  d(\bX_1,\bX_2)^\gb -   d(\bX_1,\bx_4)^\gb -   d(\bX_2,\bx_3)^\gb\in L^p.
\end{align}
The triangle inequality yields, for $j=3,4$,
\begin{align}
  \bigabs{d(\bX,\bx_j)-\norm{\bX}}\le d(\bx_j,\ox)=O(1),
\end{align}
and thus, since $\gb\le1$,
\begin{align}
  \bigabs{d(\bX,\bx_j)^\gb-\norm{\bX}^\gb}
=O(1),
%=O\bigpar{\norm{\bX}^{\max\set{\gb-1,0}}},
\end{align}
and the result follows.

\pfitemx{$\impliedby$}
Immediate (for any $\gb$), since the definition
\eqref{hxb} can be written
\begin{align}
  \hXb=
\sumiv(-1)^i\bigpar{\norm{\bX_i}^\gb+\norm{\bX_{i+1}}^\gb-d(\bX_i,\bX_{i+1})^\gb}.
\end{align}
\end{proof}

\begin{remark}
  We do not know (even for $\cX=\bbR$) 
whether \refL{Lvaa} holds also for $\gb>1$, and leave that  as an open problem.
(It holds, by a minor modification of the proof above, 
for $\gb>1$ under the additional assumption
$\E\norm{\bX}^{p(\gb-1)}<\infty$, but that seems less useful.)
\end{remark}

We next introduce a class of function spaces.

\subsection{Lorentz spaces}
The condition \eqref{gemini2} can be expressed as follows using
\emph{Lorentz spaces}, a generalization of the Lebesgue spaces $L^p$;
see \eg{} \cite{BS,BL}.
Let $\Xx$ be the \emph{decreasing rearrangement} of $\norm{\bX}$; this is the
(weakly) decreasing function $(0,1)\to\ooo$  defined by
\begin{align}
  \Xx(t):=\inf\bigset{x:\P(\norm{\bX}>x)\le t}.
\end{align}
In probabilistic terms, $\Xx$ is characterized as
the decreasing function on $(0,1)$ that, 
regarded as a random variable when $(0,1)$ is equipped with the Lebesgue
measure, has the same distribution as $\norm{\bX}$.

For a given probability space $(\gO,\cF,P)$, and $p,q\in(0,\infty)$,
the \emph{Lorentz} space $L^{p,q}(\gO,\cF,P)$ 
is defined as the linear space of all
real-valued random variables $\bX$ such that
\begin{align}
  \intoi \bigpar{t^{1/p}\Xx(t)}^q\frac{\ddx t}{t}<\infty.
\end{align}
It is well-known 
that $L^{p,p}=L^p$, and that if $q_1<q_2$ then $L^{p,q_1}\subset L^{p,q_2}$,
with strict inequality provided the probability space is large enough.

A standard Fubini argument shows that
%$\Xx(t)>x \iff \P(\norm{\bX}>x)>t$
\begin{align}
    \intoi \bigpar{t^{1/p}\Xx(t)}^q\frac{\ddx t}{t}
&=
q\intoi  \intoo \indic{\Xx(t)>x} x^{q-1} t^{q/p-1}\dd x\dd t
\notag\\&
=
q\intoi  \intoo \bigindic{\P(\norm{\bX}>x)>t} x^{q-1} t^{q/p-1}\dd x\dd t
\notag\\&
=
p  \intoo \P\bigsqpar{\norm{\bX}>x}^{q/p} x^{q-1}\dd x.
\end{align}
In particular, taking $p=\gb$ and $q=2\gb$, we see that
\eqref{gemini2} is equivalent to $\norm{\bX}\in L^{\gb,2\gb}$.

Consequently, 
for $\gb\le1$, 
\refT{TL} says that if 
$\norm{\bX}\in L^{\gb,2\gb}$, then $\hXb\in L^2$,
which weakens the condition
$\norm{\bX}\in L^{\gb}$ in \refL{L2} to 
$L^{\gb,2\gb}$.
Hence, we can extend the use of \refD{D2}; moreover, 
as shown below, also
\refDs{D3}, \ref{D4} and \ref{D5} yield the same result in this case.

\begin{theorem}\label{TL2}
Let $0<\gb\le1$, and assume that\/ $\norm{\bX},\norm{\bY}\in L^{\gb,2\gb}$.
Then:
\begin{romenumerate}
\item \label{TL2h}
\refD{D2} yields a finite value $\hDCb(\bX,\bY)$.  
\item \label{TL2t}
\refD{D3} yields a finite value $\tDCb(\bX,\bY)$,
and $\tDCb(\bX,\bY)=\hDCb(\bX,\bY)$.
\item \label{TL2E}
If $\cX$ and $\cY$ are Euclidean spaces, then
\refD{D4} yields a finite value $\EDCb(\bX,\bY)$,
and $\EDCb(\bX,\bY)=\hDCb(\bX,\bY)$.
\item \label{TL2H}
If $\cX$ and $\cY$ are Hilbert spaces, then
\refD{D5} yields a finite value $\HDCb(\bX,\bY)$,
and $\HDCb(\bX,\bY)=\hDCb(\bX,\bY)$.
\end{romenumerate}
Thus, the values that are defined are all equal (and finite).
\end{theorem}

The proof is postponed to the following subsections.

Note also that by \refR{RL1}, the Lorentz space $L^{\gb,2\gb}$ is optimal in
the strong sense that, for $\gb\le1$ and $\cX=\bbR$,
\begin{align}\label{ql}
\hDCb(\bX,\bX)<\infty
\iff
  \hXb\in L^2 \iff \norm{\bX}\in L^{\gb,2\gb}.
\end{align}

The proofs of the results above assume $\ga\le1$. 
We leave the case $\ga>1$ as open problems.
For example:

\begin{problem}\label{P>1}
  For $\gb>1$, what is the optimal Lorentz space condition that guarantees
  $\E|\hXb|^2<\infty$ and thus $\hDCb(\bX,\bX)<\infty$?
\end{problem}

By \refT{TXX}, the answer for $\tDCb(\bX,\bX)$ is the same.

\begin{remark}
  \label{R2b}
\refE{ERH2} shows that in the special case $\gb=2$, 
the condition $\norm{\bX}\in L^2$ in \refD{D2} cannot be improved; 
it is actually
necessary for $\hXb\in L^2$ and $\hDCii(\bX,\bX)<\infty$
in the case $\cX=\bbR$.
Hence, for $\gb=2$, the answer to \refP{P>1}  is
$L^2=L^{2,2}$.

A naive interpolation with \eqref{ql} yields the conjecture that for
$1<\gb<2$, the answer  is $L^{\gb,2}$. 
\end{remark}

\begin{remark}\label{Rql}
  The equivalence \eqref{ql} does not hold for all metric
spaces $\cX$, not even for $\gb=1$. For a counterexample, let $\cX=\ell^1$
with the standard  basis $(\be_n)_1^\infty$, let $0<\gam\le1/2$, and let $N$ be
an integer-valued 
  random variable with 
$\P(N=n)=p_n:=cn^{-1-\gam}$, $n\ge1$, 
where $c$ %$c=\zeta(1+\gam)\qw$ 
is a normalization constant.
Finally, let $\bX:=N\qq\be_N$. It is easily seen that, with $\bX_i$ defined
in the same way by $N_i$,
\begin{align}\label{goa}
 \hX\le 2\sumiv N_i\qq\indic{N_i=N_{i+1}}, 
\end{align}
and thus, using \CS's (or Minkowski's) inequality,
\begin{align}\label{gob}
  \E\hX^2 \le C \E \bigsqpar{N_1\indic{N_1=N_2}}
= C\sumn np_n^2
= C\sumn n^{-1-2\gam}
<\infty,
\end{align}
while for $x\ge1$,
\begin{align}\label{goc}
  \P\bigpar{\norm{\bX}>x}=\P(N>x^2)
=\sum_{n>x^2}cn^{-1-\gam}
\ge cx^{-2\gam}\ge cx\qw,
\end{align}
so \eqref{gemini2} fails, and thus $\norm{\bX}\notin L^{1,2}$.
%$\E\norm{\bX}^\gam=\sumn n^{\gam/2}p_n=\infty$, so $\norm{\bX}\notin L^{\gam}$,
\end{remark}

\begin{problem}
  Does the equivalence \eqref{ql} hold in Euclidean spaces?
In infinite-dimensional Hilbert spaces?
\end{problem}

We have not investigated whether the results on continuity and consistency
in \refS{Scon} can be extended (for $\gb\le1$) by replacing the moment
conditions with the corresponding Lorentz space condition. In particular:

\begin{problem}\label{PLcons}
Let $\gb\le1$.
 Does \refT{Tcons} hold if the moment condition is replaced by $\bX,\bY\in
 L^{\gb,2\gb}$? 
\end{problem}

\subsection{More on $\hDCb$ and $\tDCb$}

\refT{TXX} considers only the case $\bX=\bY$.
We do not know whether it extends to $\DCb(\bX,\bY)$ in general, without
further conditions. We give a partial result.

\begin{theorem}\label{TD3a}
  Suppose that $\hXb,\hYb\in L^1$, so $\tXb$ and $\tYb$ exist.
Suppose further that
$\tXb\tYb\in L^1$ and 
$\tXb(\bX_1,\bX_2)\tYb(\bY_1,\bY_3)\in L^1$.
Then $\hXb\hYb\in L^1$, and
$\hDCb(\bX,\bY)=\tDCb(\bX,\bY)$; futhermore, this value is finite.
%$\tfrac14\E[\hXb\hYb]=\E[\tXb\tYb]$, so the definitions \eqref{d2} and
%\eqref{d3} yield the same finite value of $\DCb(\bX,\bY)$.

In particular, this holds if $\hXb,\hYb\in L^2$.
\end{theorem}
\begin{proof}
This is similar to the proof of \refT{TXX}.
  We have $\tXb,\tYb\in L^1$ by \eqref{txb1}, and thus 
$\tXb(\bX_1,\bX_2)\tYb(\bY_3,\bY_4)\in L^1$ by independence.
Express $\hXb$ and $\hYb$ by \eqref{lht} and expand 
$\hXb\hYb$ as a sum of 16 terms. 
By the assumptions (and symmetry), every term is in $L^1$, so we
may take their expectations separately. 
Furthermore,
\eqref{gabriel} implies that \eg{}
$\E\sqpar{\tXb(\bX_1,\bX_2)\tYb(\bY_1,\bY_3)}=0$, and we obtain
\begin{align}
  \E\bigsqpar{\hXb\hYb}
=\sumiv \E\bigsqpar{\tXb(\bX_i,\bX_{i+1})\tYb(\bY_i,\bY_{i+1})}
=4  \E\bigsqpar{\tXb\tYb}.
\end{align}

If $\hXb,\hYb\in L^2$, then $\tXb,\tYb\in L^2$ and the assumptions above
follow by the \CSineq.
\end{proof}

\begin{proof}[Proof of \refT{TL2}\ref{TL2h}\ref{TL2t}]
By the comments before \refT{TL2}, the assumptions imply $\hXb,\hYb\in L^2$,
and thus \refT{TD3a} shows \ref{TL2h} and \ref{TL2t}.
\end{proof}

\begin{problem}\label{P=}  
Let either $\cX$ and $\cY$ be arbitrary, or consider only $\cX=\cY=\bbR$.
\begin{romenumerate}
  
\item 
Is it true for arbitrary random $\bX\in \cX$ and $\bY\in\cY$
that 
$\hDCb(\bX,\bY)$ is defined and finite 
$\iff$
$\tDCb(\bX,\bY)$ is defined and finite?

\item 
If this holds, is furthermore always
$\hDCb(\bX,\bY)=\tDCb(\bX,\bY)$?
\end{romenumerate}
\end{problem}

\subsection{More on $\EDCb$ and $\HDCb$}
 Consider now the case of Euclidean or, more generally, Hilbert spaces and 
\refDs{D4} and \ref{D5}.
We complete the proof of \refT{TL2}; recall that this assumes $\gb\le1$.

\begin{proof}[Proof of \refT{TL2}\ref{TL2E}\ref{TL2H}]
\pfitemref{TL2H}
This follows by essentially the same proof as for \refT{T52}.
  As noted in \refR{RD5M}, \eqref{jcg} holds without any moment condition.
Moreover, as said in the proof of \refT{T52}, \refL{L1} holds for $\hXbM$
defined in \eqref{jfb2}
too, uniformly in $M$; we now use \refL{L1}\ref{L1a},
and denote the \rhs{} by 
$\Xxxl$.
Hence, $|\hXbM|\le\Xxxl$, and similarly
$|\hYbM|\le\Yxxl$.

As noted above, %in the proof of \ref{TL}, 
$X\in\Lgbgb$ is equivalent to \eqref{gemini2} and to \eqref{gemm}.
Consequently, $\Xxxl\in L^2$ and, similarly, $\Yxxl\in L^2$.
Hence, $\Xxxl\Yxxl\in L^1$ and dominated convergence applies to \eqref{jcg},
just as in the proof of \refT{T52}, yielding 
$\HDCb(\bX,\bY)=\hDCb(\bX,\bY)<\infty$.

\pfitemref{TL2E}
\refT{TEH} shows the general equality
$\EDCb(\bX,\bY)=\HDCb(\bX,\bY)$, and thus
\ref{TL2E} follows from \ref{TL2H}.

This completes the proof of \refT{TL2}.
\end{proof}

\begin{problem}\label{P>1H}
  For $1<\gb<2$, what is the optimal Lorentz space condition that guarantees
$\HDCb(\bX,\bX)<\infty$ (for variables in a Hilbert space)?
Does this also imply $\HDCb(\bX,\bX)=\hDCb(\bX,\bX)$?
Does this condition imply $\HDCb(\bX,\bY)=\hDCb(\bX,\bY)$ for two variables
$\bX$ and $\bY$?
\end{problem}

\begin{problem}\label{P=H}  
Let either $\cX$ and $\cY$ be arbitrary Hilbert spaces, 
or consider only $\cX=\cY=\bbR$. Let $0<\ga<2$.
\begin{romenumerate}  
\item 
Is it true for arbitrary random $\bX\in \cX$ and $\bY\in\cY$
that 
$\hDCb(\bX,\bY)$ is defined and finite 
$\iff$
$\HDCb(\bX,\bY)$ is finite?

\item 
If this holds, is furthermore always
$\hDCb(\bX,\bY)=\HDCb(\bX,\bY)$?
\end{romenumerate}
\end{problem}

\subsection{Hilbert spaces, $\gb=2$}\label{SSH2}

Consider now the case when $\cX=\HX$ and $\cY=\HY$ are Hilbert spaces, 
as in the preceding subsection, but
take $\gb=2$.
Then $\HDCb(\bX,\bY)$ is not defined, so we consider
$\hDCii(\bX,\bY)$ and $\tDCii(\bX,\bY)$.
In \refS{S2}, we did this assuming second moments; we now remove that
assumption and generalise the results.
(This is partly for its own sake, but mainly for the application in the
next subsection.)

In this subsection,
expectations $\E\bX$ of Hilbert space valued random variables are always
interpreted in Pettis sense, see \refApp{ABP}. (This is sometimes said
explicitly for emphasis.) We use some technical results stated and proved in
\refApp{ABP}.

Recall that 
$\hXii=2\innprod{\bX_1-\bX_3,\bX_4-\bX_2}$
by
\eqref{perseus}, for any $\bX$.
We next
show that \eqref{antigone} holds under weaker conditions than assumed in
\refS{S2}.
\begin{lemma}\label{Lanti}
  Let $\cX=\cH$ be a Hilbert space.
If $\tXii$ exists, then
$\E\bX$ exists in Pettis sense, and
\begin{align}\label{antigon}
      \tXii
=-2\innprod{\bX_1-\E\bX,\bX_2-\E\bX}.
\end{align}
\end{lemma}

\begin{proof}
By \eqref{perseus}, we have $\hXii=2\innprod{\bZ,\bZ'}$ where $\bZ:=\bX_1-\bX_3$
  and $\bZ':=\bX_4-\bX_2$. 
Assume that $\tXii$ exists, which by our definition means that
$\E|\hXii|<\infty$. 
Thus $\E|\innprod{\bZ,\bZ'}|<\infty$.
%Since $\bZ'\eqd \bZ$, 
\refL{LZZ}\ref{LZZP} applies and shows
that
$\bZ=\bX_1-\bX_3$ is Pettis integrable.
Hence, for every $\bx\in\HX$, 
\begin{align}
 \innprod{\bX_1,\bx}- \innprod{\bX_3,\bx}
= \innprod{\bX_1-\bX_3,\bx}
=\innprod{\bZ,\bx}\in L^1 .
\end{align}
Since $ \innprod{\bX_1,\bx}$ and $ \innprod{\bX_3,\bx}$ are independent
random variables, this implies $\E|\innprod{\bX_1,\bx}|<\infty$, for any
$\bx\in\cH$, and thus $\E\bX_1$ exists in Pettis sense.

Using \eqref{perseus}, we may now integrate over first $\bX_4$ and then
$\bX_3$ and obtain
\begin{gather}
\E\bigpar{\hXii\mid \bX_1,\bX_2,\bX_3}
=2\innprod{\bX_1-\bX_3,\E\bX-\bX_2},
%
%\end{align}
\\
%\begin{align}
    \tXii
=\E\bigpar{\hXii\mid \bX_1,\bX_2}
=2\innprod{\bX_1-\E\bX,\E\bX-\bX_2}.
%\notag\\
%&\phantom{\tXii}
%=-2\innprod{\bX_1-\E\bX,\bX_2-\E\bX},
\end{gather}
showing \eqref{antigon}.
\end{proof}

\begin{theorem}\label{TH4}
  Let $\cX=\HX$ and $\cY=\HY$ be Hilbert spaces.

  \begin{romenumerate}
  \item\label{TH4h1}  
If\/ $\hDCii(\bX,\bY)$ is defined, \ie, 
$\E\bigsqpar{\hXii\hYii}$ is defined as an extended real number,
then $\hDCii(\bX,\bY)\in\ooox$.
\item \label{TH4h2}
If\/  $\hDCii(\bX,\bY)<\infty$, then
%=4\norm{\E\sqpar{\bX-\E\bX)\tensor(\bY-\E\bY}}
\begin{align}\label{th4h2}
  \hDCii(\bX,\bY)
&=\bignorm{\E\bigsqpar{(\bX_1-\bX_2)\tensor(\bY_1-\bY_2)}}^2_{\HX\tensor\HY},
\end{align}
where the expectation exists in Pettis sense.
  \item \label{TH4t1}
If\/ $\tDCii(\bX,\bY)$ is defined, \ie, $\tXii$ and $\tYii$ are defined and
$\E\bigsqpar{\tXii\tYii}$ is defined as an extended real number,
then $\tDCii(\bX,\bY)\in\ooox$.
\item \label{TH4t2}
If furthermore $\tDCii(\bX,\bY)<\infty$, then
%=4\norm{\E(\bX-\E\bX)\tensor(\bY-\E\bY)}
\begin{align}\label{th4t2}
  \tDCii(\bX,\bY)
&=4\bignorm{\E[\bX\tensor\bY]-\E\bX\tensor\E\bY}^2_{\HX\tensor\HY},
\end{align}
where the expectations exist in Pettis sense.
\item \label{TH4th}
If\/ $\hDCii(\bX,\bY)$ and $\tDCii(\bX,\bY)$ both are defined, 
as in \ref{TH4h1} and \ref{TH4t1}, and furthermore
$\hDCii(\bX,\bY)$ and $  \tDCii(\bX,\bY)$ both are finite, then
\begin{align}\label{th4th}
\hDCii(\bX,\bY)=\tDCii(\bX,\bY).  
\end{align}
  \end{romenumerate}
\end{theorem}

\begin{proof}
\pfitemx{\ref{TH4h1},\ref{TH4h2}}
By \eqref{perseus} and \eqref{tensor},
\begin{align}
  \hXii\hYii
&=4\innprod{\bX_1-\bX_3,\bX_4-\bX_2}\innprod{\bY_1-\bY_3,\bY_4-\bY_2}
\notag\\&
=4\biginnprod{(\bX_1-\bX_3)\tensor(\bY_1-\bY_3), 
  (\bX_4-\bX_2)\tensor(\bY_4-\bY_2)}
_{\HX\tensor\HY}.
\end{align}
This is an example of $\innprod{\bZ,\bZ'}$ as in \refL{LZZ}, with 
$\bZ:=(\bX_1-\bX_3)\tensor(\bY_1-\bY_3)
\eqd(\bX_1-\bX_2)\tensor(\bY_1-\bY_2)$.
Thus,
\ref{TH4h1} follows from \refL{LZZ}\ref{LZZ+}, and
\ref{TH4h2} from \refL{LZZ}\ref{LZZP}.

\pfitemx{\ref{TH4t1},\ref{TH4t2}}
Similarly, if $\tXii$ and $\tYii$ exist, then 
\refL{Lanti} 
shows that $\E\bX$ and $\E\bY$ exist in Pettis sense, 
and furthermore, using \eqref{tensor},
\begin{align}\label{luth}
\tXii\tYii&=4\innprod{\bX_1-\E\bX,\bX_2-\E\bX}\innprod{\bY_1-\E\bY,\bY_2-\E\bY}
\notag\\&
=4\innprod{(\bX_1-\E\bX)\tensor(\bY_1-\E\bY),(\bX_2-\E\bX)\tensor(\bY_2-\E\bY)}.
\end{align}
This is another  example of $\innprod{\bZ,\bZ'}$ as in \refL{LZZ}, now with 
$\bZ:=(\bX_1-\E\bX)\tensor(\bY_1-\E\bY)$.
Thus,
\ref{TH4t1} follows from \refL{LZZ}\ref{LZZ+}.

Finally, assume $\tDCii(\bX,\bY)<\infty$, \ie, $\tXii\tYii\in L^1$.
Then \eqref{luth} and \refL{LZZ}\ref{LZZP}  show that
$\E \bZ$ exists in Pettis sense, and that
\begin{align}\label{yen}
  \tDCii(\bX,\bY)=\E[\tXii\tYii]
=4\norm{\E\bZ}^2
=4\norm{\E\bigsqpar{(\bX-\E\bX)\tensor(\bY-\E\bY)}}^2.
\end{align}
We have
\begin{align}\label{koala}
\bX_1\tensor \bY_1 = \bZ+
(\E \bX)\tensor(\bY_1- \E\bY)
+\bX_1\tensor \E\bY.
\end{align}
Furthermore, since $\E\bX$ and $\E\bY$ are constant
vectors, it is easy to see that
$\E [\bX_1\tensor \E\bY] = \E \bX\tensor \E\bY$ and
$\E [(\E \bX)\tensor(\bY_1- \E\bY)] = \E \bX\tensor \E[\bY_1-\E\bY]=0$.
%with both existing in Pettis sense.
(This also follows from the more general \refL{LHTH}.)
Hence, \eqref{koala} shows that $\E[\bX\tensor\bY]$ exists, 
%in Pettis sense,
and 
\begin{align}
\E\sqpar{\bX\tensor \bY} 
= \E\sqpar{\bX_1\tensor \bY_1} 
=  \E \bZ +\E\bX\tensor \E\bY.
\end{align}
Thus \eqref{th4t2} follows from \eqref{yen}.

\pfitemref{TH4th}
In this case, \ref{TH4h1}--\ref{TH4t2} all hold.
By \ref{TH4t2}, the expectations $\E\bX$, $\E\bY$ and $\E[\bX\tensor\bY]$ 
exist. Hence,
$\E[\bX_i\tensor\bY_i]=\E[\bX\tensor\bY]$ exists for every $i$.
Furthermore, if $i\neq j$ so $\bX_i$ and $\bY_j$ are independent, 
$\E[\bX_i\tensor\bY_j]$ exists by \refL{LHTH} and equals 
$\E[\bX_i]\tensor\E[\bY_j]$.
Hence, $\E[\bX_i\tensor\bY_j]$ exists for every $i$ and $j$, and thus
\begin{align}
&\E\bigsqpar{(\bX_1-\bX_2)\tensor(\bY_1-\bY_2)}
\notag\\&\hskip4em
=\E[\bX_1\tensor\bY_1] + \E[\bX_2\tensor\bY_2]
-\E[\bX_1\tensor\bY_2]-\E[\bX_2\tensor\bY_1]
\notag\\&\hskip4em
=2\bigpar{\E[\bX\tensor\bY] - \E[\bX]\tensor\E[\bY]}.
\end{align}
Consequently, \eqref{th4th} follows from \eqref{th4h2} and \eqref{th4t2}.
\end{proof}

\subsection{Metric spaces of negative type}
In this subsection %(and only here)
we assume that $\cX$ and $\cY$ are metric spaces such
that  $d_\cX^\gb$ and $d_\cY^\gb$ both are of negative type,
see \refR{Rsemimetric}.
We then can embed the spaces into Hilbert spaces as in
\refR{Rembedded} and transfer the results in \refSS{SSH2}.

\begin{theorem}\label{TNT}
Let $\gb>0$ and let $\cX$ and $\cY$ be metric spaces such
that  $d_\cX^\gb$ and $d_\cY^\gb$ both are of negative type.
  \begin{romenumerate}
  \item\label{TNTh1}  
If\/ $\hDCb(\bX,\bY)$ is defined, \ie, 
$\E\bigsqpar{\hXb\hYb}$ is defined as an extended real number,
then $\hDCb(\bX,\bY)\in\ooox$.
  \item \label{TNTt1}
If\/ $\tDCb(\bX,\bY)$ is defined, \ie, $\tXb$ and $\tYb$ are defined and
$\E\bigsqpar{\tXb\tYb}$ is defined as an extended real number,
then $\tDCb(\bX,\bY)\in\ooox$.
\item \label{TNTth}
If\/ $\hDCb(\bX,\bY)$ and $\tDCb(\bX,\bY)$ both are defined, 
as in \ref{TNTh1} and \ref{TNTt1}, and furthermore
%$\hDCb(\bX,\bY)$ and $  \tDCb(\bX,\bY)$ 
both are finite, then
\begin{align}\label{tnt4th}
\hDCb(\bX,\bY)=\tDCb(\bX,\bY).  
\end{align}
  \end{romenumerate}
\end{theorem}
\begin{proof}
Immediate by \refR{Rembedded} and \refT{TH4}\ref{TH4h1}\ref{TH4t1}\ref{TH4th}.
\end{proof}

This gives a partial (but not complete) answer to \refP{P=} for spaces with
$d^\ga$ of negative type; 
recall from \refR{Rsemimetric} that when $0<\ga\le2$, this includes
Hilbert spaces, in particular $\bbR$. %\cite{Schoenberg1938}.

\begin{remark}
  If $d$ is a metric of negative type, then so is $d^\gb $ for every
  $\gb\le1$.
Hence, if $\cX$ and $\cY$ are metric spaces of negative type, then \refT{TNT}
applies at least with $0<\ga\le1$.
\end{remark}

\subsection{Negative values?}

If $\cX$ and $\cY$ are metric spaces such that $d^\gb$ is of negative type,
then \refT{TNT} shows that $\hDCb(\bX,\bY)$  and $\tDCb(\bX,\bY)$ 
may not be negative and finite, nor $-\infty$. 
%(They may be finite nonnegative, $+\infty$, or undefined.)
\refT{TD1?} then shows the same for $\xDCb(\bX,\bY)$.
%(And more: it may be finite nonnegative or undefined, only.)
The same is also, trivially, true for $\EDCb(\bX,\bY)$ and $\HDCb(\bX,\bY)$
when they are applicable. 
%(They may be finite nonnegative or $+\infty$.)
More precisely, we have the possibilities shown in \refTab{tab:neg},
by  
\refTs{T1}, \ref{T52}, \ref{TD1?} and \ref{TNT};
\refEs{Egb}, \ref{ED2bad},  \ref{ED3bad}, \ref{Egbt}, \ref{ED4};
\eqref{d4} and \eqref{d5a}.
\begin{table}[h!]
  \centering
\begin{tabular}[h]{c|c|c|c|c|c}
  & $\ooo$ & $+\infty$ & $(-\infty,0)$ & $ -\infty$ & undefined
\\\hline
$\xDCb(\bX,\bY)$ & \ja & \nej & \nej & \nej & \ja \\
$\hDCb(\bX,\bY)$ & \ja & \ja & \nej & \nej & \ja \\
$\tDCb(\bX,\bY)$ & \ja & \ja & \nej & \nej & \ja \\
$\EDCb(\bX,\bY)$ & \ja & \ja & \nej & \nej & \nej \\
$\HDCb(\bX,\bY)$ & \ja & \ja & \nej & \nej & \nej \\
\hline
\end{tabular}
  \caption{Possibilities when $d^\gb$ is of negative type}
  \label{tab:neg}
\end{table}

Conversely, 
if  $\cX$ or $\cY$ is a  metric space that is not  of negative type
then $\DC(\bX,\bY)<0$ is possible
(as soon as  both spaces have at least two points),
see \cite[Proposition 3.15]{Lyons};
by \refT{T1}, this holds 
for any of $\xDCb(\bX,\bY)$, $\hDCb(\bX,\bY)$, $\tDCb(\bX,\bY)$.
%\refDs{D1}, \ref{D2}, \ref{D3}.
\refT{TD1?} still rules out
$\pm\infty$ for $\xDCb(\bX,\bY)$,
and we find 
the possibilities shown in \refTab{tab:ejneg}.
\begin{table}[h!]
  \centering
\begin{tabular}[h]{c|c|c|c|c|c}
  & $\ooo$ & $+\infty$ & $(-\infty,0)$ & $ -\infty$ & undefined
\\\hline
$\xDCb(\bX,\bY)$ & \ja & \nej & \ja & \nej & \ja \\
$\hDCb(\bX,\bY)$ & \ja & \ja & \ja & ? & \ja \\
$\tDCb(\bX,\bY)$ & \ja & \ja & \ja & ? & \ja \\
\hline
\end{tabular}
  \caption{Possibilities when $d^\gb$ is not of negative type}
  \label{tab:ejneg}
\end{table}

%However,
For $\hDCb(\bX,\bY)$ and $\tDCb(\bX,\bY)$, %\refDs{D2} and \ref{D3},
we do not know whether 
$-\infty$ is possible
(in Case \ref{DC2} in \refS{Sopt}):
%and we state this as a problem.

\begin{problem}
  \label{P-}
Is $\hDCb(\bX,\bY)=-\infty$ or $\tDCb(\bX,\bY)=-\infty$ possible?
\end{problem}

\section*{Acknowledgement}
This work  was inspired by a lecture by Thomas Mikosch at a
mini-mini-workshop in Gothenburg in April, 2019. 
I thank Thomas Mikosch for helpful comments.

\appendix

\section{A uniform integrability lemma}\label{Aui}

We use above some well-known standard results on uniform integrability, see
\eg{} \cite[\S5.4 and \S5.5]{Gut}. We use also the following simple result which
perhaps is less well-known; since we have not found a good reference, we
provide a proof for completeness.

In this section, all random variables are real-valued.
We state the lemmas below for sequences of random variables (the case that we
use), but they are valid (with the same proofs)
for  families $(X_\gaxx)_{\gaxx\in\cAxx}$ and $(Y_\gaxx)_{\gaxx\in\cAxx}$ with an
arbitrary index set $\cAxx$.

\begin{lemma}\label{LA1}
  Let $(X_n)_n$ and $(Y_n)_n$ be uniformly integrable sequences of random
  variables, and suppose that for each $n$, $X_n$ and $Y_n$ are independent.
Then the sequence $(X_nY_n)_n$ is also uniformly integrable.
\end{lemma}

To prove this, we use another simple result
that perhaps is less well-known than it deserves.

\begin{lemma}\label{LA2}
  Let $(X_n)_{n}$ be a sequence of  random variables.
Then $(X_n)_{n}$ is uniformly integrable if and only if for every
$\eps>0$
there exists $K_\eps<\infty$ and a sequence
$(X_n^\eps)_{n}$ of random variables such that for every $n$,
\begin{align}
&  |X_n^\eps|\le K_\eps \quad\textas,\label{la2a}
\\
&\E \bigabs{X_n-X_n^\eps} <\eps.\label{la2b}
%  Let $(X_\ga)_{\ga\in\cA}$ be a family of  random variables.
%Then $(X_\ga)_{\ga\in\cA}$ is uniformly integrable if and only if for every
%$\eps>0$
%there exists $K_\eps<\infty$ and a family 
%$(X_\ga^\eps)_{\ga\in\cA}$ of random variables such that for every $\ga\in\cA$,
%\begin{align}
%  |X_\ga^\eps|&\le K_\eps \quad\textas,
%\\
%\E \bigabs{X_\ga-X_\ga^\eps} &<\eps.
\end{align}
\end{lemma}
\begin{proof}
  This is a simple exercise, using your favourite definition of uniform
  integrability. 
(See \eg{} \cite[Definition 5.4.1 and Theorem 5.4.1]{Gut}.)
%(In fact, the property in the lemma can be used as another, equivalent,
%definition.)
\end{proof}

\begin{proof}[Proof of \refL{LA1}]
The uniform integrability implies the existence of constants $B$ and $B'$
such that $\E|X_n|\le B$ and $\E |Y_n|\le B'$ for all $n$.  

Let $0<\eps<1$. \refL{LA2} shows that there exists $K_\eps<\infty$ 
and random variables $X_n^\eps$ and $Y_n^\eps$ 
such that both
\eqref{la2a}--\eqref{la2b} and the corresponding inequalities with $Y$ hold.
Then $|X_n^\eps Y_n^\eps|\le K_\eps^2$ a.s.
Since $X_n$ and $Y_n$ are independent, we may also assume that the pairs
$(X_n,X_n^\eps)$ and $(Y_n,Y_n^\eps)$ are independent, and then
\begin{align}
  \E\bigabs{X_nY_n-X_n^\eps Y_n^\eps}
&\le \E\bigabs{X_n(Y_n-Y_n^\eps)}
+
\E\bigabs{(X_n-X_n^\eps)Y_n}
\notag\\&\hskip10em{}
+
\E\bigabs{(X_n-X_n^\eps)(Y_n-Y_n^\eps)}
\notag\\&
\le B\eps+B'\eps+\eps^2
=(B+B'+1)\eps.
\end{align}
\refL{LA2} in the opposite direction shows that the sequence
$(X_nY_n)_n$ is uniformly integrable.
\end{proof}

\section{Bochner and Pettis integrals}\label{ABP}

The expectation $\E\bX$ of an $\cH$-valued random variable $\bX$, where
$\cH$ is a separable Hilbert space, 
can be defined using either the Bochner integral or the Pettis integral;
see \eg{} the summary in \cite[\S2.4]{SJ271} and the references given there.
Both integrals are defined for general Banach spaces, but in this paper we
need them only for separable Hilbert spaces.
In this case, $\E\bX$ exists in Bochner sense if and only if
$\E\norm{\bX}<\infty$, and $\E\bX$ exists in Pettis sense
if and only if $\E|\innprod{\bX,\bx}|<\infty$ for every $\bx\in\cH$, and then
$\E\bX$ is the element of $\HX$ determined by
\begin{align}\label{pettis}
  \innprod{\E\bX,\bx} = \E\innprod{\bX,\bx},
\qquad \bx\in\cH.
\end{align}
If $\E\bX$ exists in Bochner sense, then it exists in Pettis sense, and the
value is the same.
(Hence, the reader may choose to always interpret $\E\bX$ in Pettis sense.
However, the Bochner integral is more convenient when applicable.)
The converse is not true; there are $\bX$ such that $\E\bX$ exists in Pettis
sense but not Bochner sense. (See \eg \refE{EZZB-}.)

It is well-known, and easy to see, that if $\E\bX$ exists in Pettis sense,
then there exists $C<\infty$ (depending on $\bX$) such that
\begin{align}\label{PC}
  \E\abs{\innprod{\bX,\bx}} \le C\norm{\bx},
\qquad \bx\in\HX.
\end{align}

We use in \refSS{SSH2}
some results on Pettis integrals in (separable) Hilbert spaces, stated in the
lemmas below.
We believe that at least some of these are known, but since we have not
found references, we give complete proofs.

\begin{lemma}\label{LZZ}
  Let $\bZ$ be \rv{} in a separable Hilbert space $\cH$, and let $\bZ'$ be
  an independent copy of $\bZ$.
  \begin{romenumerate}
    
  \item \label{LZZB}
If\/ $\bZ$ is Bochner integrable, \ie, if $\E\norm{\bZ}<\infty$, then
\begin{align}\label{EZZ}
  \E |\innprod{\bZ,\bZ'}|<\infty.
\end{align}
\item \label{LZZP}
If \eqref{EZZ} holds, then $\bZ$ is Pettis integrable, \ie, $\E\bZ$ exists
in Pettis sense. Moreover, 
\begin{align}\label{EZZ2}
    \E \innprod{\bZ,\bZ'} = \norm{\E\bZ}^2 
\ge0.
\end{align}
\item \label{LZZ+}
If\/ 
$\E \innprod{\bZ,\bZ'}_+<\infty$,
then \eqref{EZZ} holds.
In other words, 
$\E \innprod{\bZ,\bZ'}$ may be finite (and then $\ge0$ by \eqref{EZZ2}),
$+\infty$ or undefined, but never $-\infty$.
  \end{romenumerate}
\end{lemma}

\begin{remark}
We show in \refEs{EZZB-} and \ref{EZZP-} that the implications in \ref{LZZB}
and \ref{LZZP} are strict, \ie, their converses do not hold.
%Recall that $\HX\tensor\HY$ is the Hilbert space tensor product, see \refS{S2}.
 
Furthermore,
it is easy find examples, even with $\cH=\bbR$, where
$\E\innprod{\bZ,\bZ'}$ is $+\infty$ or undefined (\ie, $\infty-\infty$);
take any real-valued random $\bZ$ with 
$\bZ\ge0$ or with a symmetric distribution, respectively, and 
further $\E|\bZ|=\infty$.
\end{remark}

\begin{proof}[Proof of \refL{LZZ}]
\pfitemref{LZZB}
By the \CSineq,
$ |\innprod{\bZ,\bZ'}|\le\norm{\bZ}\norm{\bZ'}$, and \eqref{EZZ} follows
by the independence of $\bZ$ and $\bZ'$.

\pfitemref{LZZP}
Let $A:=\E |\innprod{\bZ,\bZ'}|$ and let $\bu\in\cH$ with $\norm{u}=1$.
Furthermore, let $W:=\sgn\innprod{\bZ,\bu}$ and $W':=\sgn\innprod{\bZ',\bu}$,
and let for $M>0$,
 $I_M:=\indic{\norm{\bZ}\le M}$ and
$I_M':=\indic{\norm{\bZ'}\le M}$.
Since $I_MW\bZ\eqd I_M'W'\bZ'$ is measurable and bounded, 
$\E[I_MW\bZ]=\E[ I_M'W'\bZ']$ exists, even in Bochner sense, and
we have, for any finite $M$,
\begin{align}\label{lyx}
  A&
\ge \E\bigsqpar{I_MWI_M'W'\innprod{\bZ,\bZ'}}
= \E{ \innprod{I_MW\bZ,I_M'W'\bZ'}}
\notag\\&
%= \E\bigsqpar{ \E\bigpar{I_MWI_M'W'\innprod{\bZ,\bZ'}\mid\bZ}}
= \E\bigsqpar{ \E\bigpar{\innprod{I_MW\bZ,I_M'W'\bZ'}\mid\bZ}}
%\notag\\&
%= \E\bigsqpar{ \innprod{I_MW\bZ,\E[I_M'W'\bZ']}}
= \E{ \innprod{I_MW\bZ,\E[I_M'W'\bZ']}}
\notag\\&
= \innprod{\E[I_MW\bZ],\E[I_M'W'\bZ']}
= \norm{\E[I_MW\bZ]}^2.
\end{align}
Hence, by the \CSineq{}, $\norm{\bu}=1$, and \eqref{lyx},
\begin{align}
  \E\bigsqpar{I_M\abs{\innprod{\bu,\bZ}}}
&=  \E\bigsqpar{I_M W{\innprod{\bu,\bZ}}}
=  \E{\innprod{\bu,I_M W\bZ}}
=  \innprod{\bu,\E[I_M W\bZ]}
\notag\\&
\le
\norm{\E[I_M W\bZ]}
\le A\qq
.
\end{align}
Letting \Mtoo{} yields, by monotone convergence,
\begin{align}
  \E\bigabs{\innprod{\bu,\bZ}}
\le A\qq
\end{align}
for every $\bu$ with $\norm{\bu}=1$, which (since $\cH$ is reflexive) shows
that $\bZ$ is Pettis integrable.

Finally, the Pettis integrability yields first
\begin{align}\label{lyx6}
  \E\bigpar{\innprod{\bZ,\bZ'}\mid\bZ}
=\innprod{\bZ,\E\bZ'}
\end{align}
and then, taking the expectation of \eqref{lyx6},
\begin{align}
  \E{\innprod{\bZ,\bZ'}}
=\E\innprod{\bZ,\E\bZ'}
=\innprod{\E\bZ,\E\bZ'}
=\innprod{\E\bZ,\E\bZ},
\end{align}
which is \eqref{EZZ2}.

\pfitemref{LZZ+}
We have, similarly to \eqref{lyx},
\begin{align}\label{lyx+}
 \E\bigsqpar{I_MI_M'\innprod{\bZ,\bZ'}}
&= \E{ \innprod{I_M\bZ,I_M'\bZ'}}
%\notag\\&
%= \E\bigsqpar{ \E\bigpar{I_MI_M'\innprod{\bZ,\bZ'}\mid\bZ}}
= \E\bigsqpar{ \E\bigpar{\innprod{I_M\bZ,I_M'\bZ'}\mid\bZ}}
\notag\\&
%= \E\bigsqpar{ \innprod{I_M\bZ,\E[I_M'\bZ']}}
= \E{ \innprod{I_M\bZ,\E[I_M'\bZ']}}
%\notag\\&
= \innprod{\E[I_M\bZ],\E[I_M'\bZ']}
\notag\\&
= \norm{\E[I_M\bZ]}^2
\ge0.
\end{align}
Hence,
\begin{align}
   \E\bigsqpar{I_MI_M'\innprod{\bZ,\bZ'}_-}
\le  \E\bigsqpar{I_MI_M'\innprod{\bZ,\bZ'}_+}
\le  \E\bigsqpar{\innprod{\bZ,\bZ'}_+}<\infty,
\end{align}
and letting \Mtoo{} yields 
$   \E \innprod{\bZ,\bZ'}_-<\infty$
by monotone convergence.
Hence, \eqref{EZZ} holds, and the result follows.
\end{proof}

We give counterexamples to converses of the statements
in \refL{LZZ}.
\begin{example}
  \label{EZZB-}
Let $N$ be a positive integer-valued random variable and let $p_n:=\P(N=n)$,
let $(a_n)_1^\infty$ be a sequence of positive numbers,
and let $(\be_i)_i$ be an ON-basis in $\HX$. Define $\bZ:=a_N \be_N$.
Then 
\begin{align}
  \E\norm{\bZ}=\E a_N=\sumn a_np_n.
\end{align}
If $N'$ is an independent copy of $N$, and $\bZ':=a_{N'}e_{N'}$, then
$\innprod{\bZ,\bZ'}=a_N^2\indic{N=N'}$, and thus
\begin{align}
  \E\bigabs{\innprod{\bZ,\bZ'}}
=  \E\innprod{\bZ,\bZ'}
=\sumn a_n^2p_n^2.
\end{align}
Consequently, choosing $p_n=c/n^2$ and $a_n=n$, 
%\eqref{EZZ} holds, 
$ \E\abs{\innprod{\bZ,\bZ'}}<\infty$
but
$\E\norm{\bZ}=\infty$, so $\E \bZ$ does not exist in  Bochner sense.
Hence the converse to Lemma \ref{LZZ}\ref{LZZB} does not hold.

In this example, as is easily seen, $\E \bZ$ exists in Pettis sense if and
only if $\sumn a_n^2p_n^2<\infty$, and then $\E \bZ=\sum_n a_np_n\be_n$.
\end{example}

\begin{example}
  \label{EZZP-}
Let $(\be_i)_i$ be an ON-basis in $\HX$, let $\xi_i\sim N(0,1)$, $i\ge1$, be
independent, and let $N$ be a positive integer-valued  random variable,
independent of $(\xi_i)_i$. Define $\bZ:=\sum_{i=1}^N \xi_i\be_i$.
Then, for any $\bx\in \HX$,
\begin{align}
  \innprod{\bZ,\bx}=\sum_{i=1}^N\innprod{\be_i,\bx}\xi_i.  
\end{align}
Conditioned on $N$, this has a normal distribution with variance
$\sum_1^N\innprod{\be_i,\bx}^2\le \norm{\bx}^2$.
Hence,
\begin{align}
\E\bigpar{\abs{\innprod{\bZ,\bx}}\bigmid N}
=\sqrt{\frac{2}{\pi}}\Bigpar{\sum_{i=1}^N\innprod{\be_i,\bx}^2}\qq
\le\norm{\bx}
\end{align}
and thus $\E\bigabs{\innprod{\bZ,\bx}}\le\norm{\bx}<\infty$. Consequently, 
$\E\bZ$ exists in Pettis sense. (With $\E\bZ=0$, by symmetry.)

On the other hand, if 
$N'\eqd N$ and $\xi'_i\sim N(0,1)$ are independent of each other and of
$N$ and $(\xi_i)_i$, so
$\bZ':=\sum_{i=1}^{N'} \xi'_i\be_i$ is an independent copy of $\bZ$, then
$\innprod{\bZ,\bZ'}=\sum_1^{N\bmin N'}\xi_i\xi'_i$.
The sequence $(\xi_i\xi'_i)_i$ is \iid{} with mean 0 and variance
$\E\sqpar{(\xi_i\xi'_i)^2}=\E\sqpar{\xi_i^2}\E\sqpar{(\xi'_i)^2}=1$, 
and thus by the central limit theorem,
for some $c>0$ and every $n\ge0$,
\begin{align}\label{pta}
  \E\bigpar{\abs{\innprod{\bZ,\bZ'}}\bigmid N\bmin N'=n}
=\E\Bigabs{\sum_1^{n}\xi_i\xi'_i}
\ge c\sqrt n.
\end{align}
Hence,
\begin{align}\label{pt}
  \E\abs{\innprod{\bZ,\bZ'}}
&\ge c\E \sqrt{N\bmin N'}
=c\intoo\P\bigpar{\sqrt{N\bmin N'}>t}\dd t
\notag\\&
=c\intoo\P\bigpar{N>t^2,\, N'>t^2}\dd t
=c\intoo\P\bigpar{N>t^2}^2\dd t.
\end{align}
Choose  $N$ with
$\P(N>n)=n^{-\gam}$ for $n\ge1$, where $0<\gam\le\frac{1}4$.
Then $\P(N>t)\ge t^{-\gam}$ for $t\ge1$, and \eqref{pt} yields
$%\begin{align}
  \E\abs{\innprod{\bZ,\bZ'}}\ge c\int_1^\infty t^{-4\gam}\dd t=\infty
$. %\end{align}
Consequently, $\E \bZ$ exists in Pettis sense, but \eqref{EZZ} does not hold.
Hence, the converse to \refL{LZZ}\ref{LZZP} does not hold.

Note also that \eqref{pta} and \eqref{pt} hold in the opposite direction
with another $c$; hence, in this example, \eqref{EZZ} holds if we take
$\gam>\frac{1}4$.
Moreover, $\norm{\bZ}=\bigpar{\sum_1^N\xi_i^2}\qq$, and it follows from
the law of large numbers that
$\E\bigpar{\norm{\bZ}\mid N=n}\sim \sqrt n$ as \ntoo, and thus, if
$\gam\le\frac12$, we have
$\E\norm{\bZ} \ge c\E N\qq=\infty$.
Consequently, taking $\gam\in(\frac{1}4,\frac12]$ gives another example
showing that the converse to \ref{LZZB} does not hold.
\end{example}

Recall that  a \emph{\HS{} operator} $T:\HX\to\HY$, where $\HX$ and $\HY$
are Hilbert spaces, is a linear operator such that if $(\be_i)_i$ is an
ON-basis in $\HX$, then
\begin{align}\label{HS}
  \normHS{T}^2:=\sum_i\norm{T\be_i}^2<\infty.
\end{align}
(This is independent of the choice of basis $(\be_i)_i$.) See \eg
\cite[\S30.8]{Lax} or \cite[Exercise IX.2.19]{Conway}.
The following lemma is a version of the fact that a \HS{} operator is
absolutely 1-summing \cite[Theorem 2.5.5]{Pietsch}.
\begin{lemma}\label{LHS}
Let $\HX$ and $\HY$ be separable Hilbert spaces,
  let $\bX$ be \rv{} in  $\cH$ such that $\E\bX$ exists in Pettis sense, and
  let $T:\HX\to\HY$ be a \HS{} operator.
Then $\E\norm{T\bX}<\infty$.
\end{lemma}
\begin{proof}
  Since $T$ is a \HS{} operator, $T^*T$ is a positive self-adjoint trace
  class operator in $\HX$, and thus there exists an ON-basis $(\be_i)_i$ in
  $\HX$ consisting of eigenvectors, so $T^*T\be_i=\gl_i\be_i$, where
  $\gl_i\ge0$ and  
  \begin{align}\label{gli}
\sum_i\gl_i=\normHS{T}^2<\infty.     
  \end{align}
(See again \eg \cite[\S30]{Lax} and \cite[Exercise IX.2.19]{Conway}.)
Let $s_i:=\gl\qq$. (These are known as the \emph{singular values} of $T$.)
Then, for any $\bx\in\HX$,
\begin{align}\label{kirke}
  \norm{T\bx}^2
&=\innprod{T^*T\bx,\bx}
=\sum_i \innprod{T^*T\bx,\be_i}\innprod{\bx,\be_i}
=\sum_i \innprod{\bx,T^*T\be_i}\innprod{\bx,\be_i}
\notag\\&
=\sum_i \gl_i\innprod{\bx,\be_i}\innprod{\bx,\be_i}
=\sum_i s^2_i\innprod{\bx,\be_i}^2.
\end{align}

Let $(\eps_i)_i$ be \iid random variables with
$\P(\eps_i=1)=\P(\eps_i=-1)=\frac12$, and let them also be
independent of $\bX$. %symmetric Bernoulli random variables
Let 
\begin{align}
  \label{zeus}
\bZ:=\sum_is_i\eps_i\be_i,
\end{align}
where the sum converges in $\HX$ (surely)
since $\sum_is_i^2<\infty$ by \eqref{gli}.
Let $\bx\in\HX$ and note that
\begin{align}\label{hera}
\innprod{\bx,\bZ}=  \sum_is_i\innprod{\bx,\be_i}\eps_i.
\end{align}
Hence,  using \eqref{kirke},
\begin{align}\label{erik}
\E\abs{\innprod{\bx,\bZ}}^2=
  \E\Bigabs{\sum_is_i\innprod{\bx,\be_i}\eps_i}^2
=\sum_i s^2_i\innprod{\bx,\be_i}^2
=\norm{T\bx}^2.
\end{align}
Moreover,  Khintchine's inequality 
\cite[Lemma 3.8.1]{Gut}
applies to \eqref{hera} and yields
\begin{align}\label{andromeda}
\bigpar{\E\abs{\innprod{\bx,\bZ}}^2}\qq
\le C
\E\abs{\innprod{\bx,\bZ}}.
\end{align}
Combining \eqref{erik} and \eqref{andromeda} we find
\begin{align}\label{io}
\norm{T\bx}
\le C
\E\abs{\innprod{\bx,\bZ}}.
\end{align}
Let $\E_\bX$ and $\E_\eps$
denote integration over $\bX$ and $(\eps_i)$, respectively.
%(I.e., conditional expectation given $(\eps_i)_i$ and $\bX$, respectively.)
Then \eqref{io} yields
$
\norm{T\bX}
\le C
\E_\eps\abs{\innprod{\bX,\bZ}}
$
and thus
\begin{align}\label{ia}
\E  \norm{T\bX}
\le C
\E_\bX\E_\eps\abs{\innprod{\bX,\bZ}}
= C
\E\abs{\innprod{\bX,\bZ}}.
\end{align}
On the other hand, \eqref{PC}
yields, using also the definition \eqref{zeus} and \eqref{gli},
\begin{align}\label{ic}
\E_\bX\abs{\innprod{\bX,\bZ}}\le C\norm{\bZ}
  =C\Bigpar{\sum_is_i^2}\qq
=C\normHS{T}.
\end{align}
Thus,
\begin{align}\label{ib}
%  \E\norm{T\bX}\le
\E\abs{\innprod{\bX,\bZ}}
=\E
\E_\bX\abs{\innprod{\bX,\bZ}}
\le C\normHS{T}<\infty.
\end{align}
The result follows by \eqref{ia} and \eqref{ib}.
\end{proof}

\begin{remark}
  \refE{EZZB-} shows that the result in \refL{LHS} does not hold for $T=I$,
  the identity operator (if $\dim\HX=\infty$).
In fact, the result holds if and only if $T$ is \HS: if $T$ is a bounded
operator that is not \HS, then there exists $\bX$ such that $\E\bX$ exists
but $\E\norm{T\bX}=\infty$; this can be seen by a modification of \refE{EZZB-}.
(We omit the details.)
\end{remark}

\begin{lemma}\label{LHTH}
  Let $\bX$ and $\bY$ be independent
random variables with values in  separable Hilbert
  spaces $\HX$ and $\HY$.
If\/ $\E\bX$ and $\E\bY$ exist in Pettis sense, then $\E[\bX\tensor\bY]$
exists in Pettis sense, in $\HX\tensor\HY$, and
$\E[\bX\tensor\bY]=(\E\bX)\tensor(\E\bY)$.
\end{lemma}

\begin{proof}
  Let $\bz\in \HX\tensor\HY$, and define a linear operator $\Tz:\HX\to\HY$
  by
  \begin{align}\label{Tz}
    \innprod{\Tz\bx,\by}=\innprod{\bx\tensor\by,\bz}.
  \end{align}
Let $(\be_i)_i$ and $(\bex_j)_j$ be ON-bases in $\HX$ and $\HY$. Then 
$(\be_i\tensor\bex_j)_{i,j}$ is an ON-basis in $\HX\tensor\HY$, and thus,
using \eqref{HS} and \eqref{Tz},
\begin{align}
  \normHS{\Tz}^2
&=\sum_i\norm{\Tz\be_i}^2
=\sum_i\sum_j\innprod{\Tz\be_i,\bex_j}^2
=\sum_i\sum_j\innprod{\be_i\tensor\bex_j,\bz}^2
\notag\\&
=\norm{\bz}^2<\infty,
\end{align}
and thus $\Tz$ is a \HS{} operator.
(In fact, as is well-known, it is easy to see that $\bz\mapsto\Tz$ yields an
isometry between $\HX\tensor\HY$ and the space of \HS{} operators
$\HX\to\HY$.)
Hence, \refL{LHS} applies and shows $\E\norm{\Tz\bX}<\infty$.

Furthermore, since $\bY$ is Pettis integrable, \eqref{Tz} and
\eqref{PC} show that for every $\bx\in\HX$,
\begin{align}
  \E|\innprod{\bx\tensor\bY,\bz}|
=
\E\abs{\innprod{\Tz\bx,\bY}}
\le C\norm{\Tz\bx}.
\end{align}
Consequently, with $\E_\bY$ denoting the integral over $\bY$,
\begin{align}
  \E|\innprod{\bX\tensor\bY,\bz}|
=\E\E_\bY|\innprod{\bX\tensor\bY,\bz}|
\le C\E\norm{\Tz\bX}
<\infty.
\end{align}
Since $\bz\in\HX\tensor\HY$ is arbitrary, this shows that $\bX\tensor\bY$ is
Pettis integrable, \ie, that $\E[\bX\tensor\bY]$ exists in Pettis sense.

Finally, by \eqref{pettis}, \eqref{tensor} and independence,
for any $\be_i$ and $\bex_j$ in the bases,
\begin{align}
\innprod{\E[\bX\tensor\bY],\be_i\tensor\bex_j}
&=
\E\innprod{\bX\tensor\bY,\be_i\tensor\bex_j}
=
\E\bigsqpar{\innprod{\bX,\be_i}\innprod{\bY,\bex_j}}
\notag\\&
=\E\sqpar{\innprod{\bX,\be_i}}\E\sqpar{\innprod{\bY,\bex_j}}
={\innprod{\E\bX,\be_i}}{\innprod{\E\bY,\bex_j}}
\notag\\&
=
\innprod{(\E\bX)\tensor(\E\bY),\be_i\tensor\bex_j}.
\end{align}
Since the set of such $\be_i\tensor\bex_j$ is a basis,
$\E\sqpar{\bX\tensor\bY}=(\E\bX)\tensor(\E\bY)$ follows.
\end{proof}

\begin{remark}
In this paper we consider only the Hilbert space tensor product
defined in \refS{S2}.
Nevertheless, we note that \refL{LHTH} \emph{a fortiori} holds also for the
injective tensor product $\HX\check\tensor\HY$,
since there is a natural continuous mapping 
$\HX\tensor\HY\to\HX\check\tensor\HY$ mapping $\bx\tensor\by\mapsto
\bx\tensor\by$. 
On the other hand, the result does not hold for the projective tensor
product
$\HX\widehat\tensor\HY$, which can be seen as follows:
Let $\HX=\HY$ and note that then
$\bx\tensor\by\mapsto\innprod{\bx,\by}$ extends to a continuous linear
functional on
$\HX\widehat\tensor\HY$. Hence,
if $\E[\bX\tensor\bY]$ exists in $\HX\widehat\tensor\HY$, then
$\E\innprod{\bX,\bY}$ exists in $\bbR$, so
$\E|\innprod{\bX,\bY}|<\infty$, but \refE{EZZP-} shows that this does not
always hold for independent Pettis integrable $\bX$ and $\bY$.
\end{remark}

\newcommand\AAP{\emph{Adv. Appl. Probab.} }
\newcommand\JAP{\emph{J. Appl. Probab.} }
\newcommand\JAMS{\emph{J. \AMS} }
\newcommand\MAMS{\emph{Memoirs \AMS} }
\newcommand\PAMS{\emph{Proc. \AMS} }
\newcommand\TAMS{\emph{Trans. \AMS} }
\newcommand\AnnMS{\emph{Ann. Math. Statist.} }
\newcommand\AnnPr{\emph{Ann. Probab.} }
\newcommand\CPC{\emph{Combin. Probab. Comput.} }
\newcommand\JMAA{\emph{J. Math. Anal. Appl.} }
\newcommand\RSA{\emph{Random Structures Algorithms} }
\newcommand\DMTCS{\jour{Discr. Math. Theor. Comput. Sci.} }

\newcommand\AMS{Amer. Math. Soc.}
\newcommand\Springer{Springer-Verlag}
\newcommand\Wiley{Wiley}

\newcommand\vol{\textbf}
\newcommand\jour{\emph}
\newcommand\book{\emph}
\newcommand\inbook{\emph}
\def\no#1#2,{\unskip#2, no. #1,} %(typeset after year) 
\newcommand\toappear{\unskip, to appear}

\newcommand\arxiv[1]{\texttt{arXiv:#1}}
\newcommand\arXiv{\arxiv}
\newcommand\etal{et al\punkt}

\def\nobibitem#1\par{}

\end{document}